\documentclass[final, 12pt,reqno]{amsart}
\usepackage{amssymb,amsmath,graphicx,amsfonts,euscript}
\usepackage{color}
\usepackage[pagewise]{lineno}
\usepackage{showkeys}
\usepackage{fancyhdr}
\usepackage[margin=1in]{geometry}
\newtheorem{theorem}{\bf Theorem}[section]

\newtheorem{corollary}[theorem]{\bf Corollary}
\newtheorem{proposition}[theorem]{\bf Proposition}
\newtheorem{define}[theorem]{\bf Definition}
\newtheorem{remark}{\bf Remark}

\newtheorem{lemma}[theorem]{\bf Lemma}

\newcommand{\beq}{\begin{equation}}
\newcommand{\eeq}{\end{equation}}
\newcommand{\ben}{\begin{eqnarray}}
\newcommand{\een}{\end{eqnarray}}
\newcommand{\beno}{\begin{eqnarray*}}
\newcommand{\eeno}{\end{eqnarray*}}

\numberwithin{equation}{section}
\subjclass[2010]{35A01, 35B45, 35R11, 35Q92.}
\keywords{Generalized Keller-Segel System, Mixing, Fractional dissipation, Suppression of blow up.}
\makeatother
\title[generalized  Keller-Segel system]{Suppression of blow up by mixing in generalized  Keller-Segel system with fractional dissipation and strong singular kernel}

\author[Binbin Shi]{\sc Binbin Shi}

\address{$^1$ School of Mathematical Sciences , Shanghai Jiao Tong University, Shanghai, 200240,  P.R.China.}
\email{binbinshi@sjtu.edu.cn}

\begin{document}
\bibliographystyle{abbrv}
\maketitle
\begin{center}
\sc Abstract
\end{center}
In this paper, we consider the Cauchy problem for a generalized parabolic-elliptic Keller-Segel equation with a fractional dissipation and advection by a weakly mixing (see Definition \ref{def:2.4}). Here the attractive kernel has strong singularity, namely, the derivative appears in the nonlinear term by singular integral. Without advection, the solution of equation blows up in finite time. Under a suitable mixing condition on the advection, we show the global existence of classical solution with large initial data in the case of the derivative of dissipative term is higher than that of nonlinear term. Since the attractive kernel is strong singularity, the weakly mixing has destabilizing effect in addition to the enhanced dissipation effect, which makes the problem more complicated and difficult. In this paper, we establish the $L^\infty$-criterion and obtain the global $L^\infty$ estimate of the solution through some new ideas and techniques. Combined with \cite{Shi.2019}, we discuss all cases of generalized  Keller-Segel system with mixing effect, which was proposed by Kiselev, Xu (see \cite{Kiselev.2016}) and Hopf, Rodrigo (see \cite{Hopf.2018}). Based on more precise estimate of solution and the resolvent estimate of semigroup operator, we introduce a new method to study the enhanced dissipation effect of mixing in generalized parabolic-elliptic Keller-Segel equation with a fractional dissipation. And the RAGE theorem is no longer needed in our analysis.

\vskip .3in

\section{Introduction}

We consider the following generalized parabolic-elliptic Keller-Segel equation on torus $\mathbb{T}^d$ with fractional dissipation and strong singular kernel in the presence of an incompressible flow
\begin{equation}\label{eq:1.1}
\begin{cases}
\partial_t\rho+Au\cdot \nabla \rho+(-\Delta)^{\frac{\alpha}{2}}\rho+\nabla\cdot(\rho B(\rho))=0,\qquad  & x\in \mathbb{T}^d,\ t>0,\\
\rho(0,x)=\rho_0(x),  & x\in \mathbb{T}^d.
\end{cases}
\end{equation}
Here $\rho(t,x)$ is a real-valued function of $t$ and $x$, the domain $\mathbb{T}^d=[-\frac{1}{2},\frac{1}{2})$ is the periodic box with dimension $d\geq2$ and $A$ is a positive constant. The quantity $\rho$ denotes the density of microorganisms, $u$ is a divergence-free vector field which is an ambient flow. The $0<\alpha<2$ is parameter controlling the strength of the dissipative term. The nonlocal operator $(-\Delta)^{\frac{\alpha}{2}}$ is known as the  Laplacian of the order  $\frac{\alpha}{2}$, which is given by
\begin{equation}\label{eq:1.2}
(-\Delta)^{\frac{\alpha}{2}}\phi(x)=\mathcal{F}^{-1}(|\xi|^\alpha \hat{\phi}(\xi))(x),\ \ \ x\in\mathbb{R}^d,
\end{equation}
where $\mathcal{F}^{-1}$ is inverse Fourier transformation. The linear vector operator $B$ is called attractive kernel, which could be represented as
\begin{equation}\label{eq:1.3}
B(\rho)=\nabla((-\Delta)^{-\frac{d+2-\beta}{2}}\rho)=\nabla K\ast \rho,\ \ \  x\in \mathbb{R}^d,
\end{equation}
where
\begin{equation}\label{eq:1.4}
\nabla K\sim -\frac{x}{|x|^\beta}, \ \  \beta\in [2,d+1).
\end{equation}
And notice that if $d<\beta<d+1$, the nonlinear term in (\ref{eq:1.1}) is written as
\begin{equation}\label{eq:1.5}
\begin{aligned}
\nabla\cdot(\rho B(\rho))
&=\nabla \rho \cdot \nabla((-\Delta)^{-\frac{d+2-\beta}{2}}\rho)-\rho(-\Delta)^{\frac{\beta-d}{2}}\rho\\
&=\nabla \rho \cdot \nabla K\ast \rho+\rho\Delta K\ast \rho,
\end{aligned}
\end{equation}
where $(-\Delta)^{\frac{\beta-d}{2}}\rho$ contains the derivative of solution, or $\Delta K$ is not integrable in $L^1$ norm. For  this case, we call $B$ is {\em strong singular kernel}.

\vskip .1in

In this paper, we consider the tours $\mathbb{T}^d$, on which the definitions of $(-\Delta)^{\frac{\alpha}{2}}$ and attractive kernel $B$ are different from those in (\ref{eq:1.2}) and (\ref{eq:1.3}). The fractional Laplacian operator needs a kernel representation. The details can be referred to Section 2. We pose the following assumptions on the attractive kernel $B$ in the case of $\beta\in (d,d+1)$ (see \cite{Hopf.2018})
\begin{itemize}
  \item [(1)]\ $K$ is a periodic convolution kernel, which is smooth away from the origin,
  \vskip .02in
  \item [(2)]\ $\nabla K(x)\sim -\frac{x}{|x|^\beta}$ near $x=0$ for $\beta\in (d,d+1)$. In order to explore this conveniently, we assume $\nabla K(x)= -\frac{x}{|x|^\beta}$ on $ \Omega_0=\left\{ x\in \mathbb{T}^d \big| |x|\leq\frac{1}{4}\right\}$. And define
      \begin{equation}\label{eq:1.6}
      B(\rho)=\nabla K\ast \rho.
      \end{equation}
\end{itemize}

\vskip .1in

In the absence of the advection, the Equation (\ref{eq:1.1}) is the generalized  Keller-Segel equation with fractional dissipation
\begin{equation}\label{eq:1.7}
\partial_t\rho+(-\Delta)^{\frac{\alpha}{2}}\rho+\nabla\cdot(\rho B(\rho))=0,\quad \rho(0,x)=\rho_0(x),\quad x\in \Omega,
\end{equation}
where $\Omega$ is $\mathbb{R}^d$ or $\mathbb{T}^d$. And the Equation (\ref{eq:1.7}) describes many physical processes involving diffusion and interaction of particles (see \cite{ Biler.1999, Brenner.1999}). We know that the solution of Equation (\ref{eq:1.7}) may blow up in finite time for large initial data. Specifically, when $\alpha=2, \beta=d$, the Equation (\ref{eq:1.7}) is called classical attractive type Keller-Segel system. In one space dimension, the equation admits large data global in time smooth solution (see \cite{Hillen.2004,Koichi.2001} ). In high dimensions, there are global in time smooth solution when the initial data is small, while the solutions  may exhibit finite-time blowup for large data (see \cite{Blanchet.2006, Corrias.2004,Kiselev.2016,Nagai.1995,Senba.2002}). When $0<\alpha<2, \beta=d$, the Equation (\ref{eq:1.7}) is a classical Keller-Segel system with fractional dissipation, which has been studied by many people. For $d=1$ and $0<\alpha\leq1$, the solution of Equation (\ref{eq:1.7}) is global if $\|\rho_0\|_{L^{\frac{1}{\alpha}}}\leq C(\alpha)$, and the solution of Equation (\ref{eq:1.7}) is global if $1<\alpha<2$ (see \cite{Nikolaos.2010}). While $d\geq2$, the solution of Equation (\ref{eq:1.7}) would  blow up in finite time with large data (see \cite{Biler.2010, Hopf.2018,Li.2010,Li.2018}). In the case of $0<\alpha<2, \beta\in [2,d+1), d\geq2$, the Equation (\ref{eq:1.7}) is called a generalized Keller-Segel system with fractional dissipation, the solution of Equation (\ref{eq:1.7}) always blows up in finite time when the initial data is large  (see \cite{Biler.2010, Hopf.2018,Li.2018}).

\vskip .1in

For $\Omega=\mathbb{R}^d$, let us consider the scaling properties of the Equation (\ref{eq:1.7}), the attractive kernel $B$ is defined in (\ref{eq:1.3}). This equation is invariant under the scaling
\begin{equation}\label{eq:1.8}
\rho_{\lambda}(t,x)=\lambda^{\alpha+d-\beta}\rho(\lambda^{\alpha}t, \lambda x),\ \ \  \lambda>0.
\end{equation}
Suppose that the critical space is $L^{p_c}$, in the sense that $\|\rho_{\lambda}(t,\cdot)\|_{L^{p_c}}=\|\rho(\lambda^{\alpha}t,\cdot)\|_{L^{p_c}}$. One can get
\begin{equation}\label{eq:1.9}
p_c=\frac{d}{\alpha+d-\beta}.
\end{equation}
As the Equation (\ref{eq:1.7}) is $L^1$ conservation, if $p_c<1$, the Equation (\ref{eq:1.7}) is called $L^1$-subcritical, if $p_c=1$, the Equation (\ref{eq:1.7}) is called $L^1$-critical, and if $p_c>1$, the Equation (\ref{eq:1.7}) is called $L^1$-supercritical. We consider the classical Keller-Segel  system $(\beta=d,\alpha=2)$, then $p_c=\frac{d}{2}$. If $d=1, p_c=\frac{1}{2}<1$ is $L^1$-subcritical, the solution is global. if $d=2, p_c=1$ is $L^1$-critical, when $\|\rho_0\|_{L^1}>8\pi$, the solution blows up in finite time. And if $d\geq3, p_c>1$ is $L^1$-supercritical, when $\|\rho_0\|_{L^1}>0$, the solution blows up in finite time.

\vskip .1in

Enhanced dissipation effect of mixing is an interesting phenomenon for the equations with dissipative and convective terms. It means that when convection $u\cdot \nabla$ is added to the equation, the dissipation effect will be enhanced and the solution of the equation has a faster decaying. For incompressible vector field $u$, there are two kinds of flow which have been widely studied by many scholars. First, $u$ is shear flow, we can refer to \cite{Bedrossian.201701, Wei.2019}, and some specific shear flows are considered to the Naver-Stokes equation. For example, Coutte flow (see \cite{Bedrossian.201602,Masmoudi.2020}), Poiseuille flow (see \cite{Michele.2019}) and Kolmogorov flow (see \cite{Lin.2019, Wei.201901, Wei.2020}). Secondly,  Constantin, Kiselev,  Ryzhik, and Zlato\v{s} (see \cite{Constantin.2008})  defined the relaxation enhancing flow, and proved that $u$ is relaxation enhancing flow if and only if $u\cdot \nabla$ has no nontrivial $\dot{H}^{1}$ eigenfunction. Similarly Constantin, Hopf et.al (see \cite{Hopf.2018}) defined the $\alpha$-relaxation enhancing flow. The details are introduced in Section 2. Some other studies and examples of relaxation enhancing flow can be also referred to \cite{Fayad.2006,Fayad.2002, Kiselev.200801,Kiselev.2016}.

\vskip .1in

Recently, the chemotaxis process in fluid has been studied (see \cite{ Che.2016,Francesco.2010, Duan.2010, Liu.2011, Lorz.2010, Wang.2019, Winkler.201201}). The possible effects result from the interaction of chemotactic and fluid transport process. Many people get interested in the suppression of blow up in the  chemotactic model with mixing effect. Kiselev, Xu (see \cite{Kiselev.2016}), Hopf, Rodrigo (see \cite{Hopf.2018}), and Shi, Wang (see \cite{Shi.2019}) obtained the global solution of Keller-Segel system by mixing effect of relaxation enhancing flow. Bedrossian and He (see \cite{Bedrossian.2017, He.201802}) showed the enhanced dissipation effect and the global existence of classical solution to Keller-Segel system with shear flow.

\vskip .1in

For the Equation (\ref{eq:1.1}), our main concern is whether mixing effect can suppress the blowup phenomenon. In fact, the Equation (\ref{eq:1.1}) become $L^p$-subcritical ($p>\frac{d}{a+d-\beta}$) from $L^1$-critical ($L^1$-supercritical) by large mixing effect. For example, when $\alpha=2, \beta=d, d=2,3$, Kiselev and Xu (see \cite{Kiselev.2016}) established the global $L^2$ estimate of solution in the case of weakly mixing, and obtained the global smooth solution by $L^2$-criterion. For the case of $0<\alpha<2, \beta\in [2,d+1), d\geq2$, Hopf and Rodrigo proved the global $L^2$ estimate of solution by relaxation enhancing flow, and also got the global smooth solution if $\alpha>\max\{\beta-\frac{d}{2},1\}$(see \cite{Hopf.2018}, Theorem 4.5). In particular, for classical Keller-Segel system with fractional dissipation, when $\alpha>\frac{d}{2},d=2,3$, the solution of Equation (\ref{eq:1.1}) was global smooth. For the smaller lower bounds on $\alpha$ and higher dimension $d$, we require the $L^p(p>\frac{d}{a+d-\beta})$ estimate of the solution instead of the $L^2$ estimate. Hopf and Rodrigo only considered the case $\alpha=2, \beta=d, d\geq4$ (see \cite{Hopf.2018}, Theorem 4.6), they got the global $L^p(p>\frac{d}{2})$ estimate of the solution by relaxation enhancing flow, and obtained the global smooth solution by $L^p$-criterion. For the case of $0<\alpha<2,\beta\in [2,d], d\geq2$,  Shi and Wang \cite{Shi.2019} got the global $L^\infty$ estimate of solution to Equation (\ref{eq:1.1}) and obtained the global classical solution.

\vskip .1in

Our goal is to show that the blowup solution of Equation (\ref{eq:1.1}) can be suppressed through mixing effect. This question is motivated by the works of Kiselev et. al (see \cite{Kiselev.2016}), Hopf et. al (see \cite{Hopf.2018}) and Shi et. al (see \cite{Shi.2019}). According to the scaling properties in $\mathbb{R}^d$, if we can get the $L^\infty$  estimate of the solution to Equation (\ref{eq:1.1}), then for
$$
\alpha>\max\{\beta-d, 0\}, \ \ \  d\geq2,
$$
the $p_c=\frac{d}{\alpha+d-\beta}<\infty$, then the Equation (\ref{eq:1.1}) is $L^\infty$-subcritical, the solution of Equation (\ref{eq:1.1}) may be global existence,  where need $\alpha>0$ for obtaining $L^\infty$ estimate of solution via enhanced dissipation effect. In \cite{Shi.2019}, Shi and Wang obtained the global classical solution in the case of $0<\alpha<2, 2\leq\beta\leq d, d\geq2$. Firstly, the $L^\infty$-criterion of solution was established by energy energy. Next, the $L^\infty$ estimate of solution  was obtained  by mixing effect, where nonlinear maximum principle (see Lemma \ref{lem:2.3}) was applied. Compared with \cite{Shi.2019}, for the case of $0<\alpha<2, d<\beta< d+1, d>2$, as $B(\rho)$ is strong singular kernel, the mixing has destabilizing effect in the Equation (\ref{eq:1.1}). Therefore, there are some difficulties that need to be overcome.

\vskip .1in

\begin{itemize}
  \item [(1)] For the $L^\infty$-criterion of solution to Equation (\ref{eq:1.1}) (see Proposition \ref{prop:3.1}), as the technical difficulties of energy inequality, we need to establish the $W^{k,\infty}$ estimate of solution if $L^\infty$ estimate is bounded for all $\beta-d<\alpha<2, d<\beta< d+1, d>2$, see Remark \ref{rem:12}. In this paper, we take $k=3$, and for the $W^{3,\infty}$ estimate of solution, we need to control the nonlinear term by dissipative term. Since $\Delta K$ is not integrable in $L^1$ norm, the $W^{3,\infty}$ estimate is not obvious. We need some new techniques to deal with these problems.
  \item [(2)] In this paper, we establish the global $L^\infty$ estimate of the solution to  Equation (\ref{eq:1.1}) by bootstrap argument. Thus,  we need to obtain local $L^2, L^p(p>\frac{d}{a+d-\beta}), L^\infty$ estimate, which is independent of $A$. Since the nonlinear term is strong singularity, local estimate of solution is difficult by standard energy method. So we need some new techniques to obtain the differential inequalities of $\widetilde{\rho}$  and $\|\rho\|_{L^p}$, where $\widetilde{\rho}$ is defined in (\ref{eq:4.1}).
\end{itemize}

\vskip .1in

In this paper, we consider the generalized Keller-Segel equation with fractional dissipation and weakly mixing (see Definition \ref{def:2.4}) in the case of $\beta-d<\alpha<2, \beta\in (d,d+1), d>2$. In order to obtain the $L^\infty$-criterion of solution to Equation (\ref{eq:1.1}), we establish a nonlinear maximum principle about the derivation of solution on tours $\mathbb{T}^d$, the details can be referred to Appendix (see Lemma \ref{lem:6.1}). By singular integral form of dissipative term and nonlinear term to Equation (\ref{eq:1.1}), we use the part of dissipation to deal with the strong singularity of nonlinear term near origin, and get the estimate of nonlinear term. The details can be referred to (\ref{eq:3.20})-({\ref{eq:3.23}}). And we use the dissipative term to control the nonlinear term by nonlinear maximum principle (see Lemma \ref{lem:6.1}) of derivatives to solution, see (\ref{eq:3.24})-(\ref{eq:3.28}). We show that if the $\|\rho\|_{L^\infty}$  is bounded, then $\|D\rho\|_{L^\infty}, \|D^2\rho\|_{L^\infty}$  and $\|D^3\rho\|_{L^\infty}$ are also bounded. The local $L^p(p>\frac{d}{a+d-\beta}), L^\infty$ estimates of solution are established by similar idea with $L^\infty$-criterion, the details can be referred to Lemma \ref{lem:4.1} and Lemma \ref{lem:4.3}. Some useful techniques are from the argument of surface quasi-geostrophic equation (see \cite{Cordoba.2004}). For the analysis of mixing effect, based on the idea of contradiction and the estimate of semigroup of linear problem, we give a new proof method, which is a new observation and idea of mixing effect in Keller-Segel system. And this method can simplify the analysis. Notice that the RAGE theorem (see \cite{Constantin.2008,Cycon.1987, Kiselev.2016}) is no longer needed in this paper. For original idea and details of estimate of semigroup, we can refer to \cite{Wei.2019}, Section 2 and Appendix. The key idea is to define $\Phi(A)$ in (\ref{eq:5.8}), which describe the  time and frequency  of solution to Equation (\ref{eq:1.1}), see Lemma \ref{lem:5.2} in Section 5 for more details. And we obtain the $L^\infty$ estimate of the solution through nonlinear maximum principle on tours (see Lemma \ref{lem:2.3}), which is introduced in \cite{Shi.2019}. Combined with \cite{Shi.2019}, we discuss all cases of generalized  Keller-Segel system with mixing effect, and believe that the range of $\alpha$ and $d$ is optimal in this paper.

\vskip .1in

Let us now state our main result.

\vskip .1in

\begin{theorem}\label{thm:1.1}
Let $\beta-d<\alpha<2,\beta\in (d,d+1), d>2$, for any initial data $\rho_0\geq0, \rho_0\in W^{3,\infty}(\mathbb{T}^d)$, there exists a smooth incompressible flow $u$ and a positive constant $A_0=A(\alpha,\beta,\rho_0, d)$, when $A\geq A_0$, such that the unique solution $\rho(t,x)$ of Equation (\ref{eq:1.1}) is global in time, and we have
$$
\rho(t,x)\in C(\mathbb{R}^{+}; W^{3,\infty}(\mathbb{T}^d)).
$$
\end{theorem}

\vskip .1in

\begin{remark}\label{rem:1}
The smooth incompressible flow $u$ is weakly mixing (see Definition \ref{def:2.4}). According to the proof in this paper, the result is still remain true  for the general relaxation enhancing flow (see \cite{Constantin.2008,Hopf.2018}).
\end{remark}

\vskip .1in

\begin{remark}\label{rem:2}
Theorem \ref{thm:1.1} show that when the derivative of dissipative term is higher than that of nonlinear term, the classical solution of Equation (\ref{eq:1.1}) is global existence if $A$ is large enough. Similar phenomenon can be referred to  2d quasi-geostrophic equation \cite{Constantin.1999,Resnick.1995}  and fractal Burgers equation \cite{Alexander.2008}.
\end{remark}

\vskip .1in

\begin{remark}\label{rem:3}
Mathematically, we can continue to consider general $\beta$ and $\alpha$. For example, for any $\beta>d, \beta-d<\alpha<2$, the similar with the proof of Theorem \ref{thm:1.1}, the result is  still true. For the case of $\beta>d, \alpha>\beta-d$, when $\alpha\geq2$, the result is not obvious.
\end{remark}

\vskip .1in

In the following, we briefly state our main ideas of the proof. Firstly, we establish the $L^\infty$-criterion of solution to Equation (\ref{eq:1.1}). Namely, the global classical solution of Equation (\ref{eq:1.1}) is existence if the $L^\infty$ norm of solution is uniform bounded. Next, we obtain the $L^\infty$ estimate of the solution to Equation (\ref{eq:1.1}). The local $L^2, L^p$ and $ L^\infty$ estimate of the solution are obtained by some precise estimate, where we needs to choose $p=2^n>\frac{d}{\alpha+d-\beta}, n\in \mathbb{Z}^+$. Based on the idea of contradiction and the estimate of semigroup, the the local $L^2$ estimate of the solution is small by enhanced dissipation effect. Combined with the local $L^2$ and $L^\infty$ estimate of the solution, we deduce that the local $L^p$ estimate of the solution is controlled by its initial data. Using the nonlinear maximum principle (see Lemma \ref{lem:2.3}), the local $L^\infty$ estimate of solution is controlled by the initial data. By bootstrap argument, the local  $L^2$, $L^p$ and $L^\infty$ estimate of the solution can be extended to all time. Thus, we get the uniform $L^\infty$ estimate.

\vskip .1in

This paper is organized as follows. In Section 2, we introduce the nonlocal operator, functional space and mixing effect. Some notations and useful properties are also introduced in this section. In Section 3, we establish the local classical solution and $L^\infty$-criterion of solution to Equation (\ref{eq:1.1}), some new techniques and tools are applied, the details can be referred to the proof of Proposition \ref{prop:3.1} and Lemma \ref{lem:6.1}. In Section 4, we obtain the local $L^2, L^p(p=2^n>\frac{d}{\alpha+d-\beta})$ and $ L^\infty$ estimate of the solution to Equation (\ref{eq:1.1}). In order to control the destructive effect of mixing in nonlinear term, some new observations and estimates are applied. In Section 5, we obtain the global $L^\infty$ estimate of solution to Equation (\ref{eq:1.1}) by bootstrap argument. For the the analysis of mixing effect, we introduce new proof, and nonlinear maximum principle (see Lemma \ref{lem:2.3}) is applied in this section. In Appendix, we give the proofs of the nonlinear maximum principle of the derivation on tours $\mathbb{T}^d$ and some Lemmas.

\vskip .1in

Throughout the paper, $C$ stands for universal constant that may change from line to line.

\vskip .3in

\section{Preliminaries}

In what follows, we provide some the auxiliary results and notations.

\subsection{Nonlocal operator}
The fractional Laplacian is nonlocal operator and it has the following kernel representation on $\mathbb{T}^d$ (see \cite{Burczak.2017, Calderon.1954})
\begin{equation}\label{eq:2.1}
(-\Delta)^{\frac{\alpha}{2}}f(x)=C_{\alpha,d}\sum_{k\in \mathbb{Z}^d}P.V.\int_{\mathbb{T}^d}\frac{f(x)-f(y)}{|x-y+k|^{d+\alpha}}dy,\ \ \ 0<\alpha<2,
\end{equation}
where
\begin{equation}\label{eq:2.2}
C_{\alpha,d}=\frac{2^\alpha\Gamma(\frac{d+\alpha}{2})}{\pi^{\frac{d}{2}}|\Gamma(-\frac{\alpha}{2})|},
\end{equation}
and $P.V.$ is Cauchy principal value. For the convenience of discussion, we omitted the symbol $P.V.$ in next section. Namely
$$
P.V.\int_{\mathbb{T}^d}\triangleq\int_{\mathbb{T}^d}.
$$

\vskip .1in

Recall that we denote by
$$
0\leq\lambda_1\leq\lambda_2\leq\cdots \leq\lambda_n\leq \cdots
$$
is the eigenvalue of the operator $-\Delta$ on $\mathbb{T}^d$, then the eigenvalue of operator $(-\Delta)^{\frac{\alpha}{2}}$ is as follows (see \cite{Cusimano.2018})
\begin{equation}\label{eq:2.3}
0\leq\lambda_1^{\frac{\alpha}{2}}\leq\lambda_2^{\frac{\alpha}{2}}\leq\cdots\leq\lambda_n^{\frac{\alpha}{2}}\leq \cdots,
\end{equation}
and
\begin{equation}\label{eq:2.4}
\lim_{n\rightarrow \infty}\lambda_n^{\frac{\alpha}{2}}=\infty.
\end{equation}

The following results are important lemmas, the details can be referred to \cite{Ascasibar.2013, Cordoba.2004, Ju.2005,Shi.2019}.

\begin{lemma}[Positivity Lemma]\label{lem:2.1}
 Suppose $0\leq \alpha \leq2, \Omega=\mathbb{R}^d, \mathbb{T}^d$ and $f, (-\Delta)^{\frac{\alpha}{2}}f\in L^p$, where $p\geq2$. Then
$$
\frac{2}{p}\int_{\Omega}((-\Delta)^{\frac{\alpha}{4}}|f|^{\frac{p}{2}})^2dx\leq \int_{\Omega}|f|^{p-2}f(-\Delta)^{\frac{\alpha}{2}}f dx.
$$
\end{lemma}
\begin{lemma}\label{lem:2.2}
Suppose $0< \alpha<2, \Omega=\mathbb{R}^d, \mathbb{T}^d$ and $f\in \mathcal{S}(\Omega)$. Then
$$
\int_{\Omega}(-\Delta)^{\frac{\alpha}{2}}f(x)dx=0.
$$
\end{lemma}

\vskip .1in

\begin{lemma}\label{lem:2.3}
Let $f\in \mathcal{S}(\mathbb{T}^d)$ and denote by $\overline{x}$ the point such that
$$
f(\overline{x})=\max_{x\in \mathbb{T}^d}f(x),
$$
and $f(\overline{x})>0$. Then we have the following
\begin{equation}\label{eq:2.5}
(-\Delta)^{\frac{\alpha}{2}}f(\overline{x})\geq C(\alpha,d,p)\frac{f(\overline{x})^{1+\frac{p\alpha}{d}}}{\|f\|_{L^p}^{\frac{p\alpha}{d}}},
\end{equation}
or
\begin{equation}\label{eq:2.6}
f(\overline{x})\leq C(d,p)\|f\|_{L^p}.
\end{equation}
\end{lemma}

\vskip .1in

\begin{remark}\label{rem:4}
The proof of Lemma \ref{lem:2.3} can be referred to \cite{Shi.2019},  and the case of $\mathbb{R}^d$ can be seen in \cite{Rafael.2016}. Nonlinear maximum principle is an important technique for establishing  the $L^\infty$ estimate of solution to Equation (\ref{eq:1.1}), the more details can be seen in Section 5.
\end{remark}

\vskip .1in

\subsection{Functional spaces and notations} We write $L^p(\mathbb{T}^d)$ for the usual  Lebesgue space
$$
L^p(\mathbb{T}^d)=\left\{f\ measurable\ s.t. \int_{\mathbb{T}^d}|f(x)|^pdx<\infty \right\},
$$
the norm for the $L^p$ space is denoted as $\|\cdot\|_{L^p}$, it means
$$
\|f\|_{L^p}=\left( \int_{\mathbb{T}^d}|f|^pdx\right)^{\frac{1}{p}},
$$
with natural adjustment when $p=\infty$. The homogeneous Sobolev norm  $\|\cdot\|_{\dot{W}^{s,p}}$ is defined by
$$
\|f\|_{\dot{W}^{s,p}}=\|D^{s}f\|_{L^p},
$$
where $s\geq 0, p\geq 1$ and $ D$ is any partial  derivative. The non-homogeneous Sobolev norm  $\|\cdot\|_{W^{s,p}}$ is
$$
\|f\|_{W^{s,p}}=\|f\|_{\dot{W}^{s,p}}+\|f\|_{L^p}.
$$
In particular, if $p=2$, the homogeneous Sobolev norm  $\|\cdot\|_{\dot{W}^{s,2}}=\|\cdot\|_{\dot{H}^{s}}$, and
\begin{equation}\label{eq:2.7}
\|f\|^2_{\dot{H}^{s}(\mathbb{T}^d)}=\|(-\Delta)^{\frac{s}{2}}f\|^2_{L^2(\mathbb{T}^d)}=\sum_{k\in \mathbb{Z}^d\backslash \{0\}}|k|^{2s}|\hat{f}(k)|^2.
\end{equation}
The non-homogeneous Sobolev norm  $\|\cdot\|_{W^{s,2}}=\|\cdot\|_{H^{s}}$, and
$$
\|f\|_{H^{s}}=\|f\|_{L^2}+\|f\|_{\dot{H}^{s}}.
$$
Let us denote by $P_N$ the orthogonal projection operator on the subspace formed by Fourier modes $|k|\leq N$:
\begin{equation}\label{eq:2.8}
P_Nf(x)=\sum_{|k|\leq N}\hat{f}(k)e^{ikx}.
\end{equation}
And denote
$$f\sim g,
$$
it means that there exists positive constant $\gamma_1,\gamma_2$, such that
$$
\gamma_1 f(x)\leq g(x)\leq\gamma_2 f(x).
$$

\vskip .05in

In this paper,  some basic energy inequalities are used in the proof, we can refer to \cite{Evans.2010,Friedman.1969} for more details.

\vskip .05in

\subsection{Mixing  effect}
Given an incompressible vector field $u$ which is Lipschitz in spatial variables, if we defined the trajectories map by (see \cite{Constantin.2008, Kiselev.2016,Shi.2019})
$$
\frac{d}{dt} \Phi_t(x)=u(\Phi_t(x)), \quad \Phi_0(x)=x.
$$
Then define a unitary operator $U^t$ acting on $L^2(\mathbb{T}^d)$ as follows
$$
U^t f(x)=f(\Phi_t^{-1}(x)), \ \  f(x)\in L^2(\mathbb{T}^d).
$$
We take
\begin{equation}\label{eq:2.9}
X=\left\{ f\in L^2(\mathbb{T}^d)\big| \int_{\mathbb{T}^d}f(x)dx=0\right\},
\end{equation}
and for a fixed $t>0$, we denote
\begin{equation}\label{eq:2.10}
U=U^t.
\end{equation}

We recall the definition of weakly mixing (see \cite{Constantin.2008,Kiselev.2016}).
\begin{define}\label{def:2.4}
The incompressible flow $u$ is called weakly mixing, if $u=u(x)$ is smooth and the spectrum of the operator $U$ is purely continuous on $X$.
\end{define}

\begin{remark}\label{rem:5}
The incompressible flow $u$ is called $\alpha$-relaxation enhancing flow (see \cite{Constantin.2008,Hopf.2018}) if the corresponding unitary operator $U$ does not have any non-constant eigenfunctions in $\dot{H}^{\frac{\alpha}{2}}(\mathbb{T}^d)$, where $0<\alpha\leq2$. Obviously, the weakly mixing is $\alpha$-relaxation enhancing flow for any $0<\alpha\leq2$.
\end{remark}

\vskip .1in

\begin{remark}\label{rem:6}
The examples of relaxation enhancing flow and weakly mixing can be referred to \cite{Constantin.2008,Fayad.2006,Fayad.2002,Hopf.2018,Kiselev.2016}.
\end{remark}

\vskip .1in

Next, we introduce enhanced dissipation effect of mixing via resolvent estimate (see \cite{Wei.2019}). Let $\mathcal{H}$ be a Hilbert space, we denote by $\|\cdot\|$ the norm and by $\langle,\rangle$ the inner product. Let $H$ be a linear operator in $\mathcal{H}$ with the domain $D(H)$, it is defined as follows (see \cite{Kato.1995, Wei.2019})

\vskip .05in

\begin{define}\label{def:2.5}
A closed operator $H$ in a Hilbert space $\mathcal{H}$ is called m-accretive if the left open half-plane is contained in the resolvent set $\rho(H)$ with
$$
(H+\lambda)^{-1}\in \mathcal{B}(X),\ \ \  \|(H+\lambda)^{-1}\|\leq (Re \lambda)^{-1},\ \ \  Re \lambda>0,
$$
where $X$ is defined  in (\ref{eq:2.9}), and $\mathcal{B}(X)$ is the set of bounded linear operators on $X$.
\end{define}

\vskip .05in

\begin{remark}\label{rem:7}
An m-accretive operator $H$ is  accretive and densely defined (see \cite{Kato.1995}), namely, $D(H)$ is dense in $X$ and $Re \langle Hf,f\rangle\geq0$ for $f\in D(H)$.
\end{remark}

\vskip .05in

We denote $e^{-tH}$ is a  semigroup with $-H$ as generator and define
\begin{equation}\label{eq:2.11}
\Psi(H)=\inf\{\|(H-i\lambda)f\|: f\in D(H), \lambda\in \mathbb{R}, \|f\|=1 \}.
\end{equation}

\vskip .05in

The following result is the Gearchart-Pr\"{u}ss type theorem for accretive operators (see \cite{Bedrossian.201701,Wei.2019}).

\vskip .05in

\begin{lemma}\label{lem:2.6}
Let $H$ be an m-accretive operator in a Hilbert space $\mathcal{H}$. Then for any $t\geq 0$, we have
$$
\|e^{-tH}\|_{L^2\rightarrow L^2}\leq e^{-t\Psi(H)+\frac{\pi}{2}}.
$$
\end{lemma}

\vskip .05in

We consider linear equation with fractional dissipation and mixing effect
\begin{equation}\label{eq:2.12}
\partial_t\eta+(-\Delta)^{\frac{\alpha}{2}}\eta+Au\cdot \nabla \eta=0,\ \ \  \eta(0,x)=\rho_0(x), \ \ x\in \mathbb{T}^d, \ \ t\geq0,
\end{equation}
where $u=u(x)$ is weakly mixing. Let us denote

\vskip .05in

\begin{equation}\label{eq:2.13}
H^\alpha_A=(-\Delta)^{\frac{\alpha}{2}}+Au\cdot \nabla,\ \ \  0<\alpha<2,
\end{equation}
where
$$
D(H^\alpha_A)=H^{2}\cap X.
$$
If $\rho_0\in D(H^\alpha_A)$, then the solution of  Equation (\ref{eq:2.12}) can be expressed by semigroup method, namely
\begin{equation}\label{eq:2.14}
\eta(t,x)=e^{-tH^\alpha_A}\rho_0(x).
\end{equation}

\vskip .05in

\begin{remark}\label{rem:8}
According to the Definition \ref{def:2.5}, the $H^\alpha_A$ in (\ref{eq:2.13}) is a m-accretive operator.
\end{remark}

\vskip .05in

If we define the $\Psi(H^\alpha_A)$ by (\ref{eq:2.11}) and (\ref{eq:2.13}), the similar result with  Lemma \ref{lem:2.6} is as follows
\begin{lemma}\label{lem:2.7}
Let $H^\alpha_A$ be defined in (\ref{eq:2.13}), then for any $t\geq 0$, we have
\begin{equation}\label{eq:2.15}
\|e^{-tH^\alpha_A}\|_{L^2\rightarrow L^2}\leq e^{-t\Psi(H^\alpha_A)+\frac{\pi}{2}}.
\end{equation}
\end{lemma}

If $u$ is weakly mixing, the estimate of $\Psi(H^\alpha_A)$ is as follows

\vskip .05in

\begin{lemma}\label{lem:2.8}
Let $0< \alpha<2$, $u$ is weakly mixing, then
$$
\lim_{A\rightarrow +\infty}\Psi(H^\alpha_A)=+\infty,
$$
where $\Psi(H^\alpha_A)$ be defined by (\ref{eq:2.11}) and (\ref{eq:2.13}).
\end{lemma}

\vskip .05in

\begin{remark}\label{rem:9}
The idea of proof is from the \cite{Wei.2019}, and similar result have been proved in our other studies. For the convenience of reading, we give the proof of Lemma \ref{lem:2.8} in the Appendix.
\end{remark}

\vskip .05in

Next, we give a useful lemma for the estimation of operator $e^{-tH^\alpha_A}$, it is as follows
\begin{lemma}\label{lem:2.9}
Let $0<\alpha<2$, $u$ is weakly mixing. The $e^{-tH^\alpha_A}$ is a semigroup operator with $-H^\alpha_A$, where the $H^\alpha_A$ is defined by (\ref{eq:2.13}). If $A$ is large enough, then for any $t\geq0$ and $f\in D(H^\alpha_A)$, we have
$$
\|P_Ne^{-tH^\alpha_A}f\|_{L^2}\leq C\|P_N f\|_{L^2},
$$
where $P_N$ is defined in (\ref{eq:2.8}) and $C>1$ is a fixed constant.
\end{lemma}

\vskip .05in

\vskip .05in
\begin{remark}\label{rem:10}
The estimation of $\|P_Ne^{-tH^\alpha_A}f\|_{L^2}$ is an important technique in the proof of Lemma \ref{lem:5.1}. The proof of Lemma \ref{lem:2.9} is easy through  Lemma \ref{lem:2.7} and  Lemma \ref{lem:2.8}, the details can be referred to Appendix.
\end{remark}

\vskip .2in

\section{Local existence and continuation criterion}

\vskip .1in

In this section, we consider the generalized  Keller-Segel equation with fractional dissipation and weakly mixing in the case of $\beta\in (d,d+1)$.  The local well-posedness and the $L^\infty$-criterion of solution to Equation (\ref{eq:1.1}) are established.  We show that to get the classical solution of Equation (\ref{eq:1.1}), only need to have certain control of spatial $L^\infty$ norm of the solution.

\begin{proposition}\label{prop:3.1}
Let $\beta-d<\alpha<2, \beta\in (d,d+1), d>2$, for any initial data $\rho_0\geq0, \rho_0\in W^{3,\infty}(\mathbb{T}^d)$, there exists a time $T^\ast=T(\rho_0,\alpha, \beta, d)>0$ such that the non-negative solution of Equation (\ref{eq:1.1})
$$
\rho(t,x)\in C([0,T^\ast], W^{3,\infty}(\mathbb{T}^d)).
$$
Moreover, if for a given $T$, the solution verifies the following bound
$$
\lim_{t\rightarrow T}\sup_{0\leq \tau\leq t}\|\rho(\tau,\cdot)\|_{L^\infty}<\infty,
$$
then it may be extended up to time $T+\delta$ for small enough $\delta>0$. Furthermore, if $\rho_0\in L^1(\mathbb{T}^d)$,
then the $L^1$ norm of the solution of Equation (\ref{eq:1.1}) is preserved for all time, namely $\|\rho\|_{L^1}=\|\rho_0\|_{L^1}$.
\end{proposition}

\vskip .1in

\begin{proof}
The proofs of local existence of solution and the $\rho\geq 0$ are standard method, the $L^1$ norm  conservation of solution is obvious. In this paper, we only need to derive a priori bounds of $\|D\rho\|_{L^\infty},\|D^2\rho\|_{L^\infty}, \|D^3\rho\|_{L^\infty}$ in terms of $L^\infty$ norm of the solution, where $D$ denotes any partial derivative. Suppose that $\rho(t,x)$ is the solution of Equation (\ref{eq:1.1}) with initial data $\rho_0$, and $\|\rho(t,\cdot)\|_{L^\infty}$ is bounded. We finish following three steps.

\vskip .1in

\noindent (1)\ The boundedness of $\|D\rho\|_{L^\infty}$.

\vskip .1in

For this purpose, applying the operator $D$ to the Equation (\ref{eq:1.1}), one has
\begin{equation}\label{eq:3.1}
\begin{aligned}
\partial_tD\rho+& Au\cdot\nabla(D\rho)+ADu\cdot\nabla\rho+(-\Delta)^{\frac{\alpha}{2}}D\rho
+\nabla(D\rho)\cdot\nabla K \ast \rho\\
&+\nabla\rho\cdot\nabla K \ast D\rho+D\rho\Delta K\ast \rho+\rho\Delta K \ast D\rho=0,
\end{aligned}
\end{equation}
and multiply both sides of (\ref{eq:3.1}) by $D\rho$, to obtain
\begin{equation}\label{eq:3.2}
\begin{aligned}
\frac{1}{2}\partial_t|D\rho|^2&+\frac{1}{2}Au\cdot\nabla(|D\rho|^2)+ADu\cdot\nabla \rho D\rho
+D\rho(-\Delta)^{\frac{\alpha}{2}}D\rho\\
&+\frac{1}{2}\nabla(|D\rho|^2)\cdot\nabla K \ast \rho+\nabla\rho\cdot\nabla K \ast D\rho D\rho\\
&+D\rho \nabla K\ast \nabla\rho D\rho+\rho\Delta K \ast D\rho D\rho=0.
\end{aligned}
\end{equation}
The third term of the left-hand side of (\ref{eq:3.2}) can be estimated as
\begin{equation}\label{eq:3.3}
\big|A Du\cdot\nabla \rho D\rho\big|\leq A\|Du\|_{L^\infty}\|D\rho\|^2_{L^\infty}\leq CA\|D\rho\|^2_{L^\infty}.
\end{equation}
For the sixth term and seventh term of the left-hand side of (\ref{eq:3.2}), one has
\begin{equation}\label{eq:3.4}
\begin{aligned}
\big|\nabla\rho\cdot\nabla K \ast D\rho D\rho\big|\leq C\|D\rho\|^2_{L^\infty}\|\nabla K \ast D\rho \|_{L^\infty},
\end{aligned}
\end{equation}
and
\begin{equation}\label{eq:3.5}
\begin{aligned}
\big|D\rho \nabla K\ast \nabla\rho D\rho\big|\leq C\|D\rho\|^2_{L^\infty}\|\nabla K \ast \nabla\rho \|_{L^\infty}.
\end{aligned}
\end{equation}
According to the definition of $\nabla K$ in Section 1, we know that there exists constant $\alpha_1$ and $C>0$, such that
\begin{equation}\label{eq:3.6}
\|D^{1-\alpha_1}(\nabla K) \|_{L^1}\leq C,
\end{equation}
where $\beta-d<\alpha_1<\alpha$. Then we deduce by Young's inequality, Gagliardo-Nirenberg inequality and (\ref{eq:3.6}) that
\begin{equation}\label{eq:3.7}
\begin{aligned}
\|\nabla K \ast D\rho \|_{L^\infty}&\leq \|D^{1-\alpha_1}(\nabla K) \ast D^{\alpha_1}\rho \|_{L^\infty}\\
&\leq C\|D^{1-\alpha_1}(\nabla K) \|_{L^1}\|D^{\alpha_1}\rho \|_{L^\infty}\\
&\leq C\|D^{1-\alpha_1}(\nabla K) \|_{L^1}\|\rho \|^{1-\alpha_1}_{L^\infty}\|D\rho \|^{\alpha_1}_{L^\infty}\\
&\leq C\|D\rho \|^{\alpha_1}_{L^\infty},
\end{aligned}
\end{equation}
and by the similarity with (\ref{eq:3.7}), one has
\begin{equation}\label{eq:3.8}
\begin{aligned}
\|\nabla K \ast \nabla\rho \|_{L^\infty}\leq C\|D\rho \|^{\alpha_1}_{L^\infty}.
\end{aligned}
\end{equation}
Combining (\ref{eq:3.4}), (\ref{eq:3.5}), (\ref{eq:3.7}) and (\ref{eq:3.8}), to obtain
\begin{equation}\label{eq:3.9}
\big|\nabla\rho\cdot\nabla K \ast D\rho D\rho+D\rho \nabla K\ast \nabla\rho D\rho\big|\leq C\|D\rho\|^{2+\alpha_1}_{L^\infty}.
\end{equation}
If define
$$
\widetilde{|D\rho|}=|D\rho|(t,\overline{x}_{1,t})=\max_{x\in \mathbb{T}^d}|D\rho|(t,x).
$$
Using the vanishing of the derivation at the point maximum, the second term and fifth term of the left-hand side of (\ref{eq:3.2}) can be estimated as
\begin{equation}\label{eq:3.10}
\frac{1}{2}Au\cdot\nabla(|D\rho|^2)(t,\overline{x}_{1,t})=0,\ \ \
\frac{1}{2}\nabla(|D\rho|^2)\cdot\nabla K \ast \rho(t,\overline{x}_{1,t})=0.
\end{equation}
For simplicity, we denote $\overline{x}_1=(t,\overline{x}_{1,t})$, then
$$
\rho(t,\overline{x}_{1,t})=\rho(\overline{x}_1),\ \ \  D\rho(t,\overline{x}_{1,t})=D\rho(\overline{x}_1).
$$
According to the definition of fractional Laplacian in (\ref{eq:2.1}), the fourth term of the left-hand side of (\ref{eq:3.2}) can be written as
\begin{equation}\label{eq:3.11}
\begin{aligned}
&D\rho(-\Delta)^{\frac{\alpha}{2}}D\rho(\overline{x}_1)=C_{\alpha,d}\sum_{k\in \mathbb{Z}^d}D\rho(\overline{x}_1)\int_{\mathbb{T}^d}\frac{D\rho(\overline{x}_1)-D\rho(y)}
{|\overline{x}_1-y+k|^{d+\alpha}}dy\\
&=\frac{C_{\alpha,d}}{2}\sum_{k\in \mathbb{Z}^d}\left( \int_{\mathbb{T}^d}\frac{[D\rho(\overline{x}_1)-D\rho(y)]^2}{|\overline{x}_1-y+k|^{d+\alpha}}dy
+\int_{\mathbb{T}^d}\frac{(D\rho)^2(\overline{x}_1)-(D\rho)^2(y)}{|\overline{x}_1-y+k|^{d+\alpha}}dy \right)\\
&=\frac{C_{\alpha,d}}{2}\sum_{k\in \mathbb{Z}^d\backslash \{0\}}\left( \int_{\mathbb{T}^d}\frac{[D\rho(\overline{x}_1)-D\rho(y)]^2}{|\overline{x}_1-y+k|^{d+\alpha}}dy
+\int_{\mathbb{T}^d}\frac{(D\rho)^2(\overline{x}_1)-(D\rho)^2(y)}{|\overline{x}_1-y+k|^{d+\alpha}}dy \right)\\
&\ \ \ +\frac{C_{\alpha,d}}{4}\left( \int_{\mathbb{T}^d}\frac{[D\rho(\overline{x}_1)-D\rho(y)]^2}{|\overline{x}_1-y|^{d+\alpha}}dy
+\int_{\mathbb{T}^d}\frac{(D\rho)^2(\overline{x}_1)-(D\rho)^2(y)}{|\overline{x}_1-y|^{d+\alpha}}dy \right)\\
&\ \ \ +\frac{C_{\alpha,d}}{4}\left( \int_{\mathbb{T}^d\setminus\Omega_1}\frac{[D\rho(\overline{x}_1)-D\rho(y)]^2}{|\overline{x}_1-y|^{d+\alpha}}dy
+\int_{\mathbb{T}^d\setminus\Omega_1}\frac{(D\rho)^2(\overline{x}_1)-(D\rho)^2(y)}{|\overline{x}_1-y|^{d+\alpha}}dy \right)\\
&\ \ \ +\frac{C_{\alpha,d}}{4}\left( \int_{\Omega_1}\frac{[D\rho(\overline{x}_1)-D\rho(y)]^2}{|\overline{x}_1-y|^{d+\alpha}}dy
+\int_{\Omega_1}\frac{(D\rho)^2(\overline{x}_1)-(D\rho)^2(y)}{|\overline{x}_1-y|^{d+\alpha}}dy \right),
\end{aligned}
\end{equation}
where the domain $\Omega_1=\Omega(\overline{x}_1)$ depends on $\overline{x}_1$. In fact,  as $\|\rho\|_{L^\infty}$ is bounded and $\beta-d>0$, if denote
\begin{equation}\label{eq:3.12}
C_{\beta,d}=\beta-d,\ \ \ \ \delta_1=\frac{C_{\beta,d}}{2}\|\rho\|_{L^\infty},\ \ \ \ \delta_2=\frac{C_{\alpha,d}}{4},
\end{equation}
then the domain $\Omega_1$ can be defined by
\begin{equation}\label{eq:3.13}
\Omega_1=\left\{ y\in \mathbb{T}^d \big| |\overline{x}_1-y|\leq\epsilon_0\right\},
\end{equation}
and
\begin{equation}\label{eq:3.14}
\epsilon_0=\min\left\{\frac{1}{4}, \left(\frac{\delta_2}{\delta_1}\right)^{\frac{1}{\alpha+d-\beta}}\right\}.
\end{equation}
Since $|D\rho|$ is the maximum at point $\overline{x}_1$, then we can easily get
\begin{equation}\label{eq:3.15}
\int_{\mathbb{T}^d}\frac{(D\rho)^2(\overline{x}_1)-(D\rho)^2(y)}{|\overline{x}_1-y|^{d+\alpha}}dy \geq0.
\end{equation}
And for the first term and the third term of the right-hand side of (\ref{eq:3.11}), one has
\begin{equation}\label{eq:3.16}
\frac{C_{\alpha,d}}{2}\sum_{k\in \mathbb{Z}^d\backslash \{0\}}\left( \int_{\mathbb{T}^d}\frac{[D\rho(\overline{x}_1)-D\rho(y)]^2}{|\overline{x}_1-y+k|^{d+\alpha}}dy
+\int_{\mathbb{T}^d}\frac{(D\rho)^2(\overline{x}_1)-(D\rho)^2(y)}{|\overline{x}_1-y+k|^{d+\alpha}}dy \right)\geq0,
\end{equation}
and
\begin{equation}\label{eq:3.17}
\frac{C_{\alpha,d}}{4}\left( \int_{\mathbb{T}^d\setminus\Omega_1}\frac{[D\rho(\overline{x}_1)-D\rho(y)]^2}{|\overline{x}_1-y|^{d+\alpha}}dy
+\int_{\mathbb{T}^d\setminus\Omega_1}\frac{(D\rho)^2(\overline{x}_1)-(D\rho)^2(y)}{|\overline{x}_1-y|^{d+\alpha}}dy \right)\geq0.
\end{equation}
For the eighth term of the left-hand side of (\ref{eq:3.2}), as $K$ is a periodic convolution kernel, then one has
\begin{equation}\label{eq:3.18}
\begin{aligned}
\rho D\rho\Delta K \ast D\rho(\overline{x}_1)
&=\rho(\overline{x}_1)D\rho(\overline{x}_1)\int_{\mathbb{T}^d}\Delta K(\overline{x}_1-y) D\rho(y)dy\\
&=-\rho(\overline{x}_1)D\rho(\overline{x}_1)\int_{\mathbb{T}^d}\Delta K(\overline{x}_1-y)(D\rho(\overline{x}_1)-D\rho(y))dy\\
&=-\rho(\overline{x}_1)D\rho(\overline{x}_1)\int_{\Omega_1}\Delta K(\overline{x}_1-y)(D\rho(\overline{x}_1)-D\rho(y))dy\\
&\ \ \ -\rho(\overline{x}_1)D\rho(\overline{x}_1)\int_{\mathbb{T}^d\setminus\Omega_1}\Delta K(\overline{x}_1-y)(D\rho(\overline{x}_1)-D\rho(y))dy.
\end{aligned}
\end{equation}
According to the definitions of $\nabla K$ in Section 1 and (\ref{eq:3.13}), the first term of the right-hand side of (\ref{eq:3.18}) can be written as
\begin{equation}\label{eq:3.19}
\begin{aligned}
&\rho(\overline{x}_1)D\rho(\overline{x}_1)\int_{\Omega_1}\Delta K(\overline{x}_1-y)(D\rho(\overline{x}_1)-D\rho(y))dy\\
&=C_{\beta,d}\rho(\overline{x}_1)D\rho(\overline{x}_1)\int_{\Omega_1}\frac{D\rho(\overline{x}_1)-D\rho(y)}{|\overline{x}_1-y|^\beta}dy\\
&=\frac{C_{\beta,d}}{2}\rho(\overline{x}_1)\left( \int_{\Omega_1}\frac{[D\rho(\overline{x}_1)-D\rho(y)]^2}{|\overline{x}_1-y|^{\beta}}dy
+\int_{\Omega_1}\frac{(D\rho)^2(\overline{x}_1)-(D\rho)^2(y)}{|\overline{x}_1-y|^{\beta}}dy \right),
\end{aligned}
\end{equation}
where $C_{\beta,d}$ is defined in (\ref{eq:3.12}). Since $\beta\in (d,d+1)$, combining the fourth term of the right-hand side of (\ref{eq:3.11}) and (\ref{eq:3.19}), to obtain
\begin{equation}\label{eq:3.20}
\begin{aligned}
&\frac{C_{\beta,d}}{2}\rho(\overline{x}_1)\left( \int_{\Omega_1}\frac{[D\rho(\overline{x}_1)-D\rho(y)]^2}{|\overline{x}_1-y|^{\beta}}dy
+\int_{\Omega_1}\frac{(D\rho)^2(\overline{x}_1)-(D\rho)^2(y)}{|\overline{x}_1-y|^{\beta}}dy \right)\\
\ \ \ \ \ \ \ & -\frac{C_{\alpha,d}}{4}\left( \int_{\Omega_1}\frac{[D\rho(\overline{x}_1)-D\rho(y)]^2}{|\overline{x}_1-y|^{d+\alpha}}dy
+\int_{\Omega_1}\frac{(D\rho)^2(\overline{x}_1)-(D\rho)^2(y)}{|\overline{x}_1-y|^{d+\alpha}}dy \right)\\
\ \ \ \ \ & \leq \delta_1\left( \int_{\Omega_1}\frac{[D\rho(\overline{x}_1)-D\rho(y)]^2}{|\overline{x}_1-y|^{\beta}}dy
+\int_{\Omega_1}\frac{(D\rho)^2(\overline{x}_1)-(D\rho)^2(y)}{|\overline{x}_1-y|^{\beta}}dy \right)\\
\ \ \ \  \ & \ \ -\delta_2\left( \int_{\Omega_1}\frac{[D\rho(\overline{x}_1)-D\rho(y)]^2}{|\overline{x}_1-y|^{d+\alpha}}dy
+\int_{\Omega_1}\frac{(D\rho)^2(\overline{x}_1)-(D\rho)^2(y)}{|\overline{x}_1-y|^{d+\alpha}}dy \right)\\
\ \ \ \  & \leq \delta_1\int_{\Omega_1}\frac{[D\rho(\overline{x}_1)-D\rho(y)]^2}{|\overline{x}_1-y|^{\beta}}dy
 -\delta_2\int_{\Omega_1}\frac{[D\rho(\overline{x}_1)-D\rho(y)]^2}{|\overline{x}_1-y|^{d+\alpha}}dy\\
 \ \ \   & \ \ +\delta_1\int_{\Omega_1}\frac{(D\rho)^2(\overline{x}_1)-(D\rho)^2(y)}{|\overline{x}_1-y|^{\beta}}dy
   -\delta_2\int_{\Omega_1}\frac{(D\rho)^2(\overline{x}_1)-(D\rho)^2(y)}{|\overline{x}_1-y|^{d+\alpha}}dy.
\end{aligned}
\end{equation}
Through the $\beta-d<\alpha<2$, (\ref{eq:3.12}), (\ref{eq:3.13}) and (\ref{eq:3.14}), the first term and second term of the right-hand side of (\ref{eq:3.20}) can be estimated as
\begin{equation}\label{eq:3.21}
 \delta_1\int_{\Omega_1}\frac{[D\rho(\overline{x}_1)-D\rho(y)]^2}{|\overline{x}_1-y|^{\beta}}dy
 -\delta_2\int_{\Omega_1}\frac{[D\rho(\overline{x}_1)-D\rho(y)]^2}{|\overline{x}_1-y|^{d+\alpha}}dy\leq0,
\end{equation}
and
\begin{equation}\label{eq:3.22}
\delta_1\int_{\Omega_1}\frac{(D\rho)^2(\overline{x}_1)-(D\rho)^2(y)}{|\overline{x}_1-y|^{\beta}}dy
   -\delta_2\int_{\Omega_1}\frac{(D\rho)^2(\overline{x}_1)-(D\rho)^2(y)}{|\overline{x}_1-y|^{d+\alpha}}dy\leq0.
\end{equation}
As $\|\rho\|_{L^\infty}$ is bounded and $\|D\rho\|_{L^\infty}=|\widetilde{D\rho}|$, then for the second term of the right-hand side of (\ref{eq:3.18}), one can get
\begin{equation}\label{eq:3.23}
\left|\rho(\overline{x}_1)D\rho(\overline{x}_1)\int_{\mathbb{T}^d\setminus\Omega_1}
\Delta K(\overline{x}_1-y)(D\rho(\overline{x}_1)-D\rho(y))dy\right|\leq C|\widetilde{D\rho}|^2.
\end{equation}
 Combining (\ref{eq:3.11}),(\ref{eq:3.16}), (\ref{eq:3.17}), (\ref{eq:3.18}), (\ref{eq:3.20}), (\ref{eq:3.21}), (\ref{eq:3.22}), (\ref{eq:3.23}) and the $\overline{x}_1$ is also maximum point of $|D\rho|^2$, the fourth term and eighth term of left-hand side of (\ref{eq:3.2}) can be estimated as
\begin{equation}\label{eq:3.24}
\begin{aligned}
&-D\rho(-\Delta)^{\frac{\alpha}{2}}D\rho(\overline{x}_1)-\rho\Delta K \ast D\rho D\rho(\overline{x}_1)\\
&\leq C|\widetilde{D\rho}|^2-\frac{C_{\alpha,d}}{4}\left( \int_{\mathbb{T}^d}\frac{[D\rho(\overline{x}_1)-D\rho(y)]^2}{|\overline{x}_1-y|^{d+\alpha}}dy
+\int_{\mathbb{T}^d}\frac{(D\rho)^2(\overline{x}_1)-(D\rho)^2(y)}{|\overline{x}_1-y|^{d+\alpha}}dy \right)\\
&\ \  \ -\frac{C_{\alpha,d}}{4}\left( \int_{\mathbb{T}^d\setminus\Omega_1}\frac{[D\rho(\overline{x}_1)-D\rho(y)]^2}{|\overline{x}_1-y|^{d+\alpha}}dy
+\int_{\mathbb{T}^d\setminus\Omega_1}\frac{(D\rho)^2(\overline{x}_1)-(D\rho)^2(y)}{|\overline{x}_1-y|^{d+\alpha}}dy \right)\\
&  \ \ \ -\frac{C_{\alpha,d}}{2}\sum_{k\in \mathbb{Z}^d\setminus \{0\}}\left( \int_{\mathbb{T}^d}\frac{[D\rho(\overline{x}_1)-D\rho(y)]^2}{|\overline{x}_1-y+k|^{d+\alpha}}dy
+\int_{\mathbb{T}^d}\frac{(D\rho)^2(\overline{x}_1)-(D\rho)^2(y)}{|\overline{x}_1-y+k|^{d+\alpha}}dy \right)\\
&\leq-\frac{C_{\alpha,d}}{4} \int_{\mathbb{T}^d}\frac{[D\rho(\overline{x}_1)-D\rho(y)]^2}{|\overline{x}_1-y|^{d+\alpha}}dy
+C|\widetilde{D\rho}|^2.
\end{aligned}
\end{equation}
Combing (\ref{eq:3.2}), (\ref{eq:3.9}), (\ref{eq:3.10}) and (\ref{eq:3.24}), we have
\begin{equation}\label{eq:3.25}
\begin{aligned}
\frac{d}{dt}|\widetilde{D\rho}|^2\leq-\frac{C_{\alpha,d}}{2} \int_{\mathbb{T}^d}\frac{[D\rho(\overline{x}_1)-D\rho(y)]^2}{|\overline{x}_1-y|^{d+\alpha}}dy
+CA|\widetilde{D\rho}|^2+C|\widetilde{D\rho}|^{2+\alpha_1}.
\end{aligned}
\end{equation}
Since $\|\rho\|_{L^\infty}$ is bounded, we deduce by Lemma \ref{lem:6.1} that if $|\widetilde{D\rho}|$ satisfies
\begin{equation}\label{eq:3.26}
|\widetilde{D\rho}|\leq c_2\|\rho\|_{L^\infty},
\end{equation}
then $\|D\rho\|_{L^\infty}$ is bounded. If not, $|\widetilde{D\rho}|$ must satisfy
\begin{equation}\label{eq:3.27}
\frac{C_{\alpha,d}}{2} \int_{\mathbb{T}^d}\frac{[D\rho(\overline{x}_1)-D\rho(y)]^2}{|\overline{x}_1-y|^{d+\alpha}}dy
\geq\frac{C_{\alpha,d}}{2}c_1\frac{|D\rho|^{2+\alpha}(\overline{x}_1)}{\|\rho\|^\alpha_{L^\infty}}
=\delta_3\frac{|D\rho|^{2+\alpha}(\overline{x}_1)}{\|\rho\|^\alpha_{L^\infty}},
\end{equation}
where
$$
\delta_3=\frac{C_{\alpha,d}}{2}c_1.
$$
Then we deduce by (\ref{eq:3.25}) and (\ref{eq:3.27}) that
\begin{equation}\label{eq:3.28}
\frac{d}{dt}|\widetilde{D\rho}|^2\leq-\delta_3\frac{|\widetilde{D\rho}|^{2+\alpha}}{\|\rho\|^\alpha_{L^\infty}}
+CA|\widetilde{D\rho}|^2+C|\widetilde{D\rho}|^{2+\alpha_1}.
\end{equation}
As $\alpha>\alpha_1$,  by solving the differential inequality (\ref{eq:3.28}), we deduce that there exist a positive finite constant $C'=C(A,\|D\rho_0\|_{L^\infty},\|\rho\|_{L^\infty},\alpha,\alpha_1)$, such that
\begin{equation}\label{eq:3.29}
|\widetilde{D\rho}|\leq C'.
\end{equation}
Combining (\ref{eq:3.26}) and (\ref{eq:3.29}), there exist a positive constant $C$, such that
$$
\|D\rho\|_{L^\infty}\leq C.
$$

\vskip .2in
\noindent (2)\ The boundedness of $\|D^2\rho\|_{L^\infty}$.
\vskip .1in
Applying the operator $D^2$ to Equation (\ref{eq:1.1}), to obtain
\begin{equation}\label{eq:3.30}
\begin{aligned}
\partial_tD^2\rho&+Au\cdot\nabla(D^2\rho)+AD^2u\cdot\nabla\rho+2ADu\cdot\nabla(D\rho)
+(-\Delta)^{\frac{\alpha}{2}}D^2\rho\\
&+D^2\left(\nabla\rho\cdot\nabla K\ast\rho+\rho\Delta K\ast \rho \right)=0.
\end{aligned}
\end{equation}
The sixth term of the left-hand side of (\ref{eq:3.30}) can be written as
\begin{equation}\label{eq:3.31}
D^2\left(\nabla\rho\cdot\nabla K\ast\rho\right)=\nabla(D^2\rho)\cdot\nabla K\ast\rho
+\nabla\rho\cdot\nabla K\ast D^2\rho+2\nabla(D\rho)\cdot\nabla K\ast D\rho,
\end{equation}
and
\begin{equation}\label{eq:3.32}
D^2\left(\rho\Delta K\ast \rho \right)=D^2\rho \nabla K\ast \nabla\rho+\rho\Delta K\ast D^2\rho+2D\rho \nabla K\ast \nabla(D\rho).
\end{equation}
Combining (\ref{eq:3.30}), (\ref{eq:3.31}) and (\ref{eq:3.32}), one has
\begin{equation}\label{eq:3.33}
\begin{aligned}
\partial_tD^2\rho&+Au\cdot\nabla(D^2\rho)+AD^2u\cdot\nabla\rho+2ADu\cdot\nabla(D\rho)
+(-\Delta)^{\frac{\alpha}{2}}D^2\rho\\
&+\nabla(D^2\rho)\cdot\nabla K\ast\rho
+\nabla\rho\cdot\nabla K\ast D^2\rho+2\nabla(D\rho)\cdot\nabla K\ast D\rho\\
&+D^2\rho \nabla K\ast \nabla\rho +\rho\Delta K\ast D^2\rho +2D\rho \nabla K\ast \nabla(D\rho)=0.
\end{aligned}
\end{equation}
Multiplying  both sides of (\ref{eq:3.33}) by $D^2\rho$, to obtain
\begin{equation}\label{eq:3.34}
\begin{aligned}
\frac{1}{2}\partial_t|D^2\rho|^2&+\frac{1}{2}Au\cdot\nabla(|D^2\rho|^2)+AD^2u\cdot\nabla\rho D^2\rho+2ADu\cdot\nabla(D\rho)D^2\rho\\
&+(-\Delta)^{\frac{\alpha}{2}}D^2\rho D^2\rho+\frac{1}{2}\nabla(|D^2\rho|^2)\cdot\nabla K\ast\rho
+\nabla\rho\cdot\nabla K\ast D^2\rho D^2\rho\\
&+2\nabla(D\rho)\cdot\nabla K\ast D\rho D^2\rho+D^2\rho \nabla K\ast \nabla\rho D^2\rho+\rho\Delta K\ast D^2\rho D^2\rho\\
&+2D\rho \nabla K\ast \nabla(D\rho) D^2\rho=0.
\end{aligned}
\end{equation}
As $\|\rho\|_{L^\infty}, \|\nabla K\|_{L^1}, \|D\rho\|_{L^\infty}, \|D u\|_{L^\infty},
\|D^2 u\|_{L^\infty}$ are bounded, then for the third term and the fourth term  of the left-hand side of (\ref{eq:3.34}), one get
$$
\big|AD^2u\cdot\nabla\rho D^2\rho\big|\leq CA\|D\rho\|_{L^\infty}\|D^2\rho\|_{L^\infty},
$$
and
$$
\big|2ADu\cdot\nabla(D\rho)D^2\rho\big|\leq CA\|D^2\rho\|^2_{L^\infty}.
$$
The seventh term, eighth term, ninth term  and eleventh term of the left-hand side of (\ref{eq:3.34}) can be estimated as
$$
\big|\nabla\rho\cdot\nabla K\ast D^2\rho D^2\rho\big|\leq C\|D^2\rho\|^2_{L^\infty},
$$
$$
\big|2\nabla(D\rho)\cdot\nabla K\ast D\rho D^2\rho\big|\leq C\|D^2\rho\|^2_{L^\infty},
$$
$$
\big|D^2\rho \nabla K\ast \nabla\rho D^2\rho\big|\leq C\|D^2\rho\|^2_{L^\infty},
$$
$$
\big|2D\rho \nabla K\ast \nabla(D\rho) D^2\rho\big| \leq C\|D^2\rho\|^2_{L^\infty}.
$$
If define
\begin{equation}\label{eq:3.35}
\widetilde{|D^2\rho|}=|D^2\rho|(t,\overline{x}_{2,t})=\max_{x\in \mathbb{T}^d}|D^2\rho|(t,x),
\end{equation}
then for the second term and the sixth term  of the left-hand side of (\ref{eq:3.34}), one has
$$
\frac{1}{2}Au\cdot\nabla(|D^2\rho|^2)(t,\overline{x}_{2,t})=0,
$$
and
$$
\frac{1}{2}\nabla(|D^2\rho|^2)\cdot\nabla K\ast\rho(t,\overline{x}_{2,t})=0.
$$
Denote $\overline{x}_2=(t,\overline{x}_{2,t})$, then
$$
\rho(t,\overline{x}_{2,t})=\rho(\overline{x}_2),\ \ \ \  D^2\rho(t,\overline{x}_{2,t})=D^2\rho(\overline{x}_2).
$$
By similar discussions with (\ref{eq:3.11}) and (\ref{eq:3.18}), the fifth term and the tenth term of the left-hand side of (\ref{eq:3.34}) can be written as
\begin{equation}\label{eq:3.36}
\begin{aligned}
&(-\Delta)^{\frac{\alpha}{2}}D^2\rho D^2\rho(\overline{x}_2)\\
&\  =C_{\alpha,d}\sum_{k\in \mathbb{Z}^d}D^2\rho(\overline{x}_2)\int_{\mathbb{T}^d}\frac{D^2\rho(\overline{x}_2)-D^2\rho(y)}
{|\overline{x}_2-y+k|^{d+\alpha}}dy\\
&\  =\frac{C_{\alpha,d}}{2}\sum_{k\in \mathbb{Z}^d\backslash \{0\}}\left( \int_{\mathbb{T}^d}\frac{[D^2\rho(\overline{x}_2)-D^2\rho(y)]^2}{|\overline{x}_2-y+k|^{d+\alpha}}dy
+\int_{\mathbb{T}^d}\frac{(D^2\rho)^2(\overline{x}_2)-(D^2\rho)^2(y)}{|\overline{x}_2-y+k|^{d+\alpha}}dy \right)\\
& \ +\frac{C_{\alpha,d}}{4}\left( \int_{\mathbb{T}^d\backslash\Omega'_1}\frac{[D^2\rho(\overline{x}_2)-D^2\rho(y)]^2}{|\overline{x}_2-y|^{d+\alpha}}dy
+\int_{\mathbb{T}^d\backslash\Omega'_1}\frac{(D^2\rho)^2(\overline{x}_2)-(D^2\rho)^2(y)}{|\overline{x}_2-y|^{d+\alpha}}dy \right)\\
& \  +\frac{C_{\alpha,d}}{4}\left( \int_{\Omega'_1}\frac{[D^2\rho(\overline{x}_2)-D^2\rho(y)]^2}{|\overline{x}_2-y|^{d+\alpha}}dy
+\int_{\Omega'_1}\frac{(D^2\rho)^2(\overline{x}_2)-(D^2\rho)^2(y)}{|\overline{x}_2-y|^{d+\alpha}}dy \right),
\end{aligned}
\end{equation}
and
$$
\begin{aligned}
\rho\Delta K\ast D^2\rho D^2\rho(\overline{x}_2)
&=\rho(\overline{x}_2)D^2\rho(\overline{x}_2)\int_{\mathbb{T}^d}\Delta K(\overline{x}_2-y) D^2\rho(y)dy\\
&=-\rho(\overline{x}_2)D^2\rho(\overline{x}_2)\int_{\mathbb{T}^d}\Delta K(\overline{x}_2-y)(D^2\rho(\overline{x}_2)-D^2\rho(y))dy\\
&=-\rho(\overline{x}_2)D^2\rho(\overline{x}_2)\int_{\Omega'_1}\Delta K(\overline{x}_2-y)(D^2\rho(\overline{x}_2)-D^2\rho(y))dy\\
&\ \ \ -\rho(\overline{x}_2)D^2\rho(\overline{x}_2)\int_{\mathbb{T}^d\setminus\Omega'_1}\Delta K(\overline{x}_2-y)(D^2\rho(\overline{x}_2)-D^2\rho(y))dy,
\end{aligned}
$$
where the domain $\Omega'_1=\Omega(\overline{x}_2)$ depends on $\overline{x}_2$, and the definition of $\Omega'_1$ is similar to $\Omega_1$, the details can be referred to (\ref{eq:3.12}), (\ref{eq:3.13}) and (\ref{eq:3.14}). By the same argument as step 1, one has
\begin{equation}\label{eq:3.37}
\begin{aligned}
\frac{d}{dt}|\widetilde{D^2\rho}|^2\leq-\frac{C_{\alpha,d}}{2} \int_{\mathbb{T}^d}\frac{[D^2\rho(\overline{x}_2)-D^2\rho(y)]^2}{|\overline{x}_2-y|^{d+\alpha}}dy
+CA|\widetilde{D^2\rho}|^2+C|\widetilde{D^2\rho}|.
\end{aligned}
\end{equation}
Combining the boundedness of $\|D\rho\|_{L^\infty}$, Corollary \ref{cor:6.2} and (\ref{eq:3.37}),  we deduce that there exist a constant $C=C(A)$, such that
$$
\|D^2\rho\|_{L^\infty}\leq C.
$$

\vskip .2in

\noindent(3) \ The boundedness of $\|D^3\rho\|_{L^\infty}$.
\vskip .1in
Applying the operator $D^3$ to Equation (\ref{eq:1.1}), one has
\begin{equation}\label{eq:3.38}
\begin{aligned}
\partial_tD^3\rho&+Au\cdot\nabla(D^3\rho)+3ADu\cdot\nabla(D^2\rho)+3AD^2u\cdot\nabla(D\rho)+AD^3u\cdot \nabla \rho\\
&+(-\Delta)^{\frac{\alpha}{2}}D^3\rho+D^3\left(\nabla\rho\cdot\nabla K\ast\rho+\rho\Delta K\ast \rho \right)=0.
\end{aligned}
\end{equation}
The seventh term of the left-hand side of (\ref{eq:3.38}) can be expressed as
\begin{equation}\label{eq:3.39}
\begin{aligned}
D^3\left(\nabla\rho\cdot\nabla K\ast\rho\right)&=\nabla(D^3\rho)\cdot\nabla K\ast\rho
+3\nabla(D^2\rho)\cdot\nabla K\ast D\rho\\
& \ \ \ +3\nabla(D\rho)\cdot\nabla K\ast D^2\rho+\nabla \rho\cdot\nabla K\ast D^3\rho,
\end{aligned}
\end{equation}
and
\begin{equation}\label{eq:3.40}
\begin{aligned}
D^3\left(\rho\Delta K\ast \rho \right)&=D^3\rho \nabla K\ast \nabla\rho+3D^2\rho \nabla K\ast \nabla(D\rho)\\
&\ \ \ +3D\rho \nabla K\ast \nabla(D^2\rho)+\rho\Delta K\ast D^3\rho.
\end{aligned}
\end{equation}
Combining (\ref{eq:3.38}), (\ref{eq:3.39}) and (\ref{eq:3.40}), one has
\begin{equation}\label{eq:3.41}
\begin{aligned}
\partial_tD^3\rho&+Au\cdot\nabla(D^3\rho)+3ADu\cdot\nabla(D^2\rho)+3AD^2u\cdot\nabla(D\rho)+AD^3u\cdot \nabla \rho\\
&+(-\Delta)^{\frac{\alpha}{2}}D^3\rho+\nabla(D^3\rho)\cdot\nabla K\ast\rho
+3\nabla(D^2\rho)\cdot\nabla K\ast D\rho\\
&+3\nabla(D\rho)\cdot\nabla K\ast D^2\rho+\nabla \rho\cdot\nabla K\ast D^3\rho
+D^3\rho \nabla K\ast \nabla\rho\\
&+3D^2\rho \nabla K\ast \nabla(D\rho)
+3D\rho \nabla K\ast \nabla(D^2\rho)+\rho\Delta K\ast D^3\rho=0.
\end{aligned}
\end{equation}
Multiplying  both sides of (\ref{eq:3.41}) by $D^3\rho$, to obtain
\begin{equation}\label{eq:3.42}
\begin{aligned}
\frac{1}{2}\partial_t|D^3\rho|^2&+\frac{1}{2}Au\cdot\nabla(|D^3\rho|^2)+3ADu\cdot\nabla(D^2\rho)D^3\rho +3AD^2u\cdot\nabla(D\rho)D^3\rho\\
&+AD^3u\cdot \nabla\rho D^3\rho+(-\Delta)^{\frac{\alpha}{2}}D^3\rho D^3\rho+\frac{1}{2}\nabla(|D^3\rho|^2)\cdot\nabla K\ast\rho\\
&+3\nabla(D^2\rho)\cdot\nabla K\ast D\rho D^3\rho+3\nabla(D\rho)\cdot\nabla K\ast D^2\rho D^3\rho\\
&+\nabla\rho\cdot\nabla K\ast D^3\rho D^3\rho+D^3\rho \nabla K\ast \nabla\rho D^3\rho+3D^2\rho \nabla K\ast \nabla(D\rho) D^3\rho\\
&+3D\rho \nabla K\ast \nabla(D^2\rho) D^3\rho+\rho\Delta K\ast D^3\rho D^3\rho=0.
\end{aligned}
\end{equation}
As $\|\rho\|_{L^\infty}, \|D K\|_{L^1}, \|D\rho\|_{L^\infty}, \|D^2\rho\|_{L^\infty}, \|D u\|_{L^\infty}, \|D^2 u\|_{L^\infty}, \|D^3 u\|_{L^\infty}$ are bounded. Then some terms of the left-hand side of (\ref{eq:3.42}) can be estimated as
$$
\big|3ADu\cdot\nabla(D^2\rho)D^3\rho\big|\leq CA\|D^3\rho\|^2_{L^\infty},
$$
$$
\big|3AD^2u\cdot\nabla(D\rho)D^3\rho\big|\leq CA\|D^3\rho\|_{L^\infty},
$$
$$
\big|AD^3u\cdot \nabla\rho D^3\rho\big|\leq CA\|D^3\rho\|_{L^\infty},
$$
$$
\big|3\nabla(D^2\rho)\cdot\nabla K\ast D\rho D^3\rho\big|\leq C\|D^3\rho\|^2_{L^\infty},
$$
$$
\big|3\nabla(D\rho)\cdot\nabla K\ast D^2\rho D^3\rho\big|\leq C\|D^3\rho\|_{L^\infty},
$$
$$
\big|\nabla\rho\cdot\nabla K\ast D^3\rho D^3\rho\big|\leq C\|D^3\rho\|^2_{L^\infty},
$$
$$
\big|D^3\rho \nabla K\ast \nabla\rho D^3\rho\big|\leq C\|D^3\rho\|^2_{L^\infty},
$$
$$
\big|3D^2\rho \nabla K\ast \nabla(D\rho) D^3\rho\big|\leq C\|D^3\rho\|_{L^\infty},
$$
$$
\big|3D\rho \nabla K\ast \nabla(D^2\rho) D^3\rho\big|\leq C\|D^3\rho\|^2_{L^\infty}.
$$
If define
$$
\widetilde{|D^3\rho|}=|D^3\rho|(t,\overline{x}_{3,t})=\max_{x\in \mathbb{T}^d}|D^3\rho|(t,x).
$$
Then for the second term and the seventh term  of the left-hand side of (\ref{eq:3.42}), one has
$$
\frac{1}{2}Au\cdot\nabla(|D^3\rho|^2)(t,\overline{x}_{3,t})=0,
$$
and
$$
\frac{1}{2}\nabla(|D^3\rho|^2)\cdot\nabla K\ast\rho(t,\overline{x}_{3,t})=0.
$$
Denote $\overline{x}_3=(t,\overline{x}_{3,t})$, then
$$
\rho(t,\overline{x}_{3,t})=\rho(\overline{x}_3),\ \ \ \  D^3\rho(t,\overline{x}_{3,t})=D^3\rho(\overline{x}_3).
$$
By similar discussions with (\ref{eq:3.11}) and (\ref{eq:3.18}), the sixth term and the fourteenth term of the left-hand side of (\ref{eq:3.42}) can be written as
\begin{equation}\label{eq:3.43}
\begin{aligned}
&(-\Delta)^{\frac{\alpha}{2}}D^3\rho D^3\rho(\overline{x}_3)\\
&\  =C_{\alpha,d}\sum_{k\in \mathbb{Z}^d}D^3\rho(\overline{x}_3)\int_{\mathbb{T}^d}\frac{D^3\rho(\overline{x}_3)-D^3\rho(y)}
{|\overline{x}_3-y+k|^{d+\alpha}}dy\\
&\  =\frac{C_{\alpha,d}}{2}\sum_{k\in \mathbb{Z}^d\backslash\{0\}}\left( \int_{\mathbb{T}^d}\frac{[D^3\rho(\overline{x}_3)-D^3\rho(y)]^2}{|\overline{x}_3-y+k|^{d+\alpha}}dy
+\int_{\mathbb{T}^d}\frac{(D^3\rho)^2(\overline{x}_3)-(D^3\rho)^2(y)}{|\overline{x}_3-y+k|^{d+\alpha}}dy \right)\\
&\  +\frac{C_{\alpha,d}}{4}\left( \int_{\mathbb{T}^d}\frac{[D^3\rho(\overline{x}_3)-D^3\rho(y)]^2}{|\overline{x}_3-y|^{d+\alpha}}dy
+\int_{\mathbb{T}^d}\frac{(D^3\rho)^2(\overline{x}_3)-(D^3\rho)^2(y)}{|\overline{x}_3-y|^{d+\alpha}}dy \right)\\
&\  +\frac{C_{\alpha,d}}{4}\left( \int_{\mathbb{T}^d\backslash\Omega''_1}\frac{[D^3\rho(\overline{x}_3)-D^3\rho(y)]^2}{|\overline{x}_3-y|^{d+\alpha}}dy
+\int_{\mathbb{T}^d\backslash\Omega''_1}\frac{(D^3\rho)^2(\overline{x}_3)-(D^3\rho)^2(y)}{|\overline{x}_3-y|^{d+\alpha}}dy \right)\\
&\  +\frac{C_{\alpha,d}}{4}\left( \int_{\Omega''_1}\frac{[D^3\rho(\overline{x}_3)-D^3\rho(y)]^2}{|\overline{x}_3-y|^{d+\alpha}}dy
+\int_{\Omega''_1}\frac{(D^3\rho)^2(\overline{x}_3)-(D^3\rho)^2(y)}{|\overline{x}_3-y|^{d+\alpha}}dy \right),
\end{aligned}
\end{equation}
and
$$
\begin{aligned}
\rho\Delta K\ast D^3\rho D^3\rho(\overline{x}_3)
&=\rho(\overline{x}_3)D^3\rho(\overline{x}_3)\int_{\mathbb{T}^d}\Delta K(\overline{x}_3-y) D^3\rho(y)dy\\
&=-\rho(\overline{x}_3)D^3\rho(\overline{x}_3)\int_{\mathbb{T}^d}\Delta K(\overline{x}_3-y)(D^3\rho(\overline{x}_3)-D^3\rho(y))dy\\
&=-\rho(\overline{x}_3)D^3\rho(\overline{x}_3)\int_{\Omega''_1}\Delta K(\overline{x}_3-y)(D^3\rho(\overline{x}_3)-D^3\rho(y))dy\\
&\ \ \ -\rho(\overline{x}_3)D^3\rho(\overline{x}_3)\int_{\mathbb{T}^d\setminus\Omega''_1}\Delta K(\overline{x}_3-y)(D^3\rho(\overline{x}_3)-D^3\rho(y))dy,
\end{aligned}
$$
where the domain $\Omega''_1=\Omega(\overline{x}_3)$ depends on $\overline{x}_3$, and the definition of $\Omega''_1$ is similar to $\Omega_1$, the details can be referred to (\ref{eq:3.12}), (\ref{eq:3.13}) and (\ref{eq:3.14}). We deduce by  (\ref{eq:3.42}) and the same argument as step 1 that
\begin{equation}\label{eq:3.44}
\begin{aligned}
\frac{d}{dt}|\widetilde{D^3\rho}|^2\leq-\frac{C_{\alpha,d}}{2} \int_{\mathbb{T}^d}\frac{[D^3\rho(\overline{x}_3)-D^3\rho(y)]^2}{|\overline{x}_3-y|^{d+\alpha}}dy
+CA|\widetilde{D^3\rho}|^2+CA|\widetilde{D^3\rho}|.
\end{aligned}
\end{equation}
Combining the boundedness of $\|D^2\rho\|_{L^\infty}$, Corollary \ref{cor:6.2} and (\ref{eq:3.44}), we imply that there exist a constant $C=C(A)$, such that
$$
\|D^3\rho\|_{L^\infty}\leq C.
$$
This completes the proof of Proposition \ref{prop:3.1}.
\end{proof}

\vskip .2in

\begin{remark}\label{rem:11}
In fact, if $A$ is large enough, one has
$$
\|D \rho\|_{L^\infty}\leq CA^{O(\alpha^{-1})},\ \  \|D^2 \rho\|_{L^\infty}\leq CA^{O(\alpha^{-2})},\ \
\|D^3 \rho\|_{L^\infty}\leq  CA^{O(\alpha^{-3})}.
$$
\end{remark}

\vskip .1in

\begin{remark}\label{rem:12}
In the proof of Proposition \ref{prop:3.1}, if we consider the  Sobolev  space $W^{s,p}(\mathbb{T}^d), 2\leq p<\infty$, then the estimation of nonlinear term is complicated and difficult for all cases of $\beta-d<\alpha<2, \beta\in (d,d+1), d>2$ by energy method, which is form some energy inequalities. In this paper, we consider the $p=\infty,s=3$.
\end{remark}

\vskip .1in

\begin{remark}\label{rem:13}
In fact, for any $T>0$, we know that if $\rho(t,x)\in L^\infty(0,T; L^\infty(\mathbb{T}^d))$, then we deduce that $\rho(t,x)\in L^\infty(0,T; W^{s,p}(\mathbb{T}^d)),s\geq 0,p\geq1$.
\end{remark}

\vskip .2in

\section{Local estimate of solution to equation (\ref{eq:1.1})}

\vskip .1in

In this section, we establish the local $L^2, L^p(p=2^n>\frac{d}{\alpha+d-\beta}, n\in \mathbb{Z}^+)$ and $ L^\infty$ estimate of the solution to Equation (\ref{eq:1.1}), which is independent of $A$. As the attractive kernel is strong singularity, the estimates of the local $L^p$ and $ L^\infty$ are complicated and difficult. Some important Lemmas and new techniques are used in the proof. A useful lemma is as follows

\vskip .1in

\begin{lemma}\label{lem:4.1}
Let $\beta-d<\alpha<2, \beta\in (d,d+1),d>2$, $\rho(t,x)$ is the solution of Equation (\ref{eq:1.1}) with initial data $\rho_0(x)\geq0$. If $\rho\in L^{p}(\mathbb{T}^d), 1\leq p<\infty$, and we denote
\begin{equation}\label{eq:4.1}
\widetilde{\rho}(t)=\rho(t,\overline{x}_t)=\max_{x\in \mathbb{T}^d}\rho(t,x).
\end{equation}
Then
\begin{equation}\label{eq:4.2}
\widetilde{\rho}(t)\leq C(d,p)\|\rho\|_{L^p},
\end{equation}
or
\begin{equation}\label{eq:4.3}
\frac{d}{dt}\widetilde{\rho}\leq -\frac{C(\alpha, d, p)}{2}\frac{\widetilde{\rho}^{1+\frac{p\alpha}{d}}}{\|\rho\|^{\frac{p\alpha}{d}}_{L^p}}+C\widetilde{\rho}^{2+\frac{\beta-d}{d+\alpha-\beta}}
+C\widetilde{\rho}^2.
\end{equation}
\end{lemma}

\vskip .1in

\begin{proof}
For any fixed $t\geq0$, using the vanishing of a derivation at the point of maximum, we see that
$$
\partial_t \rho(t,\overline{x}_t)=\frac{d}{dt}\widetilde{\rho}(t),\quad
(u\cdot \nabla\rho)(t,\overline{x}_t)=u\cdot \nabla\widetilde{\rho}(t)=0.
$$
If we denote
$$
(-\Delta)^{\frac{\alpha}{2}}\rho(t,x)\big|_{x=\overline{x}_t}
=(-\Delta)^{\frac{\alpha}{2}}\widetilde{\rho}(t),
$$
then
$$
\left(\nabla\cdot(\rho B(\rho))\right)(t,\overline{x}_t)
=\widetilde{\rho}\Delta K\ast\rho(t,\overline{x}_t).
$$
Thus, we deduce by the Equation (\ref{eq:1.1}) that  the evolution of $\widetilde{\rho}$ follows
\begin{equation}\label{eq:4.4}
\frac{d}{dt}\widetilde{\rho}+(-\Delta)^{\frac{\alpha}{2}}\widetilde{\rho}
+\widetilde{\rho}\Delta K\ast\rho(t,\overline{x}_t)=0.
\end{equation}
Denote $\overline{x}=(t, \overline{x_t}), x=(t,x)$, then
\begin{equation}\label{eq:4.5}
\rho(t, \overline{x_t})=\rho(\overline{x}), \ \ \ \ \rho(t,x)=\rho(x).
\end{equation}
According to the definitions of fractional Laplancian $(-\Delta)^{\frac{\alpha}{2}}$ and the periodic convolution kernel $K$, the second term and the third  term of the left-hand side of (\ref{eq:4.4}) can be written as
\begin{equation}\label{eq:4.6}
\begin{aligned}
&(-\Delta)^{\frac{\alpha}{2}}\widetilde{\rho}+\widetilde{\rho}\Delta K\ast\rho(\overline{x})\\
&=C_{\alpha,d}\sum_{k\in\mathbb{Z}^d}\int_{\mathbb{T}^d}
\frac{\rho(\overline{x})-\rho(y)}{|\overline{x}-y+k|^{d+\alpha}}dy\\
&\ \ -\widetilde{\rho}\int_{\mathbb{T}^d\setminus\Omega_2}\Delta K(\overline{x}-y)(\rho(\overline{x})-\rho(y))dy
-C_{\beta,d}\widetilde{\rho}\int_{\Omega_2}\frac{\rho(\overline{x})-\rho(y)}{|\overline{x}-y|^{\beta}}dy\\
&=C_{\alpha,d}\sum_{k\in\mathbb{Z}^d\setminus\{0\}}\int_{\mathbb{T}^d}
\frac{\rho(\overline{x})-\rho(y)}{|\overline{x}-y+k|^{d+\alpha}}dy+\frac{C_{\alpha,d}}{2}\int_{\mathbb{T}^d}
\frac{\rho(\overline{x})-\rho(y)}{|\overline{x}-y|^{d+\alpha}}dy\\
&\ \ +\frac{C_{\alpha,d}}{2}\int_{\mathbb{T}^d\setminus\Omega_2}
\frac{\rho(\overline{x})-\rho(y)}{|\overline{x}-y|^{d+\alpha}}dy-\widetilde{\rho}\int_{\mathbb{T}^d\setminus\Omega_2}\Delta K(\overline{x}-y)(\rho(\overline{x})-\rho(y))dy\\
&\ \ -\left(C_{\beta,d}\widetilde{\rho}\int_{\Omega_2}\frac{\rho(\overline{x})-\rho(y)}{|\overline{x}-y|^{\beta}}dy
-\frac{C_{\alpha,d}}{2}\int_{\Omega_2}\frac{\rho(\overline{x})-\rho(y)}{|\overline{x}-y|^{d+\alpha}}dy\right),
\end{aligned}
\end{equation}
where $C_{\alpha,d}$ is defined in (\ref{eq:2.2}), $C_{\beta,d}$ is defined in (\ref{eq:3.12}). And the domain $\Omega_2=\Omega(\overline{x})$ depends on $\overline{x}$, which is defined by
\begin{equation}\label{eq:4.7}
\Omega_2=\left\{y\in \mathbb{T}^d \big| |\overline{x}-y|\leq \epsilon_1 \right\},\ \ \
\epsilon_1=\min \left\{ \left(\frac{C_{\alpha,d}}{2C_{\beta,d}}\right)^{\frac{1}{d+\alpha-\beta}}
\left(\frac{1}{\rho(\overline{x})}\right)^{\frac{1}{d+\alpha-\beta}}, \frac{1}{4}\right\}.
\end{equation}
Since  $\overline{x}$ is the point of the maximum value of $\rho$, then for the third term of the right-hand side of (\ref{eq:4.6}), one has
\begin{equation}\label{eq:4.8}
\frac{C_{\alpha,d}}{2}\int_{\mathbb{T}^d\setminus\Omega_2}
\frac{\rho(\overline{x})-\rho(y)}{|\overline{x}-y|^{d+\alpha}}dy\geq0.
\end{equation}
For the fourth term of the right-hand side of (\ref{eq:4.6}), we deduce by the definition of $\Omega_2$ in (\ref{eq:4.7}) that
\begin{equation}\label{eq:4.9}
\left|\widetilde{\rho}\int_{\mathbb{T}^d\setminus\Omega_2}\Delta K(\overline{x}-y)(\rho(\overline{x})-\rho(y))dy\right|
\leq C\rho(\overline{x})^{2+\frac{\beta-d}{d+\alpha-\beta}}
+C\rho(\overline{x})^2.
\end{equation}
According to $\alpha>\beta-d$ and the definition of $\Omega_2$ in (\ref{eq:4.7}), the fifth term of the right-hand side of (\ref{eq:4.6}), one has
\begin{equation}\label{eq:4.10}
\begin{aligned}
&C_{\beta,d}\widetilde{\rho}\int_{\Omega_2}\frac{\rho(\overline{x})-\rho(y)}{|\overline{x}-y|^{\beta}}dy
-\frac{C_{\alpha,d}}{2}\int_{\Omega_2}\frac{\rho(\overline{x})-\rho(y)}{|\overline{x}-y|^{d+\alpha}}dy\\
&=C_{\beta,d}\widetilde{\rho}\int_{\Omega_2}\frac{\rho(\overline{x})-\rho(y)}{|\overline{x}-y|^{\beta}}dy
-\frac{C_{\alpha,d}}{2}\int_{\Omega_2}\frac{\frac{\rho(\overline{x})-\rho(y)}{|\overline{x}-y|^{d+\alpha-\beta}}}
{|\overline{x}-y|^{\beta}}dy\\
&\leq0.
\end{aligned}
\end{equation}
Combining (\ref{eq:4.6}),  (\ref{eq:4.8}), (\ref{eq:4.9}) and (\ref{eq:4.10}), to obtain
\begin{equation}\label{eq:4.11}
\begin{aligned}
&-(-\Delta)^{\frac{\alpha}{2}}\widetilde{\rho}-\widetilde{\rho}\Delta K\ast\rho(\overline{x})\\
&\leq-C_{\alpha,d}\sum_{k\in\mathbb{Z}^d\setminus\{0\}}\int_{\mathbb{T}^d}
\frac{\rho(\overline{x})-\rho(y)}{|\overline{x}-y+k|^{d+\alpha}}dy-\frac{C_{\alpha,d}}{2}\int_{\mathbb{T}^d}
\frac{\rho(\overline{x})-\rho(y)}{|\overline{x}-y|^{d+\alpha}}dy\\
&\ \ -\frac{C_{\alpha,d}}{2}\int_{\mathbb{T}^d\setminus\Omega_2}
\frac{\rho(\overline{x})-\rho(y)}{|\overline{x}-y|^{d+\alpha}}dy+C\rho(\overline{x})^{2+\frac{\beta-d}{d+\alpha-\beta}}
+C\rho(\overline{x})^2\\
&\leq -\frac{C_{\alpha,d}}{2}\sum_{k\in\mathbb{Z}^d}\int_{\mathbb{T}^d}
\frac{\rho(\overline{x})-\rho(y)}{|\overline{x}-y+k|^{d+\alpha}}dy+C\rho(\overline{x})^{2+\frac{\beta-d}{d+\alpha-\beta}}
+C\rho(\overline{x})^2.
\end{aligned}
\end{equation}
According to (\ref{eq:4.4}), (\ref{eq:4.11}) and the definition of $(-\Delta)^{\frac{\alpha}{2}}$ in (\ref{eq:2.1}), we have
\begin{equation}\label{eq:4.12}
\frac{d}{dt}\widetilde{\rho}\leq -\frac{1}{2}(-\Delta)^{\frac{\alpha}{2}}\widetilde{\rho}+C\widetilde{\rho}^{2+\frac{\beta-d}{d+\alpha-\beta}}
+C\widetilde{\rho}^2.
\end{equation}
Through nonlinear maximum principle (see Lemma \ref{lem:2.3}), if $\widetilde{\rho}$ satisfies
$$
\widetilde{\rho}(t)\leq C(d,p)\|\rho\|_{L^p},
$$
we finish the proof of (\ref{eq:4.2}). If not, $\widetilde{\rho}$ must satisfy
\begin{equation}\label{eq:4.13}
(-\Delta)^{\frac{\alpha}{2}}\widetilde{\rho}(t)\geq C(\alpha, d, p)\frac{\widetilde{\rho}^{1+\frac{p\alpha}{d}}}{\|\rho\|^{\frac{p\alpha}{d}}_{L^p}}.
\end{equation}
Thus, we deduce by (\ref{eq:4.12}) and (\ref{eq:4.13}) that
$$
\frac{d}{dt}\widetilde{\rho}\leq -\frac{C(\alpha, d, p)}{2}\frac{\widetilde{\rho}^{1+\frac{p\alpha}{d}}}{\|\rho\|^{\frac{p\alpha}{d}}_{L^p}}+C\widetilde{\rho}^{2+\frac{\beta-d}{d+\alpha-\beta}}
+C\widetilde{\rho}^2,
$$
we finish the proof of (\ref{eq:4.3}). This completes the proof of Lemma \ref{lem:4.1}.
\end{proof}

\vskip .1in

In this paper, we consider the initial data $\rho_0$ satisfies
\begin{equation}\label{eq:4.14}
\rho_0\in W^{3,\infty}(\mathbb{T}^d),\ \ \  \rho_0\geq0.
\end{equation}
Without loss generality, suppose that there exists positive constant $B_0, D_0, C_{\infty}$, such that
\begin{equation}\label{eq:4.15}
\|\rho_0\|_{L^2}\leq B_0,\ \ \ \|\rho_0-\overline{\rho}\|_{L^p}\leq D_0,\ \ \  \|\rho_0\|_{L^\infty}\leq C_{\infty},
\end{equation}
where
\begin{equation}\label{eq:4.16}
\overline{\rho}=\frac{1}{|\mathbb{T}^d|}\int_{\mathbb{T}^d}\rho_0(x)dx,\ \ \     p=2^n>\frac{d}{\alpha+d-\beta}, \ n\in \mathbb{Z}^+, \ \ \ C_\infty\geq 2C(d,p)(D_0+\overline{\rho}),
\end{equation}
and $C(d,p)$ is defined in Lemma \ref{lem:2.3}. Denote
\begin{equation}\label{eq:4.17}
B_1=\min \left\{(B_0^2-\overline{\rho}^2)^{\frac{1}{2}},\left( \frac{D_0}{(2C_\infty+\overline{\rho})^{1-\frac{2}{p}}}\right)^{\frac{p}{2}}\right\}.
\end{equation}

\vskip .1in

Next, we establish the local $L^2$, $L^p(p=2^n>\frac{d}{\alpha+d-\beta})$ and $L^\infty$ estimates of the solution to Equation (\ref{eq:1.1}), and local time is independent of $A$. The local $L^\infty$ estimate of solution is obvious by Lemma \ref{lem:4.1}, it is as follows

\vskip .1in
\begin{corollary}\label{cor:4.2}
Let $\beta-d<\alpha<2, \beta\in(d,d+1),d>2$, $\rho(t,x)$ is the solution of Equation (\ref{eq:1.1}) with initial data $\rho_0(x)$. Suppose that the initial data $\rho_0(x)$ satisfies (\ref{eq:4.22})-(\ref{eq:4.24}). Then there exist a time $\tau_0>0$, such that for any $0\leq t \leq \tau_0$, we have
$$
\|\rho(t,\cdot)\|_{L^\infty}\leq 2C_{\infty}.
$$
\end{corollary}

\vskip .05in

\begin{proof}
According to the definition of the $\widetilde{\rho}$ in (\ref{eq:4.1}), one has
\begin{equation}\label{eq:4.18}
\widetilde{\rho}\geq 0,\ \ \ \  (-\Delta)^{\frac{\alpha}{2}}\widetilde{\rho}\geq0.
\end{equation}
Combining (\ref{eq:4.12}) and (\ref{eq:4.18}), to obtain
\begin{equation}\label{eq:4.19}
\frac{d}{dt}\widetilde{\rho}\leq C\widetilde{\rho}^{2+\frac{\beta-d}{d+\alpha-\beta}}
+C\widetilde{\rho}^2.
\end{equation}
As $\|\rho_0\|_{L^\infty}\leq C_{\infty}$, if define
\begin{equation}\label{eq:4.20}
\tau_0=\min\left\{ \frac{1}{C\left(2C_\infty+(2C_\infty)^{1+\frac{\beta-d}{d+\alpha-\beta}}\right)},T^\ast \right\},
\end{equation}
where $T^\ast$ is defined in Proposition \ref{prop:3.1}. Then by solving the differential inequality in (\ref{eq:4.19}), for any $0\leq t\leq \tau_0$, one has
\begin{equation}\label{eq:4.21}
\|\rho(t,\cdot)\|_{L^\infty}\leq 2C_{\infty}.
\end{equation}
This completes the proof of Corollary \ref{cor:4.2}.
\end{proof}

\vskip .1in

For the local $L^p(p=2^n>\frac{d}{\alpha+d-\beta}, n\in \mathbb{Z}^+)$ estimate of solution to Equation (\ref{eq:1.1}),
we obtain the  $L^p$ estimate of solution is independent of $A$. Some new techniques are used in the proof.

\vskip .1in

\begin{lemma}\label{lem:4.3}
Let $\beta-d<\alpha<2, \beta\in(d,d+1),d>2$, $\rho(t,x)$ is the solution of Equation (\ref{eq:1.1}) with initial data $\rho_0(x)$. Suppose that $\rho_0(x)$ satisfies (\ref{eq:4.14})-(\ref{eq:4.16}). Then there exist time $\tau'_0>0$, such that for any $0\leq t \leq \tau'_0$, we have
$$
\|\rho(t,\cdot)\|_{L^p}\leq 2(D_0+\overline{\rho}).
$$
\end{lemma}

\vskip .1in

\begin{proof}
Let us multiply both sides of Equation (\ref{eq:1.1}) by $|\rho|^{p-2}\rho$ and integrate over $\mathbb{T}^d$, to obtain
\begin{equation}\label{eq:4.22}
\begin{aligned}
\frac{1}{p}\frac{d}{dt}\|\rho\|_{L^p}^p+A\int_{\mathbb{T}^d}u\cdot\nabla \rho|\rho|^{p-2}\rho dx&+\int_{\mathbb{T}^d}(-\Delta)^{\frac{\alpha}{2}}\rho|\rho|^{p-2}\rho dx\\
&+\int_{\mathbb{T}^d}\nabla\cdot(\rho B(\rho))|\rho|^{p-2}\rho dx=0.
\end{aligned}
\end{equation}
For the second term of the left-hand side of (\ref{eq:4.22}), we deduce by the incompressibility of $u$ that
\begin{equation}\label{eq:4.23}
A\int_{\mathbb{T}^d}u\cdot\nabla \rho|\rho|^{p-2}\rho dx=0.
\end{equation}
As $\rho(t,x)\geq 0$, the fourth term of the left-hand side of (\ref{eq:4.22}) can be written as
\begin{equation}\label{eq:4.24}
\begin{aligned}
\int_{\mathbb{T}^d}\nabla\cdot(\rho B(\rho))|\rho|^{p-2}\rho dx
&=\int_{\mathbb{T}^d}\nabla\cdot(\rho \nabla K\ast\rho)\rho^{p-1} dx\\
&=\frac{p-1}{p}\int_{\mathbb{T}^d}\Delta K\ast \rho \rho^{p}dx.
\end{aligned}
\end{equation}
Combining  (\ref{eq:4.22}), (\ref{eq:4.23}) and (\ref{eq:4.24}), one has
\begin{equation}\label{eq:4.25}
\frac{1}{p}\frac{d}{d t}\|\rho\|^p_{L^p}+\int_{\mathbb{T}^d}(-\Delta)^{\frac{\alpha}{2}}\rho\rho^{p-1}dx
+\frac{p-1}{p}\int_{\mathbb{T}^d}\Delta K\ast \rho \rho^{p}dx=0.
\end{equation}
Next, we consider the  $(-\Delta)^{\frac{\alpha}{2}}\rho\rho^{p-1}(x)$ and $\frac{p-1}{p}\Delta K\ast \rho \rho^{p}(x)$. According to the definition of fractional Laplacian in (\ref{eq:2.1}), the $(-\Delta)^{\frac{\alpha}{2}}\rho\rho^{p-1}(x)$ can be written as
\begin{equation}\label{eq:4.26}
\begin{aligned}
(-\Delta)^{\frac{\alpha}{2}}\rho\rho^{p-1}(x)&=\frac{1}{2}(-\Delta)^{\frac{\alpha}{2}}\rho\rho^{p-1}(x)
+\frac{1}{2}C_{\alpha,d}\rho^{p-1}(x)\sum_{k\in \mathbb{Z}^d}\int_{\mathbb{T}^d}\frac{\rho(x)-\rho(y)}{|x-y+k|^{d+\alpha}}dy\\
&=\frac{1}{2}(-\Delta)^{\frac{\alpha}{2}}\rho\rho^{p-1}(x)
+\frac{1}{2}C_{\alpha,d}\rho^{p-1}(x)\sum_{k\in \mathbb{Z}^d\setminus\{0\}}\int_{\mathbb{T}^d}\frac{\rho(x)-\rho(y)}{|x-y+k|^{d+\alpha}}dy\\
&+\frac{1}{2}C_{\alpha,d}\rho^{p-1}(x)\int_{\mathbb{T}^d\backslash \Omega_3}\frac{\rho(x)-\rho(y)}{|x-y|^{d+\alpha}}dy\\
&+\frac{1}{2}C_{\alpha,d}\rho^{p-1}(x)\int_{\Omega_3}\frac{\rho(x)-\rho(y)}{|x-y|^{d+\alpha}}dy,
\end{aligned}
\end{equation}
and for the $\frac{p-1}{p}\Delta K\ast \rho \rho^{p}(x)$, we deduce by the definition of $K$ in Section 1 that
\begin{equation}\label{eq:4.27}
\begin{aligned}
\frac{p-1}{p}\Delta K\ast \rho \rho^{p}(x)&=\frac{p-1}{p} \rho^{p}(x)\int_{\mathbb{T}^d}\Delta K(x-y)\rho(y)dy\\
&=-\frac{p-1}{p} \rho^{p}(x)\int_{\mathbb{T}^d}\Delta K(x-y)(\rho(x)-\rho(y))dy\\
&=-\frac{p-1}{p} \rho^{p}(x)\int_{\mathbb{T}^d\setminus\Omega_3}\Delta K(x-y)(\rho(x)-\rho(y))dy\\
 &\ \ \ \ \ -\frac{p-1}{p} \rho^{p}(x)\int_{\Omega_3}\Delta K(x-y)(\rho(x)-\rho(y))dy\\
&=-\frac{p-1}{p} \rho^{p}(x)\int_{\mathbb{T}^d\setminus\Omega_3}\Delta K(x-y)(\rho(x)-\rho(y))dy\\
&  \ \ \ -\frac{p-1}{p}C_{\beta,d} \rho^{p}(x)\int_{\Omega_3} \frac{\rho(x)-\rho(y)}{|x-y|^{\beta}}dy,
\end{aligned}
\end{equation}
where $C_{\alpha,d}$ is defined in (\ref{eq:2.2}), $C_{\beta,d}$ is defined in (\ref{eq:3.12}). And the domain $\Omega_3=\Omega(x)$ depends on $x$, which is defined by
\begin{equation}\label{eq:4.28}
\Omega_3=\Omega(x)=\{y||x-y|\leq \epsilon_3\},\ \ \ 0<\epsilon_3\leq\frac{1}{4}.
\end{equation}
In this paper, we denote
\begin{equation}\label{eq:4.29}
\epsilon_3=\min \left\{ \frac{1}{4},\big( \frac{pC_{\alpha,d}}{4(p-1)C_{\beta,d}\rho(\overline{x})}\big)^{\frac{1}{\alpha+d-\beta}} \right\},
\end{equation}
where $\rho(\overline{x})$ is defined in (\ref{eq:4.5}). Then for any $x\in \mathbb{T}^d$, when $y\in \Omega_3$, one has
\begin{equation}\label{eq:4.30}
0<|x-y|^{d+\alpha-\beta}\frac{p-1}{p}C_{\beta,d}\rho(x)\leq |x-y|^{d+\alpha-\beta}\frac{p-1}{p}C_{\beta,d}\rho(\overline{x})<\frac{C_{\alpha,d}}{4}.
\end{equation}
Combining the fourth term of the right-hand side of (\ref{eq:4.26}), the second term of the right-hand side of (\ref{eq:4.27}) and the definition of $\Omega_3$ in (\ref{eq:4.28}) and (\ref{eq:4.30}), one get
\begin{equation}\label{eq:4.31}
\begin{aligned}
&\frac{1}{2}C_{\alpha,d}\rho^{p-1}(x)\int_{\Omega_3}\frac{\rho(x)-\rho(y)}{|x-y|^{d+\alpha}}dy
 -\frac{p-1}{p}C_{\beta,d} \rho^{p}(x)\int_{\Omega_3} \frac{\rho(x)-\rho(y)}{|x-y|^{\beta}}dy\\
&=\frac{1}{2}C_{\alpha,d}\rho^{p-1}(x)\int_{\Omega_3}\frac{\rho(x)-\rho(y)}{|x-y|^{d+\alpha}}dy\\
&\ \  -\frac{p-1}{p}C_{\beta,d} \rho^{p}(x)\int_{\Omega_3} \frac{(\rho(x)-\rho(y))|x-y|^{\alpha+d-\beta}}{|x-y|^{d+\alpha}}dy\\
&\sim C^\ast\rho^{p-1}(x)\int_{\Omega_3}\frac{\rho(x)-\rho(y)}{|x-y|^{d+\alpha}}dy,
\end{aligned}
\end{equation}
where
$$
\frac{C_{\alpha,d}}{4}<C^\ast<\frac{C_{\alpha,d}}{2}.
$$
Consider the the third term of the right-hand side of (\ref{eq:4.26}) and (\ref{eq:4.31}), we obtain
\begin{equation}\label{eq:4.32}
\begin{aligned}
&\frac{1}{2}C_{\alpha,d}\rho^{p-1}(x)\int_{\mathbb{T}^d\setminus\Omega_3}\frac{\rho(x)-\rho(y)}{|x-y|^{d+\alpha}}dy
+C^\ast\rho^{p-1}(x)\int_{\Omega_3}\frac{\rho(x)-\rho(y)}{|x-y|^{d+\alpha}}dy\\
&=C^\ast\rho^{p-1}(x)\int_{\mathbb{T}^d}\frac{\rho(x)-\rho(y)}{|x-y|^{d+\alpha}}dy
+C_\ast\rho^{p-1}(x)\int_{\mathbb{T}^d\setminus\Omega_3}\frac{\rho(x)-\rho(y)}{|x-y|^{d+\alpha}}dy,
\end{aligned}
\end{equation}
where
$$
0<C_\ast<\frac{C_{\alpha,d}}{4}, \ \ \  C^\ast+C_\ast=\frac{C_{\alpha,d}}{2}.
$$
Combining (\ref{eq:4.26}), (\ref{eq:4.27}), (\ref{eq:4.31}) and (\ref{eq:4.32}), one has
\begin{equation}\label{eq:4.33}
\begin{aligned}
&(-\Delta)^{\frac{\alpha}{2}}\rho\rho^{p-1}+\frac{p-1}{p}\Delta K\ast \rho \rho^{p}\\
&\sim \frac{1}{2}(-\Delta)^{\frac{\alpha}{2}}\rho\rho^{p-1}+\frac{1}{2}C_{\alpha,d}\rho^{p-1}(x)\sum_{k\in \mathbb{Z}^d\setminus\{0\}}\int_{\mathbb{T}^d}\frac{\rho(x)-\rho(y)}{|x-y+k|^{d+\alpha}}dy\\
&\ \ \ + C^\ast\rho^{p-1}(x)\int_{\mathbb{T}^d}\frac{\rho(x)-\rho(y)}{|x-y|^{d+\alpha}}dy
+C_\ast\rho^{p-1}(x)\int_{\mathbb{T}^d\setminus\Omega_3}\frac{\rho(x)-\rho(y)}{|x-y|^{d+\alpha}}dy\\
&\ \ \ -\frac{p-1}{p} \rho^{p}(x)\int_{\mathbb{T}^d\setminus\Omega_3}\Delta K(x-y)(\rho(x)-\rho(y))dy.
\end{aligned}
\end{equation}
Next, we claim that the following inequality holds
\begin{equation}\label{eq:4.34}
\int_{\mathbb{T}^d}\rho^{p-1}(x)\sum_{k\in \mathbb{Z}^d\setminus\{0\}}\int_{\mathbb{T}^d}\frac{\rho(x)-\rho(y)}{|x-y+k|^{d+\alpha}}dydx\geq0.
\end{equation}
For any $k\in \mathbb{Z}^d$, one has
\begin{equation}\label{eq:4.35}
\begin{aligned}
\rho(x)\int_{\mathbb{T}^d}\frac{\rho(x)-\rho(y)}{|x-y+k|^{d+\alpha}}dy
&=\int_{\mathbb{T}^d}\frac{\rho^2(x)-\rho(y)\rho(x)}{|x-y+k|^{d+\alpha}}dy\\
&=\frac{1}{2}\int_{\mathbb{T}^d}\frac{\rho^2(x)-\rho^2(y)}{|x-y+k|^{d+\alpha}}dy
+\frac{1}{2}\int_{\mathbb{T}^d}\frac{(\rho(x)-\rho(y))^2}{|x-y+k|^{d+\alpha}}dy\\
&\geq \frac{1}{2}\int_{\mathbb{T}^d}\frac{\rho^2(x)-\rho^2(y)}{|x-y+k|^{d+\alpha}}dy.
\end{aligned}
\end{equation}
Since $p=2^n$, through discussion similar with (\ref{eq:4.35}), to obtain
\begin{equation}\label{eq:4.36}
\begin{aligned}
\rho^{p-1}(x)\int_{\mathbb{T}^d}\frac{\rho(x)-\rho(y)}{|x-y+k|^{d+\alpha}}dy
&\geq \frac{1}{2}\rho^{p-2}(x)\int_{\mathbb{T}^d}\frac{\rho^2(x)-\rho^2(y)}{|x-y+k|^{d+\alpha}}dy\\
&\geq \frac{1}{2^2}\rho^{p-4}(x)\int_{\mathbb{T}^d}\frac{\rho^4(x)-\rho^4(y)}{|x-y+k|^{d+\alpha}}dy\\
&\geq \frac{1}{2^{n-1}}\rho^{2^{n-1}}(x)\int_{\mathbb{T}^d}\frac{\rho^{2^{n-1}}(x)-\rho^{2^{n-1}}(y)}{|x-y+k|^{d+\alpha}}dy\\
&=\frac{2}{p}\rho^{\frac{p}{2}}(x)\int_{\mathbb{T}^d}
\frac{\rho^{\frac{p}{2}}(x)-\rho^{\frac{p}{2}}(y)}{|x-y+k|^{d+\alpha}}.
\end{aligned}
\end{equation}
Integrate on both sides of (\ref{eq:4.36}) over $\mathbb{T}^d$, one get
\begin{equation}\label{eq:4.37}
\int_{\mathbb{T}^d}\rho^{p-1}(x)\int_{\mathbb{T}^d}\frac{\rho(x)-\rho(y)}{|x-y+k|^{d+\alpha}}dydx
\geq\frac{2}{p}\int_{\mathbb{T}^d}\rho^{\frac{p}{2}}(x)\int_{\mathbb{T}^d}
\frac{\rho^{\frac{p}{2}}(x)-\rho^{\frac{p}{2}}(y)}{|x-y+k|^{d+\alpha}}dydx.
\end{equation}
The formula on the right-hand side of (\ref{eq:4.37}) can be written as
\begin{equation}\label{eq:4.38}
\begin{aligned}
\int_{\mathbb{T}^d}\rho^{\frac{p}{2}}(x)\int_{\mathbb{T}^d}
\frac{\rho^{\frac{p}{2}}(x)-\rho^{\frac{p}{2}}(y)}{|x-y+k|^{d+\alpha}}dydx
&=\frac{1}{2}\int_{\mathbb{T}^d}
\rho^{\frac{p}{2}}(x)\int_{\mathbb{T}^d}\frac{\rho^{\frac{p}{2}}(x)-\rho^{\frac{p}{2}}(y)}{|x-y+k|^{d+\alpha}}dydx\\
&\ \ \ -\frac{1}{2}\int_{\mathbb{T}^d}
\rho^{\frac{p}{2}}(x)\int_{\mathbb{T}^d}\frac{\rho^{\frac{p}{2}}(y)-\rho^{\frac{p}{2}}(x)}{|x-y+k|^{d+\alpha}}dydx\\
&=\frac{1}{2}\int_{\mathbb{T}^d}
\rho^{\frac{p}{2}}(x)\int_{\mathbb{T}^d}\frac{\rho^{\frac{p}{2}}(x)-\rho^{\frac{p}{2}}(y)}{|x-y+k|^{d+\alpha}}dydx\\
&\ \ \ -\frac{1}{2}\int_{\mathbb{T}^d}
\rho^{\frac{p}{2}}(x)\int_{\mathbb{T}^d}\frac{\rho^{\frac{p}{2}}(y)-\rho^{\frac{p}{2}}(x)}{|y-x+(-k)|^{d+\alpha}}dydx.
\end{aligned}
\end{equation}
And as $k\in \mathbb{Z}^d\setminus\{0\}$ is symmetrical, one get
\begin{equation}\label{eq:4.39}
\begin{aligned}
&\sum_{k\in \mathbb{Z}^d\setminus\{0\}}\int_{\mathbb{T}^d}
\rho^{\frac{p}{2}}(x)\int_{\mathbb{T}^d}\frac{\rho^{\frac{p}{2}}(y)-\rho^{\frac{p}{2}}(x)}{|y-x+(-k)|^{d+\alpha}}dydx\\
&=\sum_{k\in \mathbb{Z}^d\setminus\{0\}}\int_{\mathbb{T}^d}
\rho^{\frac{p}{2}}(x)\int_{\mathbb{T}^d}\frac{\rho^{\frac{p}{2}}(y)-\rho^{\frac{p}{2}}(x)}{|y-x+k|^{d+\alpha}}dydx\\
&=\sum_{k\in \mathbb{Z}^d\setminus\{0\}}\int_{\mathbb{T}^d}
\rho^{\frac{p}{2}}(y)\int_{\mathbb{T}^d}\frac{\rho^{\frac{p}{2}}(x)-\rho^{\frac{p}{2}}(y)}{|x-y+k|^{d+\alpha}}dxdy.
\end{aligned}
\end{equation}
Then, we deduce by (\ref{eq:4.37}), (\ref{eq:4.38}) and (\ref{eq:4.39}) that
\begin{equation}\label{eq:4.40}
\begin{aligned}
&\int_{\mathbb{T}^d}\rho^{p-1}(x)\sum_{k\in \mathbb{Z}^d\setminus\{0\}}\int_{\mathbb{T}^d}\frac{\rho(x)-\rho(y)}{|x-y+k|^{d+\alpha}}dydx\\
&=\sum_{k\in \mathbb{Z}^d\setminus\{0\}}\int_{\mathbb{T}^d}\rho^{p-1}(x)
\int_{\mathbb{T}^d}\frac{\rho(x)-\rho(y)}{|x-y+k|^{d+\alpha}}dydx\\
&\geq \frac{2}{p} \sum_{k\in \mathbb{Z}^d\setminus\{0\}}\int_{\mathbb{T}^d}\rho^{\frac{p}{2}}(x)\int_{\mathbb{T}^d}
\frac{\rho^{\frac{p}{2}}(x)-\rho^{\frac{p}{2}}(y)}{|x-y+k|^{d+\alpha}}dydx\\
&=\frac{1}{p}\sum_{k\in \mathbb{Z}^d\setminus\{0\}}\int_{\mathbb{T}^d}
\rho^{\frac{p}{2}}(x)\int_{\mathbb{T}^d}\frac{\rho^{\frac{p}{2}}(x)-\rho^{\frac{p}{2}}(y)}{|x-y+k|^{d+\alpha}}dydx\\
&\ \ \ -\frac{1}{p}\sum_{k\in \mathbb{Z}^d\setminus\{0\}}\int_{\mathbb{T}^d}
\rho^{\frac{p}{2}}(x)\int_{\mathbb{T}^d}\frac{\rho^{\frac{p}{2}}(y)-\rho^{\frac{p}{2}}(x)}{|y-x+(-k)|^{d+\alpha}}dydx\\
&=\frac{1}{p}\sum_{k\in \mathbb{Z}^d\setminus\{0\}}\int_{\mathbb{T}^d}
\rho^{\frac{p}{2}}(x)\int_{\mathbb{T}^d}\frac{\rho^{\frac{p}{2}}(x)-\rho^{\frac{p}{2}}(y)}{|x-y+k|^{d+\alpha}}dydx\\
&\ \ \ -\frac{1}{p}\sum_{k\in \mathbb{Z}^d\setminus\{0\}}\int_{\mathbb{T}^d}
\rho^{\frac{p}{2}}(y)\int_{\mathbb{T}^d}\frac{\rho^{\frac{p}{2}}(x)-\rho^{\frac{p}{2}}(y)}{|x-y+k|^{d+\alpha}}dxdy\\
&=\frac{1}{p}\sum_{k\in \mathbb{Z}^d\setminus\{0\}}
\left( \int_{\mathbb{T}^d}\int_{\mathbb{T}^d}
\frac{\rho^{\frac{p}{2}}(x)-\rho^{\frac{p}{2}}(y)}{|x-y+k|^{d+\alpha}}(\rho^{\frac{p}{2}}(x)-\rho^{\frac{p}{2}}(y))dxdy \right)\\
&=\frac{1}{p}\sum_{k\in \mathbb{Z}^d\setminus\{0\}}
\left( \int_{\mathbb{T}^d}\int_{\mathbb{T}^d}
\frac{(\rho^{\frac{p}{2}}(x)-\rho^{\frac{p}{2}}(y))^2}{|x-y+k|^{d+\alpha}}dxdy \right)\geq0,
\end{aligned}
\end{equation}
so we finish the proof of (\ref{eq:4.34}). For the second term of the right-hand side of (\ref{eq:4.33}), we deduce by (\ref{eq:4.40}) that integrate over $\mathbb{T}^d$, one get
\begin{equation}\label{eq:4.41}
\frac{1}{2}\int_{\mathbb{T}^d}C_{\alpha,d}\rho^{p-1}(x)\sum_{k\in \mathbb{Z}^d\setminus\{0\}}\int_{\mathbb{T}^d}\frac{\rho(x)-\rho(y)}{|x-y+k|^{d+\alpha}}dydx\geq0.
\end{equation}
By a discussion similar to (\ref{eq:4.41}), for the third term of the right-hand side of (\ref{eq:4.33}), integrate over $\mathbb{T}^d$, one has
\begin{equation}\label{eq:4.42}
\int_{\mathbb{T}^d} C^\ast\rho^{p-1}(x)\int_{\mathbb{T}^d}\frac{\rho(x)-\rho(y)}{|x-y|^{d+\alpha}}dydx\geq0.
\end{equation}
The fourth term of the right-hand side of (\ref{eq:4.33}) can be expressed as
\begin{equation}\label{eq:4.43}
\begin{aligned}
C_\ast\rho^{p-1}(x)\int_{\mathbb{T}^d\setminus\Omega_3}\frac{\rho(x)-\rho(y)}{|x-y|^{d+\alpha}}dy
&=C_\ast\rho^{p-1}(x)\int_{\mathbb{T}^d\setminus\Omega_3}\frac{\rho(x)}{|x-y|^{d+\alpha}}dy\\
&\ \ \ -C_\ast\rho^{p-1}(x)\int_{\mathbb{T}^d\setminus\Omega_3}\frac{\rho(y)}{|x-y|^{d+\alpha}}dy.
\end{aligned}
\end{equation}
As $\rho(x)\geq 0$, for the first term of the right-hand side of (\ref{eq:4.43}), one has
\begin{equation}\label{eq:4.44}
C_\ast\rho^{p-1}(x)\int_{\mathbb{T}^d\setminus\Omega_3}\frac{\rho(x)}{|x-y|^{d+\alpha}}dy\geq 0,
\end{equation}
and we deduce by the definition of $\Omega_3$ in (\ref{eq:4.28}) and (\ref{eq:4.29}) that
\begin{equation}\label{eq:4.45}
\begin{aligned}
C_\ast\rho^{p-1}(x)\int_{\mathbb{T}^d\setminus\Omega_3}\frac{\rho(y)}{|x-y|^{d+\alpha}}dy
&\leq C(1+\rho(\overline{x})^{\frac{d+\alpha}{\alpha+d-\beta}})\rho^{p-1}(x)\int_{\mathbb{T}^d\setminus\Omega_3}\rho(y)dy\\
&\leq C(1+\rho(\overline{x})^{\frac{d+\alpha}{\alpha+d-\beta}}) |\mathbb{T}^d|^{\frac{p-1}{p}}\|\rho\|_{L^p}\rho^{p-1}(x).
\end{aligned}
\end{equation}
Combining  (\ref{eq:4.44}), (\ref{eq:4.45}) and H\"{o}lder's inequality, integrate both sides of (\ref{eq:4.43}) over $\mathbb{T}^d$ to obtain
\begin{equation}\label{eq:4.46}
\begin{aligned}
-&\int_{\mathbb{T}^d}C_\ast\rho^{p-1}(x)\int_{\mathbb{T}^d\setminus\Omega_3}\frac{\rho(x)-\rho(y)}{|x-y|^{d+\alpha}}dydx\\
&\leq C(1+\rho(\overline{x})^{\frac{d+\alpha}{\alpha+d-\beta}}) |\mathbb{T}^d|^{\frac{p-1}{p}}\|\rho\|_{L^p}\int_{\mathbb{T}^d}\rho^{p-1}(x)dx \\
&\leq C(1+\rho(\overline{x})^{\frac{d+\alpha}{\alpha+d-\beta}})|\mathbb{T}^d|\|\rho\|^p_{L^p}.
\end{aligned}
\end{equation}
According to the definition of $\Omega_3$ in (\ref{eq:4.28}) and (\ref{eq:4.29}), the fifth term of the right-hand side of (\ref{eq:4.33}) can be estimated as
\begin{equation}\label{eq:4.47}
\left|\frac{p-1}{p} \rho^{p}(x)\int_{\mathbb{T}^d\setminus\Omega_3}\Delta K(x-y)(\rho(x)-\rho(y))dy\right|\leq C(\rho(\overline{x})+\rho(\overline{x})^{1+\frac{\beta-d}{\alpha+d-\beta}})\rho^{p}(x).
\end{equation}
Integrate the fifth term of the right-hand side of (\ref{eq:4.33}) over $\mathbb{T}^d$, we deduce by (\ref{eq:4.47}) and H\"{o}lder's inequality, one has
\begin{equation}\label{eq:4.48}
\begin{aligned}
& \int_{\mathbb{T}^d}\frac{p-1}{p}\rho^{p}(x)\int_{\mathbb{T}^d\setminus\Omega_3}\Delta K(x-y)(\rho(x)-\rho(y))dydx\\
&\leq C(\rho(\overline{x})+\rho(\overline{x})^{1+\frac{\beta-d}{\alpha+d-\beta}}) \int_{\mathbb{T}^d}\rho^{p}(x)dx\\
&\leq C(\rho(\overline{x})+\rho(\overline{x})^{1+\frac{\beta-d}{\alpha+d-\beta}})\|\rho\|^{p}_{L^p}.
\end{aligned}
\end{equation}
Notice that $\rho(\overline{x})$ is defined in (\ref{eq:4.5}), the constant $C$ in (\ref{eq:4.46}) and (\ref{eq:4.48}) is dependent of $p,\alpha, \beta$ and $d$. Integrate the first term of the right-hand side of (\ref{eq:4.33}) over $\mathbb{T}^d$, and apply Lemma \ref{lem:2.1}, to obtain
\begin{equation}\label{eq:4.49}
\frac{1}{2}\int_{\mathbb{T}^d}(-\Delta)^{\frac{\alpha}{2}}\rho \rho^{p-1} dx\geq\frac{1}{p}\|\rho^{\frac{p}{2}}\|^2_{\dot{H}^{\frac{\alpha}{2}}}.
\end{equation}
Combining (\ref{eq:4.33}), (\ref{eq:4.41}), (\ref{eq:4.42}), (\ref{eq:4.46}), (\ref{eq:4.48}) and (\ref{eq:4.49}), the second term and third term of left-hand side of (\ref{eq:4.25}) can be estimated as
\begin{equation}\label{eq:4.50}
\begin{aligned}
&-\int_{\mathbb{T}^d}(-\Delta)^{\frac{\alpha}{2}}\rho\rho^{p-1}dx-\frac{p-1}{p}\int_{\mathbb{T}^d}\Delta K\ast \rho \rho^{p}dx\\
&\leq -\frac{1}{p}\|\rho^{\frac{p}{2}}\|^2_{\dot{H}^{\frac{\alpha}{2}}}+ C\left((1+\rho(\overline{x})^{\frac{d+\alpha}{\alpha+d-\beta}})+(\rho(\overline{x})
+\rho(\overline{x})^{1+\frac{\beta-d}{\alpha+d-\beta}})\right)\|\rho\|^{p}_{L^p}.
\end{aligned}
\end{equation}
By Corollary \ref{cor:4.2}, (\ref{eq:4.25}) and (\ref{eq:4.50}), we deduce that for any $0\leq t\leq \tau_0$, one get
\begin{equation}\label{eq:4.51}
\begin{aligned}
\frac{d}{dt}\|\rho\|^{p}_{L^p}
&\leq -\|\rho^{\frac{p}{2}}\|^2_{\dot{H}^{\frac{\alpha}{2}}}+ pC\left((1+\rho(\overline{x})^{\frac{d+\alpha}{\alpha+d-\beta}})+(\rho(\overline{x})
+\rho(\overline{x})^{1+\frac{\beta-d}{\alpha+d-\beta}})\right)\|\rho\|^{p}_{L^p}\\
&\leq -\|\rho^{\frac{p}{2}}\|^2_{\dot{H}^{\frac{\alpha}{2}}}+
 pC\left((1+(2C_\infty)^{\frac{d+\alpha}{\alpha+d-\beta}})+(2C_\infty
+(2C_\infty)^{1+\frac{\beta-d}{\alpha+d-\beta}})\right)\|\rho\|^{p}_{L^p},
\end{aligned}
\end{equation}
where $\tau_0$ is defined in (\ref{eq:4.20}). If denote
\begin{equation}\label{eq:4.52}
C_1=pC\left((1+(2C_\infty)^{\frac{d+\alpha}{\alpha+d-\beta}})+(2C_\infty
+(2C_\infty)^{1+\frac{\beta-d}{\alpha+d-\beta}})\right).
\end{equation}
Then the (\ref{eq:4.51}) can be written as
\begin{equation}\label{eq:4.53}
\frac{d}{dt}\|\rho\|^{p}_{L^p}\leq-\|\rho^{\frac{p}{2}}\|^2_{\dot{H}^{\frac{\alpha}{2}}}+C_1\|\rho\|^{p}_{L^p}.
\end{equation}
Define
\begin{equation}\label{eq:4.54}
\tau'_0=\min\left\{ \frac{p\ln 2}{C_1},\tau_0\right\}.
\end{equation}
Since $\|\rho_0-\overline{\rho}\|_{L^p}\leq D_0$, then $\|\rho_0\|_{L^p}\leq D_0+\overline{\rho}$. We deduce by solving  the differential inequality in (\ref{eq:4.53}) that for any $0\leq t\leq \tau'_0$, one has
$$
\|\rho(t,\cdot)\|_{L^p}\leq 2(D_0+\overline{\rho}).
$$
This completes the proof of Lemma \ref{lem:4.3}.
\end{proof}

\vskip .1in

For the local $L^2$ estimate of solution to Equation (\ref{eq:1.1}), we obtain the  $L^2$ estimate of solution is independent of $A$. In the proof, we avoid using some energy inequalities, and the main idea of proof is similar with  Lemma \ref{lem:4.3}. However, the $\rho-\overline{\rho}$ is not always a positive, the details of proof can have some different.

\vskip .1in

\begin{lemma}\label{lem:4.4}
Let $\beta-d<\alpha<2, \beta\in(d,d+1),d>2$, $\rho(t,x)$ is the solution of Equation (\ref{eq:1.1}) with initial data $\rho_0(x)$. Suppose that $\rho_0(x)$ satisfies (\ref{eq:4.14})-(\ref{eq:4.16}). Then there exist a time $\tau_1>0$, for any $0\leq t \leq \tau_1$, we have
$$
\|\rho(t,\cdot)-\overline{\rho}\|_{L^2}\leq 2(B_0^2-\overline{\rho}^2)^{\frac{1}{2}}.
$$
\end{lemma}

\vskip .05in

\begin{proof}
Let us multiply both sides of Equation (\ref{eq:1.1}) by $\rho-\overline{\rho}$ and integrate over $\mathbb{T}^d$, to obtain
\begin{equation}\label{eq:4.55}
\begin{aligned}
\frac{1}{2}\frac{d}{dt}\|\rho-\overline{\rho}\|_{L^2}^2+A\int_{\mathbb{T}^d}u\cdot\nabla \rho(\rho-\overline{\rho})dx&+\int_{\mathbb{T}^d}(-\Delta)^{\frac{\alpha}{2}}\rho(\rho-\overline{\rho})dx\\
&+\int_{\mathbb{T}^d}\nabla\cdot(\rho B(\rho))(\rho-\overline{\rho})dx=0.
\end{aligned}
\end{equation}
As the incompressibility of $u$, the second term of the left-hand side of (\ref{eq:4.55}) can be estimated as
\begin{equation}\label{eq:4.56}
A\int_{\mathbb{T}^d}u\cdot\nabla \rho(\rho-\overline{\rho})dx=0.
\end{equation}
Integrating by parts the fourth term of the left-hand side of (\ref{eq:4.55}), one get
\begin{equation}\label{eq:4.57}
\begin{aligned}
\int_{\mathbb{T}^d} \nabla(\rho B(\rho))(\rho-\overline{\rho})dx&=\int_{\mathbb{T}^d} \nabla\rho\cdot B(\rho)(\rho-\overline{\rho})dx+\int_{\mathbb{T}^d} \rho \nabla\cdot B(\rho)(\rho-\overline{\rho})dx\\
&=-\frac{1}{2}\int_{\mathbb{T}^d}(\rho-\overline{\rho})^2\Delta K\ast \rho dx
+\int_{\mathbb{T}^d}\rho(\rho-\overline{\rho})\Delta K\ast \rho dx.
\end{aligned}
\end{equation}
Combining (\ref{eq:4.55}), (\ref{eq:4.56}) and (\ref{eq:4.57}), to obtain
\begin{equation}\label{eq:4.58}
\begin{aligned}
\frac{1}{2}\frac{d}{d t}\|\rho-\overline{\rho}\|^2_{L^2}+\int_{\mathbb{T}^d}(-\Delta)^{\frac{\alpha}{2}}\rho(\rho-\overline{\rho}) dx
&-\frac{1}{2}\int_{\mathbb{T}^d}(\rho-\overline{\rho})^2\Delta K\ast \rho dx\\
&+\int_{\mathbb{T}^d}\rho(\rho-\overline{\rho})\Delta K\ast \rho dx=0.
\end{aligned}
\end{equation}
Next, we consider $(-\Delta)^{\frac{\alpha}{2}}\rho(\rho-\overline{\rho})$ and $-\frac{1}{2}(\rho-\overline{\rho})^2\Delta K\ast \rho+\rho(\rho-\overline{\rho})\Delta K\ast \rho$. According to the definition of fractional Laplacian in (\ref{eq:2.1}), the $(-\Delta)^{\frac{\alpha}{2}}\rho(\rho-\overline{\rho})$ can be written as
\begin{equation}\label{eq:4.59}
\begin{aligned}
(-\Delta)^{\frac{\alpha}{2}}\rho(\rho-\overline{\rho})
&=\frac{1}{2}(-\Delta)^{\frac{\alpha}{2}}\rho(\rho-\overline{\rho})\\
&\ \ \ +\frac{1}{2}C_{\alpha,d}(\rho-\overline{\rho})(x)\sum_{k\in \mathbb{Z}^d\setminus\{0\}}\int_{\mathbb{T}^d}\frac{\rho(x)-\rho(y)}{|x-y+k|^{d+\alpha}}dy\\
&\ \ \ +\frac{1}{2}C_{\alpha,d}(\rho-\overline{\rho})(x)\int_{\mathbb{T}^d\setminus\Omega_4}\frac{\rho(x)-\rho(y)}{|x-y|^{d+\alpha}}dy\\
&\ \ \ +\frac{1}{2}C_{\alpha,d}(\rho-\overline{\rho})(x)\int_{\Omega_4}\frac{\rho(x)-\rho(y)}{|x-y|^{d+\alpha}}dy,
\end{aligned}
\end{equation}
and for the $-\frac{1}{2}(\rho-\overline{\rho})^2\Delta K\ast \rho+\rho(\rho-\overline{\rho})\Delta K\ast \rho$, we deduce by the definition of $K$ in Section 1 that
\begin{equation}\label{eq:4.60}
\begin{aligned}
-&\frac{1}{2}(\rho-\overline{\rho})^2\Delta K\ast \rho+\rho(\rho-\overline{\rho})\Delta K\ast \rho\\
=&\frac{1}{2} (\rho-\overline{\rho})^{2}(x)\int_{\mathbb{T}^d}\Delta K(x-y)(\rho(x)-\rho(y))dy\\
&-\rho(\rho-\overline{\rho})(x)\int_{\mathbb{T}^d}\Delta K(x-y)(\rho(x)-\rho(y))dy\\
=&\frac{1}{2} (\rho-\overline{\rho})^{2}(x)\int_{\mathbb{T}^d\setminus\Omega_4}\Delta K(x-y)(\rho(x)-\rho(y))dy\\
&+\frac{1}{2}C_{\beta,d} (\rho-\overline{\rho})^{2}(x)\int_{\Omega_4} \frac{\rho(x)-\rho(y)}{|x-y|^{\beta}}dy\\
&-\rho(\rho-\overline{\rho})(x)\int_{\mathbb{T}^d\setminus\Omega_4}\Delta K(x-y)(\rho(x)-\rho(y))dy\\
&-C_{\beta,d}\rho(\rho-\overline{\rho})(x)\int_{\Omega_4} \frac{\rho(x)-\rho(y)}{|x-y|^{\beta}}dy,
\end{aligned}
\end{equation}
where $C_{\alpha,d}$ is defined in (\ref{eq:2.2}), $C_{\beta,d}$ is defined in (\ref{eq:3.12}). And the domain $\Omega_4=\Omega(x)$ depends on $x$, which is defined by
\begin{equation}\label{eq:4.61}
\Omega_4=\Omega(x)=\{y||x-y|\leq \epsilon_4\},\ \ \ 0<\epsilon_4\leq\frac{1}{4}.
\end{equation}
In this paper, we denote
\begin{equation}\label{eq:4.62}
\epsilon_4=\min \left\{ \frac{1}{4},\big( \frac{C_{\alpha,d}}{2C_{\beta,d}(\rho(\overline{x})+\overline{\rho})}\big)^{\frac{1}{\alpha+d-\beta}} \right\},
\end{equation}
where $\rho(\overline{x})$ is defined in (\ref{eq:4.5}). Then for any $x\in \mathbb{T}^d$, when $y\in \Omega_4$, one has
\begin{equation}\label{eq:4.63}
0<\frac{1}{2}|x-y|^{d+\alpha-\beta}C_{\beta,d}(\rho(x)+\overline{\rho})\leq \frac{1}{2}|x-y|^{d+\alpha-\beta}C_{\beta,d}(\rho(\overline{x})+\overline{\rho})<\frac{C_{\alpha,d}}{4}.
\end{equation}
Combining (\ref{eq:4.61}), (\ref{eq:4.62}) and (\ref{eq:4.63}), the second term and fourth term of right-hand side of (\ref{eq:4.60}) can be written as
\begin{equation}\label{eq:4.64}
\begin{aligned}
&\frac{1}{2}C_{\beta,d} (\rho-\overline{\rho})^{2}(x)\int_{\Omega_4} \frac{\rho(x)-\rho(y)}{|x-y|^{\beta}}dy\\
&\ \ \ -C_{\beta,d}(\rho-\overline{\rho})\rho(x)\int_{\Omega_4} \frac{\rho(x)-\rho(y)}{|x-y|^{\beta}}dy\\
&=-\frac{1}{2}(\rho+\overline{\rho})C_{\beta,d}(\rho-\overline{\rho})\int_{\Omega_4} \frac{\rho(x)-\rho(y)}{|x-y|^{\beta}}dy\\
&=-\frac{1}{2}(\rho+\overline{\rho})C_{\beta,d}(\rho-\overline{\rho})\int_{\Omega_4} \frac{\rho(x)-\rho(y)|x-y|^{d+\alpha-\beta}}{|x-y|^{d+\alpha}}dy\\
&\sim -C^\ast_0(\rho-\overline{\rho})\int_{\Omega_4} \frac{\rho(x)-\rho(y)}{|x-y|^{d+\alpha}}dy,
\end{aligned}
\end{equation}
where
\begin{equation}\label{eq:4.65}
0<C^\ast_0<\frac{C_{\alpha,d}}{4}.
\end{equation}
Combining the third term and the fourth term of right-hand side of (\ref{eq:4.59}), the second term and fourth term of right-hand side of (\ref{eq:4.60}), (\ref{eq:4.64}) and (\ref{eq:4.65}), one has
\begin{equation}\label{eq:4.66}
\begin{aligned}
&\frac{1}{2}C_{\alpha,d}(\rho-\overline{\rho})(x)\int_{\mathbb{T}^d\setminus\Omega_4}\frac{\rho(x)-\rho(y)}{|x-y|^{d+\alpha}}dy
+\frac{1}{2}C_{\alpha,d}(\rho-\overline{\rho})(x)\int_{\Omega_4}\frac{\rho(x)-\rho(y)}{|x-y|^{d+\alpha}}dy\\
&\ \ \ +\frac{1}{2}C_{\beta,d} (\rho-\overline{\rho})^{2}(x)\int_{\Omega_4} \frac{\rho(x)-\rho(y)}{|x-y|^{\beta}}dy-C_{\beta,d}(\rho-\overline{\rho})\rho(x)\int_{\Omega_4} \frac{\rho(x)-\rho(y)}{|x-y|^{\beta}}dy\\
&\sim\frac{1}{2}C_{\alpha,d}(\rho-\overline{\rho})(x)\int_{\mathbb{T}^d\setminus\Omega_4}
\frac{\rho(x)-\rho(y)}{|x-y|^{d+\alpha}}dy+\frac{1}{2}
C_{\alpha,d}(\rho-\overline{\rho})(x)\int_{\Omega_4}\frac{\rho(x)-\rho(y)}{|x-y|^{d+\alpha}}dy\\
&\ \ \ -C^\ast_0(\rho-\overline{\rho})\int_{\Omega_4} \frac{\rho(x)-\rho(y)}{|x-y|^{d+\alpha}}dy\\
&\sim \frac{1}{2}C_{\alpha,d}(\rho-\overline{\rho})(x)\int_{\mathbb{T}^d\setminus\Omega_4}\frac{\rho(x)-\rho(y)}{|x-y|^{d+\alpha}}dy
+C^\ast_1(\rho-\overline{\rho})\int_{\Omega_4} \frac{\rho(x)-\rho(y)}{|x-y|^{d+\alpha}}dy\\
&=C_{1,\ast}(\rho-\overline{\rho})(x)\int_{\mathbb{T}^d\setminus\Omega_4}\frac{\rho(x)-\rho(y)}{|x-y|^{d+\alpha}}dy
+C^\ast_1(\rho-\overline{\rho})\int_{\mathbb{T}^d} \frac{\rho(x)-\rho(y)}{|x-y|^{d+\alpha}}dy,
\end{aligned}
\end{equation}
where
\begin{equation}\label{eq:4.67}
\frac{1}{4}C_{\alpha,d}<C^\ast_1<\frac{1}{2}C_{\alpha,d},\ \ \ \  C_{1,\ast}+C^\ast_1=\frac{1}{2}C_{\alpha,d}.
\end{equation}
Therefore, we deduce by (\ref{eq:4.59}), (\ref{eq:4.60}) and (\ref{eq:4.66}) that
\begin{equation}\label{eq:4.68}
\begin{aligned}
&(-\Delta)^{\frac{\alpha}{2}}\rho(\rho-\overline{\rho})-\frac{1}{2}(\rho-\overline{\rho})^2\Delta K\ast \rho+\rho(\rho-\overline{\rho})\Delta K\ast \rho\\
&\sim\frac{1}{2}(-\Delta)^{\frac{\alpha}{2}}\rho(\rho-\overline{\rho})
+\frac{1}{2}C_{\alpha,d}(\rho-\overline{\rho})(x)\sum_{k\in \mathbb{Z}^d\setminus\{0\}}\int_{\mathbb{T}^d}\frac{\rho(x)-\rho(y)}{|x-y+k|^{d+\alpha}}dy\\
&\ \ +\frac{1}{2} (\rho-\overline{\rho})^{2}(x)\int_{\mathbb{T}^d\setminus\Omega_4}\Delta K(x-y)(\rho(x)-\rho(y))dy\\
&\ \ -\rho(\rho-\overline{\rho})(x)\int_{\mathbb{T}^d\setminus\Omega_4}\Delta K(x-y)(\rho(x)-\rho(y))dy\\
&\ \ +C_{1,\ast}(\rho-\overline{\rho})(x)\int_{\mathbb{T}^d\setminus\Omega_4}\frac{\rho(x)-\rho(y)}{|x-y|^{d+\alpha}}dy
+C^\ast_1(\rho-\overline{\rho})\int_{\mathbb{T}^d} \frac{\rho(x)-\rho(y)}{|x-y|^{d+\alpha}}dy.
\end{aligned}
\end{equation}
By a discussion similar with (\ref{eq:4.41}), for the second term and the sixth term of right-hand side of (\ref{eq:4.68}) , integrate over $\mathbb{T}^d$, one get
\begin{equation}\label{eq:4.69}
\begin{aligned}
&\int_{\mathbb{T}^d}C_{\alpha,d}(\rho-\overline{\rho})(x)\sum_{k\in \mathbb{Z}^d\setminus\{0\}}\int_{\mathbb{T}^d}\frac{\rho(x)-\rho(y)}{|x-y+k|^{d+\alpha}}dydx\\
&=\int_{\mathbb{T}^d}C_{\alpha,d}(\rho-\overline{\rho})(x)\sum_{k\in \mathbb{Z}^d\setminus\{0\}}\int_{\mathbb{T}^d}\frac{(\rho(x)-\overline{\rho})-(\rho(y)-\overline{\rho})}
{|x-y+k|^{d+\alpha}}dydx\geq0,
\end{aligned}
\end{equation}
and
\begin{equation}\label{eq:4.70}
\int_{\mathbb{T}^d} C_1^\ast(\rho-\overline{\rho})(x)\int_{\mathbb{T}^d}\frac{\rho(x)-\rho(y)}{|x-y|^{d+\alpha}}dydx\geq0.
\end{equation}
According to the definitions of $K$ in Section 1, the domain $\Omega_4$ in (\ref{eq:4.61}) and the $\rho(\overline{x})$ in (\ref{eq:4.5}), to obtain
\begin{equation}\label{eq:4.71}
|\Delta K(x-y)|\leq C(1+\rho(\overline{x})^{\frac{\beta}{\alpha+d-\beta}}), \ \ \  y\in \mathbb{T}^d\backslash\Omega_4,
\end{equation}
and
\begin{equation}\label{eq:4.72}
\int_{\mathbb{T}^d\setminus\Omega_4}|\Delta K(x-y)|dy\leq C(1+\rho(\overline{x})^{\frac{\beta-d}{\alpha+d-\beta}}),
\end{equation}
where $C=C(\alpha,d,\beta,\overline{\rho})$. Then the third term of right-hand side in (\ref{eq:4.68}) can be estimated as
\begin{equation}\label{eq:4.73}
\begin{aligned}
&\left|\frac{1}{2} (\rho-\overline{\rho})^{2}(x)\int_{\mathbb{T}^d\setminus\Omega_4}\Delta K(x-y)(\rho(x)-\rho(y))dy\right|\\
&\leq \rho(\overline{x}) (\rho-\overline{\rho})^{2}(x)\int_{\mathbb{T}^d\setminus\Omega_4}|\Delta K(x-y)|dy\\
&\leq C(\rho(\overline{x})+\rho(\overline{x})^{1+\frac{\beta-d}{\alpha+d-\beta}})(\rho-\overline{\rho})^{2}(x).
\end{aligned}
\end{equation}
Integrate the third term of right-hand side in (\ref{eq:4.68}) over $\mathbb{T}^d$,  we deduced by (\ref{eq:4.73}) that
\begin{equation}\label{eq:4.74}
\begin{aligned}
-&\frac{1}{2}\int_{\mathbb{T}^d} (\rho-\overline{\rho})^{2}(x)\int_{\mathbb{T}^d\setminus\Omega_4}\Delta K(x-y)(\rho(x)-\rho(y))dydx\\
&\leq C(\rho(\overline{x})+\rho(\overline{x})^{1+\frac{\beta-d}{\alpha+d-\beta}})
\int_{\mathbb{T}^d}(\rho-\overline{\rho})^{2}(x)dx\\
&\leq C(\rho(\overline{x})+\rho(\overline{x})^{1+\frac{\beta-d}{\alpha+d-\beta}})\|\rho-\overline{\rho}\|^2_{L^2}.
\end{aligned}
\end{equation}
Combining (\ref{eq:4.71}), (\ref{eq:4.72}) and H\"{o}lder's inequality, for the fourth term of right-hand side of (\ref{eq:4.68}), one get
\begin{equation}\label{eq:4.75}
\begin{aligned}
&\left|(\rho-\overline{\rho})\rho(x)\int_{\mathbb{T}^d\setminus\Omega_4}\Delta K(x-y)(\rho(x)-\rho(y))dy\right|\\
&=\left|(\rho-\overline{\rho})\rho(x)\int_{\mathbb{T}^d\setminus\Omega_4}\Delta K(x-y)\left((\rho(x)-\overline{\rho})-(\rho(y)-\overline{\rho})\right)dy\right|\\
&\leq \rho(\overline{x})|(\rho-\overline{\rho})|^2\int_{\mathbb{T}^d\setminus\Omega_4}|\Delta K(x-y)|dy\\
&\ \  +\rho(\overline{x})|(\rho-\overline{\rho})|\int_{\mathbb{T}^d\setminus\Omega_4}|\Delta K(x-y)||(\rho(y)-\overline{\rho})|dy\\
&\leq C(\rho(\overline{x})+\rho(\overline{x})^{1+\frac{\beta-d}{\alpha+d-\beta}})|(\rho-\overline{\rho})|^2\\
&\ \ +C(\rho(\overline{x})+\rho(\overline{x})^{1+\frac{\beta-d}{\alpha+d-\beta}})|(\rho-\overline{\rho})|
\int_{\mathbb{T}^d\setminus\Omega_4}|(\rho(y)-\overline{\rho})|dy\\
&\leq C(\rho(\overline{x})+\rho(\overline{x})^{1+\frac{\beta-d}{\alpha+d-\beta}})|(\rho-\overline{\rho})|^2\\
&\ \ +C\sqrt{|\mathbb{T}^d|}(\rho(\overline{x})+\rho(\overline{x})^{1+\frac{\beta-d}{\alpha+d-\beta}})|(\rho-\overline{\rho})|
\|\rho-\overline{\rho})\|_{L^2}\\
&\leq  C(\rho(\overline{x})+\rho(\overline{x})^{1+\frac{\beta-d}{\alpha+d-\beta}})\left((\rho-\overline{\rho})|^2+|(\rho-\overline{\rho})|
\|\rho-\overline{\rho})\|_{L^2}\right).
\end{aligned}
\end{equation}
According to (\ref{eq:4.75}) and H\"{o}lder's inequality, integrate the fourth term of right-hand side in (\ref{eq:4.68}) over $\mathbb{T}^d$, one has
\begin{equation}\label{eq:4.76}
\begin{aligned}
&\int_{\mathbb{T}^d}(\rho-\overline{\rho})\rho(x)\int_{\mathbb{T}^d\setminus\Omega_4}\Delta K(x-y)(\rho(x)-\rho(y))dydx\\
&\leq C(\rho(\overline{x})+\rho(\overline{x})^{1+\frac{\beta-d}{\alpha+d-\beta}})
\int_{\mathbb{T}^d}\left((\rho-\overline{\rho})|^2+|(\rho-\overline{\rho})|
\|\rho-\overline{\rho})\|_{L^2}\right)dx\\
&\leq C(\rho(\overline{x})+\rho(\overline{x})^{1+\frac{\beta-d}{\alpha+d-\beta}})
(\|\rho-\overline{\rho}\|^2_{L^2}+\sqrt{|\mathbb{T}^d|}\|\rho-\overline{\rho}\|^2_{L^2})\\
&\leq C(\rho(\overline{x})+\rho(\overline{x})^{1+\frac{\beta-d}{\alpha+d-\beta}})
\|\rho-\overline{\rho}\|^2_{L^2}.
\end{aligned}
\end{equation}
The fifth term of right-hand side of (\ref{eq:4.68}) can be written as
\begin{equation}\label{eq:4.77}
\begin{aligned}
&C_{1,\ast}(\rho-\overline{\rho})(x)\int_{\mathbb{T}^d\setminus\Omega_4}\frac{\rho(x)-\rho(y)}{|x-y|^{d+\alpha}}dy\\
&=C_{1,\ast}(\rho-\overline{\rho})(x)\int_{\mathbb{T}^d\setminus\Omega_4}
\frac{(\rho(x)-\overline{\rho})-(\rho(y)-\overline{\rho})}{|x-y|^{d+\alpha}}dy\\
&= C_{1,\ast}(\rho-\overline{\rho})^2(x)\int_{\mathbb{T}^d\setminus\Omega_4}
\frac{1}{|x-y|^{d+\alpha}}dy\\
&\ \ -C_{1,\ast}\rho-\overline{\rho}(x)\int_{\mathbb{T}^d\setminus\Omega_4}
\frac{\rho(y)-\overline{\rho}}{|x-y|^{d+\alpha}}dy.
\end{aligned}
\end{equation}
Obviously, for the first term of the right-hand side of (\ref{eq:4.77}), one has
\begin{equation}\label{eq:4.78}
C_{1,\ast}(\rho-\overline{\rho})^2(x)\int_{\mathbb{T}^d\setminus\Omega_4}
\frac{1}{|x-y|^{d+\alpha}}dy\geq 0,
\end{equation}
and combining the definition of $\Omega_4$ in (\ref{eq:4.61}) and H\"{o}lder's inequality, the second term of the right-hand side of (\ref{eq:4.77}) can be estimated as
\begin{equation}\label{eq:4.79}
\begin{aligned}
&\left|C_{1,\ast}(\rho-\overline{\rho})(x)\int_{\mathbb{T}^d\setminus\Omega_4}
\frac{\rho(y)-\overline{\rho}}{|x-y|^{d+\alpha}}dy\right|\\
&\leq  C(1+\rho(\overline{x})^{\frac{\alpha+d}{\alpha+d-\beta}})
|(\rho-\overline{\rho})|(x)\int_{\mathbb{T}^d\setminus\Omega_4}|\rho-\overline{\rho}|dy\\
&\leq C\sqrt{|\mathbb{T}^d|}(1+\rho(\overline{x})^{\frac{\alpha+d}{\alpha+d-\beta}})
\|\rho-\overline{\rho})\|_{L^2}|(\rho-\overline{\rho})|(x).
\end{aligned}
\end{equation}
Through (\ref{eq:4.79}) and H\"{o}lder's inequality, integrate the fifth term of right-hand side in (\ref{eq:4.68}) over $\mathbb{T}^d$, to obtain
\begin{equation}\label{eq:4.80}
\begin{aligned}
-&\int_{\mathbb{T}^d}C_{1,\ast}(\rho-\overline{\rho})(x)
\int_{\mathbb{T}^d\setminus\Omega_4}\frac{\rho(x)-\rho(y)}{|x-y|^{d+\alpha}}dydx\\
&\leq C\sqrt{|\mathbb{T}^d|}(1+\rho(\overline{x})^{\frac{\alpha+d}{\alpha+d-\beta}})
\|\rho-\overline{\rho})\|_{L^2}\int_{\mathbb{T}^d}|(\rho-\overline{\rho})|(x)dx\\
&\leq C|\mathbb{T}^d|(1+\rho(\overline{x})^{\frac{\alpha+d}{\alpha+d-\beta}})
\|\rho-\overline{\rho})\|^2_{L^2}\\
&\leq C(1+\rho(\overline{x})^{\frac{\alpha+d}{\alpha+d-\beta}})\|\rho-\overline{\rho})\|^2_{L^2}.
\end{aligned}
\end{equation}
Integrate the first term of the right-hand side of (\ref{eq:4.68}) over $\mathbb{T}^d$, and applying the Lemma \ref{lem:2.1}, one get
\begin{equation}\label{eq:4.81}
\frac{1}{2}\int_{\mathbb{T}^d}(-\Delta)^{\frac{\alpha}{2}}\rho(\rho-\overline{\rho}) dx=\frac{1}{2}\|\rho\|^2_{\dot{H}^{\frac{\alpha}{2}}}.
\end{equation}
Combining (\ref{eq:4.68}), (\ref{eq:4.69}), (\ref{eq:4.70}), (\ref{eq:4.74}), (\ref{eq:4.76}), (\ref{eq:4.80}) and (\ref{eq:4.81}), the second term, the third term and the fourth term of the left-hand side of (\ref{eq:4.58}) can be estimated as
\begin{equation}\label{eq:4.82}
\begin{aligned}
-&\int_{\mathbb{T}^d}(-\Delta)^{\frac{\alpha}{2}}\rho(\rho-\overline{\rho}) dx
+\frac{1}{2}\int_{\mathbb{T}^d}(\rho-\overline{\rho})^2\Delta K\ast \rho dx
-\int_{\mathbb{T}^d}\rho(\rho-\overline{\rho})\Delta K\ast \rho dx\\
\leq &-\frac{1}{2}\|\rho\|^2_{\dot{H}^{\frac{\alpha}{2}}}
+C(\rho(\overline{x})+\rho(\overline{x})^{1+\frac{\beta-d}{\alpha+d-\beta}})\|\rho-\overline{\rho}\|^2_{L^2}\\
&+ C(\rho(\overline{x})+\rho(\overline{x})^{1+\frac{\beta-d}{\alpha+d-\beta}})
\|\rho-\overline{\rho}\|^2_{L^2}
+C(1+\rho(\overline{x})^{\frac{\alpha+d}{\alpha+d-\beta}})\|\rho-\overline{\rho})\|^2_{L^2}\\
\leq &-\frac{1}{2}\|\rho\|^2_{\dot{H}^{\frac{\alpha}{2}}}
+C\left((\rho(\overline{x})+\rho(\overline{x})^{1+\frac{\beta-d}{\alpha+d-\beta}})
+(1+\rho(\overline{x})^{\frac{\alpha+d}{\alpha+d-\beta}})\right)
\|\rho-\overline{\rho}\|^2_{L^2}.
\end{aligned}
\end{equation}
Combining $\|\rho_0\|_{L^\infty}\leq C_\infty$, Corollary \ref{cor:4.2} and Lemma \ref{lem:4.3}, we imply that for any $t\in [0,\tau'_0]$, the second term of right-hand side of (\ref{eq:4.82}) can be expressed as
\begin{equation}\label{eq:4.83}
\begin{aligned}
&C\left((\rho(\overline{x})+\rho(\overline{x})^{1+\frac{\beta-d}{\alpha+d-\beta}})
+(1+\rho(\overline{x})^{\frac{\alpha+d}{\alpha+d-\beta}})\right)\|\rho-\overline{\rho}\|^2_{L^2}\\
&\leq C\left((2C_\infty+(2C_\infty)^{1+\frac{\beta-d}{\alpha+d-\beta}})
+(1+(2C_\infty)^{\frac{\alpha+d}{\alpha+d-\beta}})\right)\|\rho-\overline{\rho}\|^2_{L^2},
\end{aligned}
\end{equation}
where $\tau'_0$ is defined in (\ref{eq:4.54}). If denote
\begin{equation}\label{eq:4.84}
C_2=2C\left((2C_\infty+(2C_\infty)^{1+\frac{\beta-d}{\alpha+d-\beta}})
+(1+(2C_\infty)^{\frac{\alpha+d}{\alpha+d-\beta}})\right).
\end{equation}
Then we deduce by (\ref{eq:4.58}), (\ref{eq:4.82}), (\ref{eq:4.83}) and (\ref{eq:4.84}) that for any $t\in [0,\tau'_0]$, one has
\begin{equation}\label{eq:4.85}
\frac{d}{dt}\|\rho-\overline{\rho}\|^2_{L^2}\leq -\|\rho\|^2_{\dot{H}^{\frac{\alpha}{2}}}+C_2\|\rho-\overline{\rho}\|^2_{L^2}.
\end{equation}
Define
\begin{equation}\label{eq:4.86}
\tau_1=\min\left\{ \frac{2\ln 2}{C_2},\tau'_0\right\}.
\end{equation}
Since $\|\rho_0-\overline{\rho}\|_{L^2}\leq (B^2_0-\overline{\rho}^2)^{\frac{1}{2}}$. By solving  the differential inequality in (\ref{eq:4.85}), then we deduce that for any $0<t\leq \tau_1$, one has
$$
\|\rho(t,\cdot)-\overline{\rho}\|_{L^2}\leq 2(B^2_0-\overline{\rho}^2)^{\frac{1}{2}}.
$$
This completes the proof of Lemma \ref{lem:4.4}.
\end{proof}

\vskip .1in

\begin{remark}\label{rem:14}
(1). In fact, since $L^1$ norm of solution of Equation (\ref{eq:1.1}) is conserved, so we also get the local $L^2, L^p$ estimates of the solution only depend on the local $L^\infty$ estimate by interpolation inequality. In this paper, we need to get the local $L^2, L^p$ estimates of the solution only depend on the $L^2$ and $ L^p$ of initial data $\rho_0$. This is a necessary technical requirement in our discussion.

\vskip .05in

(2). In fact, the estimation of $\|\rho-\overline{\rho}\|_{L^2}$ can also be established by Kato-Ponce type inequality (see \cite{Vicol.2014,Loukas.2014,Hopf.2018}), Young's inequality and Gagliardo-Nirenberg inequality. In this paper, we give a new proof, the main idea of proof is from the estimation of $\|\rho\|_{L^p}$ in Lemma \ref{lem:4.3}. Notice that, our proof is more general, the estimate of $\|\rho-\overline{\rho}\|_{L^2}$ can be obtained for the case of $\beta>d, \beta-d<\alpha<2$, see Remark \ref{rem:3}.

\vskip .05in

(3). In the proof of Lemma \ref{lem:4.3} and Lemma \ref{lem:4.4}, we omit the discussion of the commutativity of integration and summation. The strategies and ideas of discussion can be referred to \cite{Cordoba.2004}.

\vskip .05in

(4). Comparison with \cite{Shi.2019}, the local estimates of solution to Equation (\ref{eq:1.1}) is complicated and difficult, which is from the destabilizing effect of mixing. Namely, for the $L^\infty$ estimate, since $B(\rho)$ is strong singular kernel, the $\Delta K\notin L^1$, we use the singular integral formulation of fractional Laplacian to control the singular part of nonlinear term, which is applied in the proof of Proposition \ref{prop:3.1}. For the $L^p$ estimates, the nonlinear term produces an estimate of the derivative of solution by energy method. As technical difficulties, it can not be controlled by dissipative term.  According to of Proposition \ref{prop:3.1}, the estimate of the derivative of solution depends on constant $A$, see Remark \ref{rem:11}. Then the local time is very short if $A$ is large enough,  and it may not satisfy (\ref{eq:5.21}).  In other words, mixing has is destabilizing effect for the estimation of the derivative of solution. So we need some new techniques and piece estimate to overcome this difficulties. In this paper, we use the similar idea with the proof of $L^\infty$ estimate and some useful techniques, which is from 2d quasi-geostrophic equation (see \cite{Cordoba.2004}).
\end{remark}

\vskip .2in

\section{Global estimate of solution to equation (\ref{eq:1.1})}

\vskip .1in

In this section, we establish the global $L^\infty$ estimate of solution to Equation (\ref{eq:1.1}) by bootstrap argument. For the analysis of mixing effect, some new ideas and techniques are introduced, and obtain the global $L^\infty$ estimate by nonlinear maximum principle. Firstly, we give an approximation lemma.

\begin{lemma}\label{lem:5.1}
Let $\beta-d<\alpha<2, \beta\in (d,d+1), d>2$, $u(x)$ is weakly mixing. Suppose that $\rho(t,x), \eta(t,x)$ are the solution of Equations (\ref{eq:1.1}) and (\ref{eq:2.12}) respectively with initial data $\rho_0(x)$. If $\rho_0(x)$ satisfies (\ref{eq:4.14})-(\ref{eq:4.16}), then for every $t\in [0,\tau_1]$, there exist a finite positive constant $C>0$, such that
\begin{equation}\label{eq:5.1}
\|P_{N}(\rho-\eta)(t,\cdot)\|_{L^2}\leq  CNC_\infty(B^2_0-\overline{\rho}^2)^{\frac{1}{2}}t,
\end{equation}
where $\tau_1$ is defined in (\ref{eq:4.86}) and $P_N$ is defined in (\ref{eq:2.8}).
\end{lemma}

\vskip .1in

\begin{proof}
Consider the Equations (\ref{eq:1.1}) and (\ref{eq:2.12}), to obtain
\begin{equation}\label{eq:5.2}
\partial_t(\rho-\eta)+Au\cdot \nabla (\rho-\eta)+(-\Delta)^{\frac{\alpha}{2}}(\rho-\eta)+\nabla\cdot(\rho B(\rho))=0.
\end{equation}
By Duhamel's principe and (\ref{eq:2.14}), then the solution $\rho-\eta$  of (\ref{eq:5.2}) can be expressed as
\begin{equation}\label{eq:5.3}
(\rho-\eta)(t,x)=\int_{0}^{t}e^{-sH^\alpha_A}\left(\nabla\cdot(\rho B(\rho))(t-s,\cdot)\right)ds,
\end{equation}
where the definition of $e^{-sH^\alpha_A}$ can be referred to Section 2. The operator $P_N$ is applied to (\ref{eq:5.3}), then one get
\begin{equation}\label{eq:5.4}
\begin{aligned}
P_{N}(\rho-\eta)(t,x)&=\int_{0}^{t}P_{N}e^{-sH^\alpha_A}\left(\nabla\cdot(\rho B(\rho))(t-s,\cdot)\right)ds.
\end{aligned}
\end{equation}
Combining Lemma \ref{lem:2.9}, H\"{o}lder's inequality and Young's inequality, to obtain
\begin{equation}\label{eq:5.5}
\begin{aligned}
&\|P_{N}e^{-sH^\alpha_A}\left(\nabla\cdot(\rho B(\rho))(t-s,\cdot)\right)\|_{L^2}\\
&\leq \|P_{N}\left(\nabla\cdot(\rho B(\rho))(t-s,\cdot)\right)\|_{L^2} \\
&\leq N\|\rho\nabla K\ast(\rho-\overline{\rho})(t-s,\cdot)\|_{L^2}\\
&\leq CN\|\rho(t-s,\cdot)\|_{L^\infty}\|\nabla K\|_{L^1}\|(\rho-\overline{\rho})(t-s,\cdot)\|_{L^2}.
\end{aligned}
\end{equation}
For any $0\leq s\leq t\leq \tau_1$, where the $\tau_1$ is defined in (\ref{eq:4.86}), then $t-s\in [0,\tau_1]$.
We deduce by Corollary \ref{cor:4.2} and Lemma \ref{lem:4.4} that
\begin{equation}\label{eq:5.6}
\|\rho(t-s,\cdot)\|_{L^\infty}\leq 2C_\infty,\ \ \ \|(\rho-\overline{\rho})(t-s,\cdot)\|_{L^2}\leq 2(B^2_0-\overline{\rho}^2)^{\frac{1}{2}}.
\end{equation}
As $\|\nabla K\|_{L^1}$ is bounded, then for any $0 \leq t\leq \tau_1$, combining (\ref{eq:5.4}), (\ref{eq:5.5}) and (\ref{eq:5.6}), we have
\begin{equation}\label{eq:5.7}
\begin{aligned}
\|P_{N}(\rho-\eta)(t,\cdot)\|_{L^2}&\leq \int_{0}^{t}\|P_{N}e^{-sH^\alpha_A}\left(\nabla\cdot(\rho B(\rho))(t-s,\cdot)\right)\|_{L^2}ds\\
&\leq \int_{0}^{t}CN\|\rho(t-s,\cdot)\|_{L^\infty}\|\nabla K\|_{L^1}\|(\rho-\overline{\rho})(t-s,\cdot)\|_{L^2}ds\\
&\leq 4C\|\nabla K\|_{L^1}NC_\infty(B^2_0-\overline{\rho}^2)^{\frac{1}{2}}t\\
&\leq CNC_\infty(B^2_0-\overline{\rho}^2)^{\frac{1}{2}}t.
\end{aligned}
\end{equation}
This completes the proof of Lemma \ref{lem:5.1}.
\end{proof}

\vskip .1in

\begin{remark}\label{rem:15}
In the Lemma \ref{lem:5.1},  we only need to consider the low frequency part of the solution to  Equations (\ref{eq:1.1}) and (\ref{eq:2.12}) for establishing the approximation lemma, which  is different from \cite{Constantin.2008,Shi.2019,Hopf.2018,Kiselev.2016}. And for the case of $\beta>d, \beta-d<\alpha<2$ (see Remark \ref{rem:3}), the proof of Lemma \ref{lem:5.1} needs simple modification.
\end{remark}

\vskip .1in

Let us denote $\rho(t,x)$ is the solution of Equation (\ref{eq:1.1}) with initial data $\rho_0(x)$. Define
\begin{equation}\label{eq:5.8}
\Phi(A)=\inf_{t\geq0} \frac{\|\Lambda^{\frac{\alpha}{2}}(\rho-\overline{\rho})\|^2_{L^2}}{\|\rho-\overline{\rho}\|^2_{L^2}}
=\inf_{t\geq0} \frac{\|\Lambda^{\frac{\alpha}{2}}\rho\|^2_{L^2}}{\|\rho-\overline{\rho}\|^2_{L^2}}.
\end{equation}
Based on the idea of contradiction and the estimate of semigroup to linear problem, we establish the estimate of $\Phi(A)$, this is a new observation for mixing effect. The main lemma is as follows

\vskip .1in

\begin{lemma}\label{lem:5.2}
Let $\beta-d<\alpha<2, \beta\in (d,d+1), d> 2$, $\rho(t,x)$ is the solution of Equation (\ref{eq:1.1}) with initial data $\rho_0(x)$. Suppose that $\rho_0(x)$ satisfies (\ref{eq:4.14})-(\ref{eq:4.16}). If $u(x)$ is weakly mixing, then there exist a set $\Sigma(t)\subset [0,\tau_1]$, and $|\Sigma(t)|= \frac{\tau_1}{2}$,  such that for all $t\in\Sigma(t)$, we have
\begin{equation}\label{eq:5.9}
\lim_{A\rightarrow +\infty}\Phi(A)=+\infty,
\end{equation}
where $\tau_1$ is defined in (\ref{eq:4.86}), $\Phi(A)$ is defined in (\ref{eq:5.8}) and $|\Sigma(t)|$ is measure of $\Sigma(t)$.
\end{lemma}

\vskip .1in

\begin{proof}
In fact, we need to prove that for any $G>0$, there exists a positive constant $A_0$ and a set $\Sigma(t)\subset [0,\tau_1]$, and $|\Sigma(t)|= \frac{\tau_1}{2}$, such that for any $A\geq A_0$ and all $t\in\Sigma(t)$, one has
\begin{equation}\label{eq:5.10}
\frac{\|\Lambda^{\frac{\alpha}{2}}\rho\|^2_{L^2}}{\|\rho-\overline{\rho}\|^2_{L^2}}>G.
\end{equation}
Next, we give the proof by contradiction. If (\ref{eq:5.10}) is not true, then there exist a positive constant $\delta'>2C_2$, a sequence $\{A_n\}_{n=1}^{\infty}$ and a set $\Sigma_1(t)\subset [0,\tau_1]$, such that
\begin{equation}\label{eq:5.11}
\lim_{n\rightarrow \infty}A_n=\infty, \ \ \ |\Sigma_1(t)|=\frac{\tau_1}{2},
\end{equation}
where $C_2$ is defined in (\ref{eq:4.84}). And for any $A_n, n=1,2,\cdots$, one get
\begin{equation}\label{eq:5.12}
\sup_{t\in \Sigma_1(t)} \frac{\|\Lambda^{\frac{\alpha}{2}}\rho\|^2_{L^2}}{\|\rho-\overline{\rho}\|^2_{L^2}}\leq\delta'.
\end{equation}
Without loss of generality, we assume that for any $t\in  \Sigma_1(t)$, one has
\begin{equation}\label{eq:5.13}
\|\rho(t,\cdot)-\overline{\rho}\|_{L^2}\geq B_1,
\end{equation}
where $B_1$ is defined in (\ref{eq:4.17}), the details can be referred in Remark \ref{rem:16}. According to (\ref{eq:5.12}), we know that for any $t\in \Sigma_1(t)$, one get
\begin{equation}\label{eq:5.14}
\|\Lambda^{\frac{\alpha}{2}}\rho(t,\cdot)\|^2_{L^2}\leq \delta' \|\rho(t,\cdot)-\overline{\rho}\|^2_{L^2}.
\end{equation}
If choose $N$, such that
\begin{equation}\label{eq:5.15}
\lambda_N^{\frac{\alpha}{2}} > 4\delta'.
\end{equation}
We claim that for any $t\in\Sigma_1(t)$, we have
\begin{equation}\label{eq:5.16}
\|(I-P_N)(\rho(t,\cdot)-\overline{\rho})\|^2_{L^2}\leq\frac{1}{4}\|\rho(t,\cdot)-\overline{\rho}\|^2_{L^2}.
\end{equation}
Because if (\ref{eq:5.16}) is not true, then there exist $t'\in \Sigma_1(t)$, such that
\begin{equation}\label{eq:5.17}
\|(I-P_N)(\rho(t',\cdot)-\overline{\rho})\|^2_{L^2}> \frac{1}{4}\|\rho(t',\cdot)-\overline{\rho}\|^2_{L^2}.
\end{equation}
Notice that $P_N$ is defined in (\ref{eq:2.8}), we deduce by (\ref{eq:5.15}) and (\ref{eq:5.17}) that
\begin{equation}\label{eq:5.18}
\begin{aligned}
\|\Lambda^{\frac{\alpha}{2}}\rho(t',\cdot)\|^2_{L^2}&\geq\|\Lambda^{\frac{\alpha}{2}}(I-P_N)\rho(t',\cdot)\|^2_{L^2}\\
&\geq \lambda_N^{\frac{\alpha}{2}}\|(I-P_N)(\rho(t',\cdot)-\overline{\rho})\|^2_{L^2}\\
&>\delta'\|\rho(t',\cdot)-\overline{\rho}\|^2_{L^2}.
\end{aligned}
\end{equation}
This is in contradiction with (\ref{eq:5.14}), we finish the proof of (\ref{eq:5.16}). According to (\ref{eq:5.13}) and (\ref{eq:5.16}), we deduce that for any $t\in \Sigma_1(t)$, one get
\begin{equation}\label{eq:5.19}
\|P_N(\rho(t,\cdot)-\overline{\rho})\|^2_{L^2}\geq \frac{3}{4}\|\rho(t,\cdot)-\overline{\rho}\|^2_{L^2}\geq \frac{3}{4}B^2_1.
\end{equation}
Denote
\begin{equation}\label{eq:5.20}
T_0=\frac{80(B_0^2-\overline{\rho}^2)}{B_1^2}e^{\pi}.
\end{equation}
Combining $\lim\limits_{n\rightarrow \infty}A_n=\infty$ and Lemma \ref{lem:2.8}, for $A_n$ is large enough, we can define
\begin{equation}\label{eq:5.21}
\tau=\frac{T_0}{\Psi(H^\alpha_{A_n})}<\frac{\tau_1}{6}.
\end{equation}
We claim that there exists $t_0\in \Sigma_1(t)$ and $[t_0,t_0+\tau)\subset [0,\tau_1]$, such that
\begin{equation}\label{eq:5.22}
\frac{1}{\tau}\int_{t_0}^{t_0+\tau}\|P_N(\rho(t,\cdot)-\overline{\rho})\|^2_{L^2}dt> \frac{1}{10}B^2_1.
\end{equation}
If (\ref{eq:5.22}) is not true, then for all $ t_i\in \Sigma_1(t), [t_i, t_i+\tau)\subset [0,\tau_1]$, and
\begin{equation}\label{eq:5.23}
[t_i, t_i+\tau)\cap [t_j, t_j+\tau)=\emptyset,\ \ \ i=0,1,2,\cdots,
\end{equation}
one can get
\begin{equation}\label{eq:5.24}
\frac{1}{\tau}\int_{t_i}^{t_i+\tau}\|P_N(\rho(t,\cdot)-\overline{\rho})\|^2_{L^2}dt\leq \frac{1}{10}B^2_1.
\end{equation}
If denote
\begin{equation}\label{eq:5.25}
E=\bigcup_{i}[t_i, t_i+\tau), \ \ \ E_1=E\cap \Sigma_1(t)=\bigcup_{i}([t_i, t_i+\tau)\cap \Sigma_1(t)).
\end{equation}
Combining (\ref{eq:5.23}), (\ref{eq:5.24}) and (\ref{eq:5.25}), we imply that
\begin{equation}\label{eq:5.26}
\begin{aligned}
\int_{E_1}\|P_N(\rho(t,\cdot)-\overline{\rho})\|^2_{L^2}dt&\leq \int_{E}\|P_N(\rho(t,\cdot)-\overline{\rho})\|^2_{L^2}dt\\
&=\sum_{i}\int_{t_i}^{t_i+\tau}\|P_N(\rho(t,\cdot)-\overline{\rho})\|^2_{L^2}dt\\
&\leq \frac{1}{10}B^2_1\sum_{i}\tau\leq \frac{1}{10}B^2_1\tau_1.
\end{aligned}
\end{equation}
According to (\ref{eq:5.23}) and the definition of $E_1$ in (\ref{eq:5.25}), one has
\begin{equation}\label{eq:5.27}
\frac{\tau_1}{3}\leq |E_1|\leq \frac{\tau_1}{2}.
\end{equation}
As $E_1\subset \Sigma_1(t)$, then we deduce by (\ref{eq:5.19}) and (\ref{eq:5.27}) that
\begin{equation}\label{eq:5.28}
\int_{E_1}\|P_N(\rho(t,\cdot)-\overline{\rho})\|^2_{L^2}dt\geq \frac{3}{4}B^2_1|E_1|\geq \frac{B^2_1}{4}\tau_1.
\end{equation}
Obviously, (\ref{eq:5.26}) and (\ref{eq:5.28})  are contradictory, so we finish the proof of (\ref{eq:5.22}). Next, consider the equations
\begin{equation}\label{eq:5.29}
\partial_t\rho+A_n u\cdot \nabla \rho+(-\Delta)^{\frac{\alpha}{2}}\rho+\nabla\cdot(\rho B(\rho))=0,\ \ \ \rho_0(x)=\rho(t_0,x),
\end{equation}
and
\begin{equation}\label{eq:5.30}
\partial_t\eta+A_n u\cdot \nabla\eta +(-\Delta)^{\frac{\alpha}{2}}\eta=0, \ \ \ \omega_0(x)=\rho(t_0,x).
\end{equation}
Then for the solution $\rho(t,x)$ of Equation (\ref{eq:5.29}), we deduce by Corollary \ref{cor:4.2} and Lemma \ref{lem:4.4} that
\begin{equation}\label{eq:5.31}
\|\rho(t,\cdot)-\overline{\rho}\|_{L^2}\leq 2(B_0^2-\overline{\rho}^2)^{\frac{1}{2}},\quad \|\rho(t,\cdot)\|_{L^\infty}\leq 2C_{\infty},\quad 0\leq t \leq \tau_1,
\end{equation}
where $\tau_1$ be defined in (\ref{eq:4.54}). Combining the Equation (\ref{eq:5.30}) and (\ref{eq:2.14}), to obtain
\begin{equation}\label{eq:5.32}
\eta(t_0+t,x)-\overline{\rho}=e^{-tH^\alpha_{A_n}}(\rho(t_0,x)-\overline{\rho}).
\end{equation}
Therefore, we deduce by the Lemma \ref{lem:2.7}, (\ref{eq:5.20}), (\ref{eq:5.21}) and (\ref{eq:5.32}) that for $A_n$ is large enough, one has
\begin{equation}\label{eq:5.33}
\begin{aligned}
\frac{1}{\tau}\int_{0}^{\tau}\|P_N(\eta-\overline{\rho})\|_{L^2}^2dt&= \frac{1}{\tau}\int_{0}^{\tau}\|P_N e^{-tH^\alpha_{A_n}}(\rho(t_0,x)-\overline{\rho})\|_{L^2}^2dt \\
 &\leq\frac{1}{\tau}\int_{0}^{\tau}\|e^{-tH^\alpha_{A_n}}(\rho(t_0,x)-\overline{\rho})\|_{L^2}^2dt\\
 &\leq e^{\pi}\|\rho(t_0,x)-\overline{\rho}\|_{L^2}^2\frac{1}{\tau}\int_{0}^{\tau}e^{-2t\Psi(H^\alpha_{A_n})}dt\\
 &\leq \frac{e^{\pi}}{2\Psi(H^\alpha_{A_n})\tau}\|\rho(t_0,x)-\overline{\rho}\|_{L^2}^2\\
 &\leq \frac{e^{\pi}}{2T_0}\|\rho(t_0,x)-\overline{\rho}\|_{L^2}^2\\
 &\leq \frac{1}{40}B_1^2.
\end{aligned}
\end{equation}
For any $t\in [0,\tau]$ and $A_n$ is large enough, the right-hand side of (\ref{eq:5.1}) can be estimated as
\begin{equation}\label{eq:5.34}
\begin{aligned}
CN^2C^2_\infty(B^2_0-\overline{\rho}^2)t^2&\leq CN^2C^2_\infty(B^2_0-\overline{\rho}^2)\left(\frac{T_0}{\Psi(H^\alpha_{A_n})}\right)^2\leq \frac{1}{40}B_1^2.
\end{aligned}
\end{equation}
Then we deduce by  Lemma \ref{lem:5.1} and (\ref{eq:5.34}) that for any $t\in [0,\tau]$, one has
\begin{equation}\label{eq:5.35}
\|P_N(\rho-\eta)(t_0+t,\cdot)\|^2_{L^2}\leq \frac{1}{40}B^2_1.
\end{equation}
Furthermore, we deduce by (\ref{eq:5.33}) and (\ref{eq:5.35}) that for any $t\in [0,\tau]$, we have
\begin{equation}\label{eq:5.36}
\begin{aligned}
\frac{1}{\tau}\int_{0}^{\tau}\|P_N(\rho(t_0+t,\cdot)-\overline{\rho})\|_{L^2}^2dt
& \leq\frac{2}{\tau}\int_{0}^{\tau}\|P_N(\eta(t_0+t,\cdot)-\overline{\rho})\|_{L^2}^2dt\\
&\quad +\frac{2}{\tau}\int_{0}^{\tau}\|P_N(\rho(t_0+t,\cdot)-\eta(t_0+t,\cdot))\|_{L^2}^2dt\\
&\leq\frac{B^2_1}{10}.
\end{aligned}
\end{equation}
However, we deduce by (\ref{eq:5.22}) that
\begin{equation}\label{eq:5.37}
\frac{1}{\tau}\int_{0}^{\tau}\|P_N(\rho(t_0+t,\cdot)-\overline{\rho})\|_{L^2}^2dt= \frac{1}{\tau}\int_{t_0}^{t_0+\tau}\|P_N(\rho(t,\cdot)-\overline{\rho})\|^2_{L^2}dt> \frac{1}{10}B^2_1.
\end{equation}
Obviously, (\ref{eq:5.36}) and (\ref{eq:5.37}) are contradictory, namely, (\ref{eq:5.10}) is true. Then
\begin{equation}\label{eq:5.38}
\lim_{A\rightarrow +\infty}\Phi(A)=+\infty.
\end{equation}
This completes the proof of Lemma \ref{lem:5.2}.
\end{proof}

\vskip .1in

\begin{remark}\label{rem:16}
For (\ref{eq:5.13}), if there exist $t_2\in \Sigma_1(t)$, such that
$$
\|\rho(t_2,\cdot)-\overline{\rho}\|_{L^2}< B_1.
$$
then the local solution can be extended to $[0,t_2+\tau_1]$. The details can be referred to the proof of Proposition \ref{prop:5.3}.
\end{remark}

\vskip .05in

\begin{remark}\label{rem:17}
In the Lemma \ref{lem:5.2}, the relationship between $\Phi(A)$ and $A$ depends on $\alpha,\beta,\rho_0, d$ and $u$.
\end{remark}

\vskip .1in

Now, we establish global $L^\infty$ estimate of the solution to Equation (\ref{eq:1.1}) in the case of weakly mixing.

\vskip .1in
\begin{proposition}[Global $L^\infty$ estimate]\label{prop:5.3}
Let $\beta-d<\alpha<2, \beta\in (d,d+1), d>2$, $\rho(t,x)$ is the solution of Equation (\ref{eq:1.1}) with initial data $\rho_0(x)$. Suppose that $\rho_0(x)$ satisfies (\ref{eq:4.14})-(\ref{eq:4.16}). If $u(x)$ is weakly mixing, then there exist a positive constant $A_0=A(\alpha,\beta,\rho_0, d)$, such that for $A\geq A_0$, we have
$$
\|\rho(t,\cdot)\|_{L^\infty}\leq C_{L^\infty},\quad  t\in [0,+\infty].
$$
where $C_{L^\infty}$ is a finite positive constant.
\end{proposition}

\vskip .1in

\begin{proof}
Combining (\ref{eq:4.85}) and the definition of $\Phi(A)$ in (\ref{eq:5.8}), to obtain
\begin{equation}\label{eq:5.39}
\begin{aligned}
\frac{d}{dt}\|\rho-\overline{\rho}\|_{L^2}^2&\leq-\|\rho\|_{\dot{H}^{\frac{\alpha}{2}}}^2
+C_2\|\rho-\overline{\rho}\|^{2}_{L^2}\\
&\leq -\Phi(A)\|\rho-\overline{\rho}\|^{2}_{L^2}+C_2\|\rho-\overline{\rho}\|^{2}_{L^2}.
\end{aligned}
\end{equation}
Then we deduce by (\ref{eq:5.39}) and Gronwall's inequality that
\begin{equation}\label{eq:5.40}
\begin{aligned}
\|(\rho-\overline{\rho})(\tau_1,\cdot)\|_{L^2}^2
&\leq \|\rho_0-\overline{\rho}\|_{L^2}^2\exp\left(\int_0^{\tau_1}(-\Phi(A)+C_2)dt\right)\\
&\leq e^{C_2\tau_1}\|\rho_0-\overline{\rho}\|_{L^2}^2\exp(-\int_0^{\tau_1}\Phi(A)dt).
\end{aligned}
\end{equation}
According to Lemma \ref{lem:5.2} and (\ref{eq:5.40}), when $A$ is large enough, one has
\begin{equation}\label{eq:5.41}
\begin{aligned}
\|(\rho-\overline{\rho})(\tau_1,\cdot)\|_{L^2}^2
&\leq e^{C_2\tau_1}\|\rho_0-\overline{\rho}\|_{L^2}^2\exp\left(-\int_{\Sigma(t)}\Phi(A)dt\right)\\
&\leq e^{C_2\tau_1}\|\rho_0-\overline{\rho}\|_{L^2}^2\exp\left(-\frac{\tau_1}{2}\Phi(A)dt\right)\\
&\leq B^2_1.
\end{aligned}
\end{equation}
Combining (\ref{eq:4.17}), (\ref{eq:5.41}), Lemma \ref{lem:4.3} and interpolation inequality, to obtain
\begin{equation}\label{eq:5.42}
\|\rho(\tau_1,\cdot)-\overline{\rho}\|_{L^p}\leq \|\rho-\overline{\rho}\|^{\frac{2}{p}}_{L^2}\|\rho-\overline{\rho}\|^{1-\frac{2}{p}}_{L^\infty}
\leq D_0.
\end{equation}
Thus one has
\begin{equation}\label{eq:5.43}
\|\rho(\tau_1,\cdot)\|_{L^p}\leq D_0+\overline{\rho}.
\end{equation}
According to Lemma \ref{lem:4.3}, for any $t\in [0,\tau_1]$, one has
\begin{equation}\label{eq:5.44}
\|\rho(t,\cdot)\|_{L^p}\leq 2(D_0+\overline{\rho}).
\end{equation}
If denote
$$
\widetilde{\rho}(t)=\rho(t,\overline{x}_t)=\max_{x\in \mathbb{T}^d}\rho(t,x).
$$
Combining Lemma \ref{lem:2.3} and $t\in [0,\tau_1]$, if $\widetilde{\rho}(t)$ satisfies the (\ref{eq:2.6}), we deduce by (\ref{eq:4.16}) and (\ref{eq:5.44}) that
\begin{equation}\label{eq:5.45}
\widetilde{\rho}(t)\leq C(d,p)\|\rho(t,\cdot)\|_{L^p}\leq2C(d,p)(D_0+\overline{\rho})\leq C_\infty.
\end{equation}
If not, then $\widetilde{\rho}(t)> 2C(d,p)(D_0+\overline{\rho})$ and $\widetilde{\rho}(t)$ satisfies (\ref{eq:2.5}). According to (\ref{eq:4.3}) and ({\ref{eq:5.44}}), one get
\begin{equation}\label{eq:5.46}
\frac{d}{dt}\widetilde{\rho}\leq -C_5\widetilde{\rho}^{1+\frac{p\alpha}{d}}+C\rho(\overline{x})^{2+\frac{\beta-d}{d+\alpha-\beta}}
+C\rho(\overline{x})^2,   \ \ \ t\in [0,\tau_1],
\end{equation}
where
$$
C_5=\frac{C(\alpha,d,p)}{2^{1+\frac{p\alpha}{d}}(D_0+\overline{\rho}))^{\frac{p\alpha}{d}}}.
$$
Notice that the constants $C(d,p)$ and $C(\alpha, d,p)$ are from Lemma \ref{lem:2.3}. We set
\begin{equation}\label{eq:5.47}
M_0=\max\{x|-C_5x^{1+\frac{p\alpha}{d}}+Cx^{2+\frac{\beta-d}{d+\alpha-\beta}}
+Cx^2=0\},
\end{equation}
and  denote
\begin{equation}\label{eq:5.48}
C_{L^\infty}=\max\left\{2C(d,p)(D_0+\overline{\rho}), M_0, \|\rho_0\|_{L^\infty}\right\}.
\end{equation}
Since $p>\frac{d}{\alpha+d-\beta}$, then
\begin{equation}\label{eq:5.49}
1+\frac{p\alpha}{d}>2+\frac{\beta-d}{d+\alpha-\beta}.
\end{equation}
Solving the differential inequality of (\ref{eq:5.46}), we deduce by the definition of $C_{L^\infty}$ in (\ref{eq:5.48}) and (\ref{eq:5.49}) that
\begin{equation}\label{eq:5.50}
\widetilde{\rho}(t)\leq  C_{L^\infty},\ \ \ t\in [0,\tau_1].
\end{equation}
According to (\ref{eq:5.47}) and (\ref{eq:5.48}), we know that $C_{L^\infty}$ is  independent of $\rho(x)$. Without loss of generality, we take $C_\infty=C_{L^\infty}$. Combining (\ref{eq:5.45}) and (\ref{eq:5.50}), to obtain
\begin{equation}\label{eq:5.51}
\|\rho(t,\cdot)\|_{L^\infty}\leq C_{L^\infty}\ \ \ t\in [0,\tau_1].
\end{equation}
For the solution $\rho(t,x)$ of Equation (\ref{eq:1.1}), by the similarity with (\ref{eq:5.41}) and (\ref{eq:5.43}), we deduce by bootstrap argument that for any $n\in\mathbb{Z}^+$, one has
$$
\|\rho(n\tau_1,\cdot)-\overline{\rho}\|_{L^2}\leq B_1,\ \ \ \ \|\rho(n\tau_1,\cdot)\|_{L^p}\leq D_0+\overline{\rho}.
$$
Then by the similarity with  (\ref{eq:5.51}), for any $n\in\mathbb{Z}^+$, one get
$$
\|\rho(t,\cdot)\|_{L^\infty}\leq C_{L^\infty}\ \ \ \ t\in [0,n\tau_1].
$$
Thus for any $t\geq 0$, we have
$$
\|\rho(t,\cdot)\|_{L^\infty}\leq C_{L^\infty}.
$$
This completes the proof of Proposition \ref{prop:5.3}.
\end{proof}

\vskip .05in

Let us prove the Theorem \ref{thm:1.1} briefly.
\begin{proof}[The proof of Theorem \ref{thm:1.1}]
According to the proof of Proposition \ref{prop:3.1} and Proposition \ref{prop:5.3}, we deduce that for the solution $\rho(t,x)$ of Equation (\ref{eq:1.1}), one has
$$
\|D\rho\|_{L^\infty}< \infty,\ \ \ \|D^2\rho\|_{L^\infty}< \infty,\ \ \ \|D^3\rho\|_{L^\infty}< \infty \quad 0\leq t\leq \infty.
$$
By using standard continuation argument, we have
$$
\rho(t,x)\in C(\mathbb{R}^{+};W^{3,\infty}(\mathbb{T}^d)).
$$
This completes the proof of Theorem \ref{thm:1.1}.
\end{proof}

\begin{remark}\label{rem:18}
In fact, for any $k\geq1$, $\rho_0\in W^{k,\infty}(\mathbb{T}^d)$, we can obtain
$$
\rho(t,x)\in C(\mathbb{R}^{+}; W^{k,\infty}(\mathbb{T}^d)).
$$
\end{remark}

\vskip .2in

\section{Appendix}

\vskip .1in

\subsection{Nonlinear maximum principle}
The nonlinear maximum principle of fractional Laplancian operator has been extensively studied and applied, the details can be referred to \cite{Burczak.2017, Constantin.2012, Rafael.2016, Shi.2019}. In this section, we establish a nonlinear lower bound for the $|D^kf|, k\in \mathbb{Z^+}$ on tours $\mathbb{T}^d$. The main idea of proof is from \cite{Constantin.2012}.

\vskip .1in

\begin{lemma}\label{lem:6.1}
Let $f\in \mathcal{S}(\mathbb{T}^d)$ and denote by $\overline{x}\in \mathbb{T}^d$ the point such that
$$
|Df|^2(\overline{x})=\max_{x\in \mathbb{T}^d}|Df|^2(x).
$$
Then we have the following
\begin{equation}\label{eq:6.1}
\int_{\mathbb{T}^d}\frac{[Df(\overline{x})-Df(y)]^2}{|\overline{x}-y|^{d+\alpha}}dy
\geq c_1\frac{|Df|^{2+\alpha}(\overline{x})}{\|f\|^\alpha_{L^\infty}},
\end{equation}
or
\begin{equation}\label{eq:6.2}
|Df|(\overline{x})\leq c_2\|f\|_{L^\infty},
\end{equation}
where $D$ denotes any partial derivative, and $c_1,c_2$ are positive constant.
\end{lemma}

\begin{proof}
We denote
\begin{equation}\label{eq:6.3}
F=\int_{\mathbb{T}^d}\frac{[Df(\overline{x})-Df(y)]^2}{|\overline{x}-y|^{d+\alpha}}dy.
\end{equation}
Define a radially non-decreasing smooth cut-off function, it is as follows
\begin{equation}\label{eq:6.4}
\chi(x)=
\begin{cases}
0\ \ \ \ \ \ & |x|\leq\frac{1}{2},\\
1\ \ \ \ \ \ & |x|\geq 1,
\end{cases}
\end{equation}
and $0\leq\chi\leq1$. Let $R>0$ be a constant, to be chose later. Since for any $y\in \mathbb{T}^d$, one has
\begin{equation}\label{eq:6.5}
[Df(\overline{x})-Df(y)]^2\geq (Df)^2(\overline{x})-2Df(\overline{x})Df(y).
\end{equation}
Then if $R<1$, combining (\ref{eq:6.3}), (\ref{eq:6.4}) and (\ref{eq:6.5}), to obtain
\begin{equation}\label{eq:6.6}
\begin{aligned}
F&\geq \int_{\mathbb{T}^d}\frac{[Df(\overline{x})-Df(y)]^2}{|\overline{x}-y|^{d+\alpha}}\chi(y/R)dy\\
&\geq\int_{\mathbb{T}^d}\frac{(Df)^2(\overline{x})}{|\overline{x}-y|^{d+\alpha}}\chi(y/R)dy
-2\int_{\mathbb{T}^d}\frac{Df(\overline{x})Df(y)}{|\overline{x}-y|^{d+\alpha}}\chi(y/R)dy\\
&\geq |Df|^2(\overline{x})\int_{\mathbb{T}^d}\frac{\chi(y/R)}{|\overline{x}-y|^{d+\alpha}}dy
-|Df|(\overline{x})\left| \int_{\mathbb{T}^d}Df(y)\frac{\chi(y/R)}{|\overline{x}-y|^{d+\alpha}}dy
\right|\\
&\geq|Df|^2(\overline{x})\int_{\mathbb{T}^d \cap B_R^c}\frac{1}{|\overline{x}-y|^{d+\alpha}}dy
-|Df|(\overline{x})\|f\|_{L^\infty}\int_{\mathbb{T}^d}
\left|D\frac{\chi(y/R)}{|\overline{x}-y|^{d+\alpha}}\right|dy\\
&\geq c'|Df|^2(\overline{x})(\frac{1}{R^\alpha}-1)
-c'_2\frac{|Df|(\overline{x})\|f\|_{L^\infty}}{R^{\alpha+1}}\\
&\geq c'_1\frac{|Df|^2(\overline{x})}{R^\alpha}
-c'_2\frac{|Df|(\overline{x})\|f\|_{L^\infty}}{R^{\alpha+1}},
\end{aligned}
\end{equation}
where
$$
c'_1=c_1(d,\alpha,\chi)>0,\ \ \ \ \ \ c'_2=c_2(d,\alpha,\chi)>0.
$$
Define
\begin{equation}\label{eq:6.7}
R=\frac{2c'_2\|f\|_{L^\infty}}{c'_1|Df|(\overline{x})}.
\end{equation}
If $R<1$, we deduce by (\ref{eq:6.6}) and (\ref{eq:6.7}) that
\begin{equation}\label{eq:6.8}
\begin{aligned}
F&\geq c'_1\frac{|Df|^2(\overline{x})}{R^\alpha}
-c'_2\frac{|Df|(\overline{x})\|f\|_{L^\infty}}{R^{\alpha+1}}\\
&=c'_1\left(\frac{c'_1}{2c'_2}\right)^\alpha
\frac{(|Df|)^{2+\alpha}(\overline{x})}{\|f\|^\alpha_{L^\infty}}
-c'_2\left(\frac{c'_1}{2c'_2}\right)^{1+\alpha}
\frac{(|Df|)^{2+\alpha}(\overline{x})}{\|f\|^\alpha_{L^\infty}}\\
&=\left( \frac{1}{2^\alpha}-\frac{1}{2^{1+\alpha}}\right)\frac{(c'_1)^{1+\alpha}}{(c'_2)^{\alpha}}
\frac{(|Df|)^{2+\alpha}(\overline{x})}{\|f\|^\alpha_{L^\infty}}\\
&=c_1\frac{|Df|^{2+\alpha}(\overline{x})}{\|f\|^\alpha_{L^\infty}}.
\end{aligned}
\end{equation}
If $R\geq1$, by the definition of $R$ in (\ref{eq:6.7}), one get
\begin{equation}\label{eq:6.9}
R=\frac{2c'_2\|f\|_{L^\infty}}{c'_1|Df|(\overline{x})}\geq1,
\end{equation}
then through (\ref{eq:6.9}), we obtain
\begin{equation}\label{eq:6.10}
|Df|(\overline{x})\leq\frac{2c'_2}{c'_1}\|f\|_{L^\infty}\leq c_2\|f\|_{L^\infty}.
\end{equation}
Therefore, for any $f\in \mathcal{S}(\mathbb{T}^d)$, at least one of (\ref{eq:6.8}) and (\ref{eq:6.10}) is true. This completes the proof of Lemma \ref{lem:6.1}.
\end{proof}

\vskip .1in

\begin{remark}\label{rem:19}
The case of $\mathbb{R}^d$ has been proved, the details can be referred to \cite{Constantin.2012}.
\end{remark}

\vskip .1in

For the case of $|D^k f|,k\in \mathbb{Z}^+$, we also has the similar result with Lemma \ref{lem:6.1}, it is as follows

\vskip .1in

\begin{corollary}\label{cor:6.2}
Let $f\in \mathcal{S}(\mathbb{T}^d)$ and denote  by $\overline{x}\in \mathbb{T}^d$ the point such that
$$
|D^kf|^2(\overline{x})=\max_{x\in \mathbb{T}^d}|Df|^2(x).
$$
Then we have the following
\begin{equation}\label{eq:6.11}
\int_{\mathbb{T}^d}\frac{[D^kf(\overline{x})-D^kf(y)]^2}{|\overline{x}-y|^{d+\alpha}}dy
\geq c_1\frac{|D^kf|^{2+\alpha}(\overline{x})}{\|D^{k-1}f\|^\alpha_{L^\infty}},
\end{equation}
or
\begin{equation}\label{eq:6.12}
|D^kf|(\overline{x})\leq c_2\|D^{k-1}f\|_{L^\infty},
\end{equation}
where $k\in \mathbb{Z}^{+}$, $D$ denotes any partial derivative, and $c_1,c_2$ are positive constant.
\end{corollary}

\vskip .1in

\begin{proof}
By the similar with the proof of Lemma \ref{lem:6.1}, we can get the Corollary \ref{cor:6.2}.
\end{proof}

\vskip .2in

\subsection{The estimate of operator}

We give the proofs of Lemma \ref{lem:2.8} and Lemma \ref{lem:2.9}.

\vskip .05in

\begin{proof}[The proof of Lemma \ref{lem:2.8}]
We only claim that $\liminf\limits_{A\rightarrow +\infty}\Psi(H^\alpha_A)<+\infty$ implies that $U^t$ has a nonzero eigenfunction in $\dot{H}^{\frac{\alpha}{2}}(\mathbb{T}^d)\cap X$, where $U^t$ is defined in (\ref{eq:2.10}). In this case, there exists a sequence $\{A_m\}_{m=1}^{\infty}$ and constant $\delta_0\in \mathbb{R}^+$, such that  $\lim\limits_{m\rightarrow +\infty}A_m=+\infty$ and
\begin{equation}\label{eq:6.13}
\Psi(H^{\alpha}_{A_m})<\delta_0, \ \ \ m=1,2,\cdots.
\end{equation}
Combining (\ref{eq:2.11}), (\ref{eq:2.13}) and (\ref{eq:6.13}), we imply that for any fixed $m$, there exists $\lambda_m\in \mathbb{R}, f_m\in D(H)$, such that $\|f_m\|_{L^2}=1$ and
\begin{equation}\label{eq:6.14}
\|(H^\alpha_{A_m}-i\lambda_m)f_m\|_{L^2}< \delta_0.
\end{equation}
According to Cauchy-Schwartz inequality and (\ref{eq:6.14}), one has
\begin{equation}\label{eq:6.15}
\|\Lambda^{\frac{\alpha}{2}}f_m\|^2_{L^2}=Re \langle f_m,g_m\rangle\leq \|f_m\|_{L^2}\|g_m\|_{L^2}<\delta_0,
\end{equation}
where
\begin{equation}\label{eq:6.16}
g_m=(H^\alpha_{A_m}-i\lambda_m)f_m.
\end{equation}
Thus there exists $f_0\in \dot{H}^{\frac{\alpha}{2}}(\mathbb{T}^d)$ and a subsequence of $\{f_m\}_{m=1}^{\infty}$, still denote by $\{f_m\}_{m=1}^{\infty}$, such that
\begin{equation}\label{eq:6.17}
f_m\rightarrow f_0\ \ \ \  in\ \  L^2(\mathbb{T}^d)
\end{equation}
is strongly convergence and
$$
\|f_0\|_{L^2}=1, \ \ \ f_0\in X.
$$
Combining the definition of (\ref{eq:6.16}) and the property of fractional Laplacian, we know that for any $f\in H^{1}(\mathbb{T}^d)$, one has
\begin{equation}\label{eq:6.18}
\langle g_m,f\rangle=\langle \Lambda^{\frac{\alpha}{2}} f_m, \Lambda^{\frac{\alpha}{2}} f\rangle
+A_m\langle u\cdot\nabla f_m,f\rangle-i\lambda_m\langle f_m,f\rangle.
\end{equation}
As $f\in H^{\frac{\alpha}{2}}(\mathbb{T}^d)$ and $\lim\limits_{m\rightarrow \infty}A_m=\infty$, we deduce by (\ref{eq:6.14}), (\ref{eq:6.17}) and (\ref{eq:6.18}) that
\begin{equation}\label{eq:6.19}
\langle u\cdot\nabla f_m,f\rangle-i\frac{\lambda_m}{A_m}\langle f_m,f\rangle=\frac{\langle g_m,f\rangle-\langle \Lambda^{\frac{\alpha}{2}} f_m, \Lambda^{\frac{\alpha}{2}} f\rangle}{A_m}\rightarrow 0,\ \ \ m\rightarrow \infty,
\end{equation}
where we use the
$$
\begin{aligned}
|\langle g_m,f\rangle-\langle \Lambda^{\frac{\alpha}{2}} f_m, \Lambda^{\frac{\alpha}{2}} f\rangle|&
\leq \|g_m\|_{L^2}\|f\|_{L^2}+\|\Lambda^{\frac{\alpha}{2}} f_m\|_{L^2}\|\Lambda^{\frac{\alpha}{2}} f\|_{L^2}\\
&\leq \delta_0\|f\|_{L^2}+\delta_0^{\frac{1}{2}}\|\Lambda^{\frac{\alpha}{2}} f\|_{L^2}.
\end{aligned}
$$
As $u$ is divergence-free and (\ref{eq:6.17}), one has
\begin{equation}\label{eq:6.20}
\langle u\cdot \nabla f_m,f\rangle=-\langle f_m, u\cdot \nabla f\rangle\rightarrow -\langle f_0, u\cdot \nabla f\rangle
=\langle u\cdot \nabla f_0,  f\rangle,\ \ \ m\rightarrow \infty,
\end{equation}
and
$$
\langle f_m,f\rangle\rightarrow\langle f_0,f\rangle,\ \ \ m\rightarrow \infty.
$$
Notices that the (\ref{eq:6.20}) indicates that $u\cdot \nabla f_0$ is well define in the sense of inner product. Combining (\ref{eq:6.19}) and (\ref{eq:6.20}), to obtain
\begin{equation}\label{eq:6.21}
\lim_{m\rightarrow +\infty}i\frac{\lambda_m}{A_m}\langle f_m,f\rangle=\langle  u\cdot \nabla f_0,f\rangle.
\end{equation}
Since $\dot{H}^{\frac{\alpha}{2}}(\mathbb{T}^d)$ is dense in $\dot{H}^{1}(\mathbb{T}^d)$, then there exist a $f_1\in \dot{H}^{1}(\mathbb{T}^d)$, such that
$$
|\langle f_0-f_1,f_0\rangle|=|\langle f_0,f_0\rangle-\langle f_0,f_1\rangle|\leq\frac{1}{2},
$$
without loss of generality, we assume that $\langle f_0,f_1\rangle=\frac{1}{2}$. If we take $f=f_1$,  then one has
$$
\langle f_m,f_1\rangle \rightarrow \langle f_0, f_1\rangle=\langle f_0,f_1\rangle=\frac{1}{2}\neq 0,\ \ \ m\rightarrow \infty.
$$
Therefore, we deduce by (\ref{eq:6.21}) that
$$
\frac{1}{2}\times i\frac{\lambda_m}{A_m}\rightarrow \langle u\cdot \nabla f_0,f_1\rangle\doteq i\lambda.
$$
Then for every $f\in H^{1}(\mathbb{T}^d)$, we get
$$
i2\lambda\langle f_0,f\rangle=\lim_{m\rightarrow +\infty}i\frac{\lambda_m}{A_m}\langle f_m,f\rangle=\langle  u\cdot \nabla f_0,f\rangle.
$$
If $\lambda=0$, one can know that $f_0$ is a constant, this is contradictory to $f_0\in X$. Then
$$
 u\cdot \nabla f_0=i2\lambda f_0,\ \ \ f_0\neq0.
$$
Thus $f_0$ is a nonzero eigenfunction of operator $U^t$ in $\dot{H}^{\frac{\alpha}{2}}(\mathbb{T}^d)\cap X$,  this contradicts the definition that $u$ is weakly mixing. So we have
$$
\lim_{A\rightarrow +\infty}\Psi(H^\alpha_A)=+\infty.
$$
This completes the proof of Lemma \ref{lem:2.8}.
\end{proof}

\vskip .1in

Next, we give the proof of Lemma \ref{lem:2.9} by Lemma \ref{lem:2.8}.

\vskip .1in

\begin{proof}[The proof of Lemma \ref{lem:2.9}]
According to  the definition of $U^t$ in (\ref{eq:2.10}) and $H^\alpha_A$ in (\ref{eq:2.13}),  the semigroup $e^{-tH^\alpha_A}$ can written as
\begin{equation}\label{eq:6.22}
e^{-tH^\alpha_A}=e^{-(-\Delta)^{\frac{\alpha}{2}}t}U^{At}.
\end{equation}
Combining the Fourier series and (\ref{eq:6.22}), for $f\in D(H^\alpha_A)$, one has
\begin{equation}\label{eq:6.23}
e^{-tH^\alpha_A}f(x)=\sum_{k\in \mathbb{Z}^d}e^{-|k|^\alpha t}\widehat{U^{At}f}(k)e^{ikx}, \ \ \ x\in \mathbb{T}^d.
\end{equation}
The operator $P_N$ is applied to (\ref{eq:6.23}), then one get
\begin{equation}\label{eq:6.24}
P_Ne^{-tH^\alpha_A}f(x)=\sum_{|k|\leq N}e^{-|k|^\alpha t}\widehat{U^{At}f}(k)e^{ikx}.
\end{equation}
We deduce by Parseval's identity that
\begin{equation}\label{eq:6.25}
\|P_Ne^{-tH^\alpha_A}f\|^2_{L^2}=\sum_{|k|\leq N}\left|e^{-|k|^\alpha t}\widehat{U^{At}f}(k)\right|^2.
\end{equation}
We claim that for any $t\geq 0$ and $A$ is large enough, one has
\begin{equation}\label{eq:6.26}
\|P_Ne^{-tH^\alpha_A}f\|^2_{L^2}\leq C\|P_Nf\|^2_{L^2},
\end{equation}
where $C>1$ is a fixed constant. If $t=0$, we can easily get
\begin{equation}\label{eq:6.27}
\|P_Ne^{-tH^\alpha_A}f\|^2_{L^2}=\|P_Nf\|^2_{L^2}.
\end{equation}
For $t>0$ and $A$ is large enough, if
\begin{equation}\label{eq:6.28}
\|P_Ne^{-tH^\alpha_A}f\|^2_{L^2}>C\|P_Nf\|^2_{L^2},
\end{equation}
then the $e^{-tH^\alpha_A}f(x)$  has the process of energy moving to low frequency. In fact, according to the definition $\Psi(H^\alpha_A)$, (\ref{eq:2.15}) and the Lemma \ref{lem:2.8}, we deduce that the $e^{-tH^\alpha_A}$  has only the process of energy moving to high frequency if $A$ is large enough.
Thus the (\ref{eq:6.28}) is not true. Then for $t>0$ and $A$ is large enough, one has
\begin{equation}\label{eq:6.29}
\|P_Ne^{-tH^\alpha_A}f\|_{L^2}\leq C\|P_Nf\|_{L^2}.
\end{equation}
Combining (\ref{eq:6.27}) and (\ref{eq:6.29}), we show that for any $t\geq 0$, if $A$ is large enough,  we have
\begin{equation}\label{eq:6.30}
\|P_Ne^{-tH^\alpha_A}f\|_{L^2}\leq C\|P_Nf\|_{L^2}.
\end{equation}
This completes the proof of Lemma \ref{lem:2.9}.
\end{proof}

\vskip .1in

\noindent \textbf{Acknowledgement}

\noindent The author would like to thank  Professor Weike Wang for helpful discussion. The research are supported by the National Natural Science Foundation of China 11771284.

\vskip .3in


\begin{thebibliography}{10}

\bibitem{Ascasibar.2013}
Y.~Ascasibar, R.~Granero-Belinch\'{o}n, and J.~M. Moreno.
\newblock An approximate treatment of gravitational collapse.
\newblock {\em Phys. D}, 262:71--82, 2013.

\bibitem{Bedrossian.201701}
J.~Bedrossian and M.~Coti~Zelati.
\newblock Enhanced dissipation, hypoellipticity, and anomalous small noise
  inviscid limits in shear flows.
\newblock {\em Arch. Ration. Mech. Anal.}, 224(3):1161--1204, 2017.

\bibitem{Bedrossian.2017}
J.~Bedrossian and S.~He.
\newblock Suppression of blow-up in {P}atlak-{K}eller-{S}egel via shear flows.
\newblock {\em SIAM J. Math. Anal.}, 49(6):4722--4766, 2017.

\bibitem{Bedrossian.201602}
J.~Bedrossian, N.~Masmoudi, and V.~Vicol.
\newblock Enhanced dissipation and inviscid damping in the inviscid limit of
  the {N}avier-{S}tokes equations near the two dimensional {C}ouette flow.
\newblock {\em Arch. Ration. Mech. Anal.}, 219(3):1087--1159, 2016.

\bibitem{Biler.2010}
P.~Biler and G.~Karch.
\newblock Blowup of solutions to generalized {K}eller-{S}egel model.
\newblock {\em J. Evol. Equ.}, 10(2):247--262, 2010.

\bibitem{Biler.1999}
P.~Biler and W.~A. Woyczy\'{n}ski.
\newblock Global and exploding solutions for nonlocal quadratic evolution
  problems.
\newblock {\em SIAM J. Appl. Math.}, 59(3):845--869, 1999.

\bibitem{Blanchet.2006}
A.~Blanchet, J.~Dolbeault, and B.~Perthame.
\newblock Two-dimensional {K}eller-{S}egel model: optimal critical mass and
  qualitative properties of the solutions.
\newblock {\em Electron. J. Differential Equations}, pages No. 44, 32, 2006.

\bibitem{Nikolaos.2010}
N.~Bournaveas and V.~Calvez.
\newblock The one-dimensional {K}eller-{S}egel model with fractional diffusion
  of cells.
\newblock {\em Nonlinearity}, 23(4):923--935, 2010.

\bibitem{Brenner.1999}
M.~P. Brenner, P.~Constantin, L.~P. Kadanoff, A.~Schenkel, and S.~C.
  Venkataramani.
\newblock Diffusion, attraction and collapse.
\newblock {\em Nonlinearity}, 12(4):1071--1098, 1999.

\bibitem{Burczak.2017}
J.~Burczak and R.~Granero-Belinch\'{o}n.
\newblock Suppression of blow up by a logistic source in 2{D} {K}eller-{S}egel
  system with fractional dissipation.
\newblock {\em J. Differential Equations}, 263(9):6115--6142, 2017.

\bibitem{Calderon.1954}
A.~P. Calder\'{o}n and A.~Zygmund.
\newblock Singular integrals and periodic functions.
\newblock {\em Studia Math.}, 14:249--271 (1955), 1954.

\bibitem{Che.2016}
J.~Che, L.~Chen, B.~Duan, and Z.~Luo.
\newblock On the existence of local strong solutions to chemotaxis--shallow
  water system with large data and vacuum.
\newblock {\em J. Differential Equations}, 261(12):6758--6789, 2016.

\bibitem{Vicol.2014}
P.~Constantin, N.~Glatt-Holtz, and V.~Vicol.
\newblock Unique ergodicity for fractionally dissipated, stochastically forced
  2{D} {E}uler equations.
\newblock {\em Comm. Math. Phys.}, 330(2):819--857, 2014.

\bibitem{Constantin.2008}
P.~Constantin, A.~Kiselev, L.~Ryzhik, and A.~Zlato\v{s}.
\newblock Diffusion and mixing in fluid flow.
\newblock {\em Ann. of Math. (2)}, 168(2):643--674, 2008.

\bibitem{Constantin.2012}
P.~Constantin and V.~Vicol.
\newblock Nonlinear maximum principles for dissipative linear nonlocal
  operators and applications.
\newblock {\em Geom. Funct. Anal.}, 22(5):1289--1321, 2012.

\bibitem{Constantin.1999}
P.~Constantin and J.~Wu.
\newblock Behavior of solutions of 2{D} quasi-geostrophic equations.
\newblock {\em SIAM J. Math. Anal.}, 30(5):937--948, 1999.

\bibitem{Cordoba.2004}
A.~C\'ordoba and D.~C\'ordoba.
\newblock A maximum principle applied to quasi-geostrophic equations.
\newblock {\em Comm. Math. Phys.}, 249(3):511--528, 2004.

\bibitem{Corrias.2004}
L.~Corrias, B.~Perthame, and H.~Zaag.
\newblock Global solutions of some chemotaxis and angiogenesis systems in high
  space dimensions.
\newblock {\em Milan J. Math.}, 72:1--28, 2004.

\bibitem{Cusimano.2018}
N.~Cusimano, F.~del Teso, L.~Gerardo-Giorda, and G.~Pagnini.
\newblock Discretizations of the spectral fractional {L}aplacian on general
  domains with {D}irichlet, {N}eumann, and {R}obin boundary conditions.
\newblock {\em SIAM J. Numer. Anal.}, 56(3):1243--1272, 2018.

\bibitem{Cycon.1987}
H.~L. Cycon, R.~G. Froese, W.~Kirsch, and B.~Simon.
\newblock {\em Schr\"{o}dinger operators with application to quantum mechanics
  and global geometry}.
\newblock Texts and Monographs in Physics. Springer-Verlag, Berlin, study
  edition, 1987.

\bibitem{Francesco.2010}
M.~Di~Francesco, A.~Lorz, and P.~Markowich.
\newblock Chemotaxis-fluid coupled model for swimming bacteria with nonlinear
  diffusion: global existence and asymptotic behavior.
\newblock {\em Discrete Contin. Dyn. Syst.}, 28(4):1437--1453, 2010.

\bibitem{Duan.2010}
R.~Duan, A.~Lorz, and P.~Markowich.
\newblock Global solutions to the coupled chemotaxis-fluid equations.
\newblock {\em Comm. Partial Differential Equations}, 35(9):1635--1673, 2010.

\bibitem{Evans.2010}
L.~C. Evans.
\newblock {\em Partial differential equations}, volume~19 of {\em Graduate
  Studies in Mathematics}.
\newblock American Mathematical Society, Providence, RI, second edition, 2010.

\bibitem{Fayad.2006}
B.~Fayad.
\newblock Smooth mixing flows with purely singular spectra.
\newblock {\em Duke Math. J.}, 132(2):371--391, 2006.

\bibitem{Fayad.2002}
B.~R. Fayad.
\newblock Weak mixing for reparameterized linear flows on the torus.
\newblock {\em Ergodic Theory Dynam. Systems}, 22(1):187--201, 2002.

\bibitem{Friedman.1969}
A.~Friedman.
\newblock {\em Partial differential equations}.
\newblock Holt, Rinehart and Winston, Inc., New York-Montreal, Que.-London,
  1969.

\bibitem{Loukas.2014}
L.~Grafakos and S.~Oh.
\newblock The {K}ato-{P}once inequality.
\newblock {\em Comm. Partial Differential Equations}, 39(6):1128--1157, 2014.

\bibitem{Rafael.2016}
R.~Granero-Belinch\'{o}n.
\newblock On a drift-diffusion system for semiconductor devices.
\newblock {\em Ann. Henri Poincar\'{e}}, 17(12):3473--3498, 2016.

\bibitem{He.201802}
S.~He.
\newblock Suppression of blow-up in parabolic-parabolic
  {P}atlak-{K}eller-{S}egel via strictly monotone shear flows.
\newblock {\em Nonlinearity}, 31(8):3651--3688, 2018.

\bibitem{Hillen.2004}
T.~Hillen and A.~Potapov.
\newblock The one-dimensional chemotaxis model: global existence and asymptotic
  profile.
\newblock {\em Math. Methods Appl. Sci.}, 27(15):1783--1801, 2004.

\bibitem{Hopf.2018}
K.~Hopf and J.~L. Rodrigo.
\newblock Aggregation equations with fractional diffusion: preventing
  concentration by mixing.
\newblock {\em Commun. Math. Sci.}, 16(2):333--361, 2018.

\bibitem{Ju.2005}
N.~Ju.
\newblock The maximum principle and the global attractor for the dissipative
  2{D} quasi-geostrophic equations.
\newblock {\em Comm. Math. Phys.}, 255(1):161--181, 2005.

\bibitem{Kato.1995}
T.~Kato.
\newblock {\em Perturbation theory for linear operators}.
\newblock Classics in Mathematics. Springer-Verlag, Berlin, 1995.
\newblock Reprint of the 1980 edition.

\bibitem{Alexander.2008}
A.~Kiselev, F.~Nazarov, and R.~Shterenberg.
\newblock Blow up and regularity for fractal {B}urgers equation.
\newblock {\em Dyn. Partial Differ. Equ.}, 5(3):211--240, 2008.

\bibitem{Kiselev.200801}
A.~Kiselev, R.~Shterenberg, and A.~Zlato\v{s}.
\newblock Relaxation enhancement by time-periodic flows.
\newblock {\em Indiana Univ. Math. J.}, 57(5):2137--2152, 2008.

\bibitem{Kiselev.2016}
A.~Kiselev and X.~Xu.
\newblock Suppression of chemotactic explosion by mixing.
\newblock {\em Arch. Ration. Mech. Anal.}, 222(2):1077--1112, 2016.

\bibitem{Li.2010}
D.~Li, J.~L. Rodrigo, and X.~Zhang.
\newblock Exploding solutions for a nonlocal quadratic evolution problem.
\newblock {\em Rev. Mat. Iberoam.}, 26(1):295--332, 2010.

\bibitem{Li.2018}
L.~Li, J.-G. Liu, and L.~Wang.
\newblock Cauchy problems for {K}eller-{S}egel type time-space fractional
  diffusion equation.
\newblock {\em J. Differential Equations}, 265(3):1044--1096, 2018.

\bibitem{Lin.2019}
Z.~Lin and M.~Xu.
\newblock Metastability of {K}olmogorov flows and inviscid damping of shear
  flows.
\newblock {\em Arch. Ration. Mech. Anal.}, 231(3):1811--1852, 2019.

\bibitem{Liu.2011}
J.-G. Liu and A.~Lorz.
\newblock A coupled chemotaxis-fluid model: global existence.
\newblock {\em Ann. Inst. H. Poincar\'{e} Anal. Non Lin\'{e}aire},
  28(5):643--652, 2011.

\bibitem{Lorz.2010}
A.~Lorz.
\newblock Coupled chemotaxis fluid model.
\newblock {\em Math. Models Methods Appl. Sci.}, 20(6):987--1004, 2010.

\bibitem{Masmoudi.2020}
N.~Masmoudi and W.~Zhao.
\newblock Enhanced dissipation for the 2d couette flow in critical space.
\newblock {\em https://doi.org/10.1080/03605302.2020.1791180}.

\bibitem{Michele.2019}
K.~W. Michele Coti~Zelati, Tarek M.~Elgindi.
\newblock Enhanced dissipation in the navier-stokes equations near the
  poiseuille flow.
\newblock {\em Submitted arXiv:1901.01571}.

\bibitem{Nagai.1995}
T.~Nagai.
\newblock Blow-up of radially symmetric solutions to a chemotaxis system.
\newblock {\em Adv. Math. Sci. Appl.}, 5(2):581--601, 1995.

\bibitem{Koichi.2001}
K.~Osaki and A.~Yagi.
\newblock Finite dimensional attractor for one-dimensional {K}eller-{S}egel
  equations.
\newblock {\em Funkcial. Ekvac.}, 44(3):441--469, 2001.

\bibitem{Resnick.1995}
S.~G. Resnick.
\newblock {\em Dynamical problems in non-linear advective partial differential
  equations}.
\newblock ProQuest LLC, Ann Arbor, MI, 1995.
\newblock Thesis (Ph.D.)--The University of Chicago.

\bibitem{Senba.2002}
T.~Senba and T.~Suzuki.
\newblock Weak solutions to a parabolic-elliptic system of chemotaxis.
\newblock {\em J. Funct. Anal.}, 191(1):17--51, 2002.

\bibitem{Shi.2019}
B.~Shi and W.~Wang.
\newblock Suppression of blow up by mixing in generalized {K}eller-{S}egel
  system with fractional dissipation.
\newblock {\em Commun. Math. Sci.}, 18(5):1413--1440, 2020.

\bibitem{Wang.2019}
W.~Wang and Y.~Wang.
\newblock The {$L^p$} decay estimates for the chemotaxis-shallow water system.
\newblock {\em J. Math. Anal. Appl.}, 474(1):640--665, 2019.

\bibitem{Wei.2019}
D.~Wei.
\newblock Diffusion and mixing in flud flow via resolvent estimate.
\newblock {\em Submitted arXiv:1811.11904}.

\bibitem{Wei.201901}
D.~Wei and Z.~Zhang.
\newblock Enhanced dissipation for the {K}olmogorov flow via the hypocoercivity
  method.
\newblock {\em Sci. China Math.}, 62(6):1219--1232, 2019.

\bibitem{Wei.2020}
D.~Wei, Z.~Zhang, and W.~Zhao.
\newblock Linear inviscid damping and enhanced dissipation for the {K}olmogorov
  flow.
\newblock {\em Adv. Math.}, 362:106963, 103, 2020.

\bibitem{Winkler.201201}
M.~Winkler.
\newblock Global large-data solutions in a chemotaxis-({N}avier-){S}tokes
  system modeling cellular swimming in fluid drops.
\newblock {\em Comm. Partial Differential Equations}, 37(2):319--351, 2012.

\end{thebibliography}

\end{document}